\documentclass[11pt, a4paper, abstract=yes]{scrartcl}
\usepackage[utf8]{inputenc}
\usepackage[T1]{fontenc}
\usepackage{graphicx}
\usepackage{longtable}
\usepackage{wrapfig}
\usepackage{rotating}
\usepackage[normalem]{ulem}
\usepackage{amsmath}
\usepackage{amssymb}
\usepackage{hyperref}
\usepackage{amsthm,amssymb,amsmath}
\usepackage{hyperref}
\usepackage{graphicx}
\usepackage[T1]{fontenc}\usepackage{lmodern}
\theoremstyle{definition}
\newtheorem{definition}{Definition}
\newtheorem{mexample}[definition]{Example}
\theoremstyle{remark}
\newtheorem{remark}[definition]{Remark}
\theoremstyle{plain}
\newtheorem{lemma}[definition]{Lemma}

\newtheorem{proposition}[definition]{Proposition}
\newtheorem{theorem}[definition]{Theorem}
\newtheorem{corollary}[definition]{Corollary}

\newcommand{\vek}[1]{\boldsymbol{#1}}

\author{Jan Snellman\thanks{Department of Mathematics, Linköping University, 581 83 Linköping, Sweden. \texttt{jan.snellman@liu.se}. ORCID: \href{https://orcid.org/0009-0002-6676-5068}{0009-0002-6676-5068}}}
\usepackage[T1]{fontenc}
\usepackage{orcidlink}
\usepackage{booktabs}
\usepackage{tikz}
\usepackage{etoolbox}
\AtEndEnvironment{mexample}{\hfill$\Diamond$}
\AtEndEnvironment{remark}{\hfill$\triangleright$}
\AtEndEnvironment{definition}{\hfill$\triangleleft$}
\hypersetup{hidelinks}
\newcommand{\Koz}{\operatorname{Koz}}
\newcommand{\shift}{\operatorname{shift}}

\newcommand{\conv}{\operatorname{conv}}
\newcommand{\indeg}{\operatorname{indeg}}
\newcommand{\tp}{\operatorname{top}}
\newcommand{\HF}{\operatorname{HF}}
\newcommand{\Ext}{\mathord{\wedge }V}
\author{Jan Snellman \orcidlink{0009-0002-6676-5068}}
\date{October 10, 2026}
\title{Hilbert functions over the exterior algebra and \(f\)-vectors in higher rank\\\medskip
\large II: Generalizing Linusson's count to sumsets of f-vectors}
\hypersetup{
 pdfauthor={Jan Snellman \orcidlink{0009-0002-6676-5068}},
 pdftitle={Hilbert functions over the exterior algebra and \(f\)-vectors in higher rank},
 pdfkeywords={},
 pdfsubject={},
 pdfcreator={Emacs 30.2 (Org mode 9.7.11)}, 
 pdflang={English}}
\let\autocite\cite

\begin{document}

\maketitle
\begin{abstract}
Let \(V\) be an \(n\)-dimensional vector space over a field \(k\) and \(\Ext\) its exterior algebra,
let \(F=\bigoplus_{i=1}^{r}(\Ext)g_i\) be a graded free \(\Ext\)-module with \(\deg g_i=d_i\), and let \(M\)
range over the monomial submodules of Amata and Crupi. The Hilbert function of \(F/M\) is a shifted
sum of \(f\)-vectors of simplicial complexes, and the companion paper \autocite{SnellmanArticle1} determines the convex hull of the
set of such sums. Here we count them. For \(r=1\) the count is Linusson's (1999), obtained from a
recursion through compressed complexes; we show that \textbf{no new recursion is needed for general
\(r\)}. We introduce two statistics of an \(f\)-vector, its top nonvanishing degree and the initial
degree of the corresponding ideal, and prove that the number of achievable Hilbert functions is a
product of \(r-1\) transfer matrices, one per consecutive generator-degree gap, whose entries are
differences of Linusson's own refined counts \(E^p(n,k)\). The rank therefore enters only as the
length of a matrix product. The bridge is a reformulation of the counting problem, proved from
Amata and Crupi's lex submodules, that replaces a sumset cardinality by a count of tuples subject to
a nearest-neighbour constraint. We give the same theorem for the \emph{proper} variant that the companion
paper's polytope is built from, prove the lemma that makes the restriction legitimate (lexification
preserves properness), and derive consequences: a linear recurrence in \(r\) for
constant gaps together with the exact exponential asymptotics \(K\lambda^{\,r}\) when the gap is
positive, polynomiality in \(r\) when all gaps vanish, an exact threshold beyond which the count
degenerates to a plain power, and polynomiality in the number of vertices.
\end{abstract}

\medskip\noindent\textbf{Keywords.} exterior algebra; Hilbert function; \(f\)-vector; simplicial
complex; enumeration; transfer matrix; lex submodule; Kruskal--Katona theorem.

\smallskip\noindent\textbf{2020 Mathematics Subject Classification.} Primary 05A15, 13F55;
Secondary 05E40, 13D40, 06A07.
\medskip

\setcounter{tocdepth}{2}
\tableofcontents
\section{Introduction}
\label{sec-intro}
\subsection{The rank-one count}
\label{sec:orgeb0ff88}

Let \(k\) be a field, let \(V=k^n\), and let \(\Ext\) be the exterior algebra on \(V\), graded with
degree-\(j\) component the exterior power \(\wedge^jV\) of dimension \(\binom nj\). The dimension
\(n=\dim V\) is suppressed from the notation \(\Ext\). (We write \(\Ext\) rather than
\(E\) because the letter \(E\) is spent below on Linusson's counting function \(E^p(n,k)\);
Section \ref{sec-notation} says more, and the companion article keeps \(E\).) Every monomial of \(\Ext\) is
squarefree, so every monomial ideal of \(\Ext\) is generated by the non-faces of an order family
\(\Delta\subseteq2^{[n]}\),
\begin{equation}\label{eq:exterior-face-ideal}
  I_\Delta\;=\;\bigl(e_\sigma\;:\;\sigma\notin\Delta\bigr)\;\subseteq\;\Ext ,
\end{equation}
and \(\dim_k\bigl((\Ext)/I_\Delta\bigr)_j=f_j(\Delta)\) exactly, with no correction term. The construction is the
one Stanley \autocite{Stanley1996} carries out in a polynomial ring, but \(I_\Delta\) here is \textbf{not} the
Stanley--Reisner ideal, because it lives in \(\Ext\) and not in \(k[x_1,\dots,x_n]\); it is the
\textbf{exterior face ideal}, and \((\Ext)/I_\Delta\) is the \textbf{exterior face ring}, also called Sköldberg's
\textbf{indicator algebra} \autocite{Skoldberg1999}. Unlike the ordinary Stanley--Reisner ring it is
finite-dimensional automatically, since \(\Ext\) is. The companion paper keeps the name
Stanley--Reisner ideal for \eqref{eq:exterior-face-ideal} while defining it inside \(\Ext\); the object
is the same one, and the two papers agree on the symbol \(I_\Delta\). Since passing to
an initial ideal preserves Hilbert functions, the set of Hilbert functions of graded quotients of
\(\Ext\) \emph{is} the set of \(f\)-vectors of simplicial complexes on at most \(n\) vertices.

Three classical results describe that set. Kruskal--Katona \autocite{Katona1968,Kruskal1963} determines it exactly; Kozlov \autocite{Kozlov1997}
computes its convex hull; and Linusson \autocite{Linusson1999} \textbf{counts} it. Linusson's method is a recursion on
compressed complexes: peel off the last vertex, split the complex into the part avoiding it and the
link of it, and record how far the first part is filled. The bookkeeping variable of that recursion
is his \(E^p(n,k)\), the number of \(f\)-vectors that are maximal up to level \(k\) and no
further, and that is the object this paper turns out to need.
\subsection{The rank-\(r\) question}
\label{sec:org8905989}

Let \(F=\bigoplus_{i=1}^r(\Ext)g_i\) with \(\deg g_i=d_i\in\mathbb Z\), \(d_1\le\cdots\le d_r\). Amata and
Crupi's monomial submodules are those generated by monomials of \(F\); such an \(M\) is automatically
a direct sum \(I_1g_1\oplus\cdots\oplus I_rg_r\) of monomial ideals of \(\Ext\), and the Hilbert function
of the quotient splits,
\begin{equation}\label{eq:auto-1}
  \HF\Bigl(\frac FM\Bigr)(j)\;=\;\sum_{i=1}^{r}f_{j-d_i}(\Delta_i),
\end{equation}
a sum of \(r\) independently chosen \(f\)-vectors, each shifted by its own generator degree. Writing
\(\mathcal F_n\) for the set of \(f\)-vectors of order ideals of \(\mathcal P([n])\) and
\(\shift^d\) for the degree-shift, the achievable set is the sumset
\begin{equation}\label{eq:auto-2}
  \mathrm{FF}(n;\vek d)\;=\;\Bigl\lbrace\sum_{i=1}^{r}\shift^{d_i}(\vek s_i)\;:\;\vek s_i\in\mathcal F_n\Bigr\rbrace .
\end{equation}
The companion paper \autocite{SnellmanArticle1} determines \(\conv(\mathrm{FF}(n;\vek d))\): it is a Minkowski sum of \(r\) shifted copies
of the Kozlov simplex. \textbf{This paper computes \(|\mathrm{FF}(n;\vek d)|\).}

The problem is a sumset cardinality, and those are usually hard. Two facts show where the difficulty
sits. Write
\begin{equation}\label{eq:sumset-map}
\begin{split}
  \Sigma_{\vek d}\colon\mathcal F_n^{\,r} &\longrightarrow \mathrm{FF}(n;\vek d),\\
  \Sigma_{\vek d}(\vek s_1,\dots,\vek s_r) &= \sum_{i=1}^{r}\shift^{d_i}(\vek s_i)
\end{split}
\end{equation}
for the summation map, which is surjective by the definition \eqref{eq:auto-2} of the sumset. It is
\textbf{injective} exactly when the generator degrees are far enough apart: precisely, when every gap
\(d_{i+1}-d_i\) is at least \(n+1\). In that regime \(|\mathrm{FF}(n;\vek d)|=|\mathcal F_n|^r\) and
there is nothing to do. Everything of interest happens when the shifts overlap and distinct tuples
collide, and the whole of this paper is the search for a set of representatives that
\(\Sigma_{\vek d}\) maps bijectively.
\subsection{The result}
\label{sec:org72502b7}

The answer is a transfer-matrix product, and it is built from Linusson's own refinement rather than
from a new recursion. Attach to \(\vek s\in\mathcal F_n\) two statistics: \(\tp(\vek s)\), the greatest \(j\)
with \(s(j)>0\); and \(\indeg(\vek s)\), the least \(j\) with \(s(j)<\binom nj\), the initial degree
of the corresponding monomial ideal. Let \(A_n[\tau][\rho]\) count the \(\vek s\in\mathcal F_n\) with
\(\tp(\vek s)=\tau\) and \(\indeg(\vek s)=\rho\).

Three objects are built from \(A_n\), and all three are defined again, with their index ranges spelled
out, in Definition \ref{def-uT}. The \textbf{transfer matrix} \(T_n(\delta)\), one for each integer
\(\delta\ge0\), has entries
\begin{equation}\label{eq:intro-T}
  T_n(\delta)[\tau][\tau']=\sum_{\rho>\tau-\delta}A_n[\tau'][\rho] ,
\end{equation}
so its rows and its columns are both indexed by values of \(\tp\): the row index is the \(\tp\) of
the previous summand and the column index the \(\tp\) of the current one. The \textbf{row vector}
\(\vek u_n\) is the \(\tp\) marginal of \(A_n\),
\begin{equation}\label{eq:intro-u}
  u_n[\tau]=\sum_{\rho}A_n[\tau][\rho] ,
\end{equation}
counting the \(f\)-vectors with \(\tp=\tau\) and no condition at all on \(\indeg\); it enters because
the first summand of a tuple is the one nothing constrains. And \(\vek 1_n\) is the all-ones column
vector, whose only role is to sum the entries of what stands to its left. Finally the \textbf{gaps}: there
are \(r-1\) of them, one for each consecutive pair of generator degrees, so the index \(i\) runs from
\(2\) to \(r\) and there is no \(\delta_1\).

With those in hand, the answer (Theorem \ref{thm-main}) is
\begin{equation}\label{eq:auto-3}
  \lvert\mathrm{FF}(n;\vek d)\rvert\;=\;\vek u_n\cdot T_n(\delta_2)\,T_n(\delta_3)\cdots T_n(\delta_r)\cdot\vek 1_n,
  \qquad \delta_i=d_i-d_{i-1} .
\end{equation}
The product has \(r-1\) factors, and is empty when \(r=1\).

Lemma \ref{lem-joint} identifies \(A_n\) as a difference of Linusson's numbers,
\begin{equation}\label{eq:auto-4}
  A_n[\tau][\rho]\;=\;E^n(\tau,\rho-1)-E^n(\tau-1,\rho-1),
\end{equation}
so one run of his recursion at \(p=n\) supplies everything. The rank \(r\) enters only as the length
of a matrix product.

The bridge is Theorem \ref{thm-chaining}, which replaces the sumset by a count of tuples subject to a
constraint relating consecutive entries only:
\begin{equation}\label{eq:auto-5}
  \lvert\mathrm{FF}(n;\vek d)\rvert=\#\bigl\lbrace(\vek s_1,\dots,\vek s_r)\in\mathcal F_n^{\,r}:
   \tp(\vek s_{i-1})<\indeg(\vek s_i)+\delta_i\ \text{ for } i=2,\dots,r\bigr\rbrace .
\end{equation}
That is proved from Amata and Crupi's theory of lex submodules \autocite{AmataCrupi2020KK,AmataCrupi2020}, which supplies a canonical
representative for each achievable Hilbert function; the constraint is the chaining condition of
\autocite[Definition 2.7]{AmataCrupi2020KK}, read off the \(f\)-vectors. Once the constraint is nearest-neighbour, the transfer-matrix
form is a rearrangement.
\subsection{What is proved, and what is only observed}
\label{sec:orgb73b7d5}

The distinction is kept explicit throughout.

\begin{center}
\small
\begin{tabular}{ll}
result & status\\
\hline
Thm. \ref{thm-linusson}, Linusson's recursion & cited; transcription verified against his tables\\
Lem. \ref{lem-joint}, the joint distribution & proved; verified entrywise for \(n\le6\)\\
Thm. \ref{thm-chaining}, the chaining reformulation & proved\\
Thm. \ref{thm-main}, the counting theorem & proved; verified, 41 shift vectors\\
Lem. \ref{lem-lexproper}, lexification preserves properness & proved\\
Thm. \ref{thm-proper}, the proper count & proved; reproduces recorded sequences\\
Lem. \ref{lem-transfer-support}, the support of \(T_n(\delta)\) & proved; verified, \(n\le6\)\\
Prop. \ref{prop-recurrence}, recurrence in \(r\) & proved\\
Prop. \ref{prop-growth-constant}, the constant \(K\) for \(\delta\ge1\) & proved; hypotheses verified, \(n\le6\)\\
Prop. \ref{prop-delta0}, degree bound at \(\delta=0\) & proved\\
Thm. \ref{thm-delta0-degrees}, the \textbf{exact} degrees at \(\delta=0\) & proved; verified, \(n\le8\)\\
Prop. \ref{prop-perron-threshold}, the Perron root above threshold & proved; verified, \(n\le6\)\\
Prop. \ref{prop-normalform}, the normal form \(T_n(\delta)=P_n^{\delta}B_n\) & proved; verified, \(n\le8\)\\
Thm. \ref{thm-poly-p}, polynomiality in \(p\) & proved modulo Linusson's Thm. 3.3\\
Thm. \ref{thm-asymptotics-p}, degree and leading coefficient in \(p\) & proved modulo the same; verified, \(n\le5\)\\
Thm. \ref{thm-asymptotics-p-proper}, the proper analogue in \(p\) & proved modulo the same; verified, \(n\le5\)\\
Prop. \ref{prop-fnn2}, \(\lvert S_n\rvert=F^n(n-2)\) & proved, by Alexander duality\\
Cor. \ref{cor-persymmetry}, \(A_n\) is persymmetric & proved\\
\end{tabular}
\captionof{table}{\label{tab-what-is-proved}What is proved and what is only observed, statement by statement.}

\end{center}

Two words in that table carry a fixed meaning, and it is worth fixing it here rather than leaving it
to be inferred. \textbf{Proved} means proved in this paper, or in the cases marked so, proved modulo a
citation. \textbf{Verified} means something weaker and quite specific: the statement was checked by
computation against a route independent of the formula being tested, over the range given in the
cell. The three independent routes are brute-force enumeration of the order ideals themselves,
Linusson's own printed tables, and Amata and Crupi's \texttt{ExteriorModules} package for Macaulay2; Section
\ref{sec-verification} says which was used where. For a few entries, which that section lists, the
second route is instead a different algebraic expression of the same count, because beyond small
\(n\) no enumeration is feasible; such a check confirms the algebra, not the count. So ``verified, \(n\le8\)'' is a finite check and not
an argument, and a result labelled only ``verified'' is not claimed to be proved.

A fourth route is of a different kind and is deliberately not folded into the table. The
combinatorial content of the paper has been \textbf{formalised in Lean 4} with Mathlib, which is a stronger
assurance than any of the three above and covers a different, overlapping set of statements. It is
not a column here because the honest status does not fit in a cell: the formalisation is purely
combinatorial, so the counting theorems are certified unconditionally as statements about chained
tuples and only conditionally as statements about \(\lvert\mathrm{FF}(n;\vek d)\rvert\), the bridge to
the algebra being a hypothesis rather than a proof. Section \ref{sec-verification} sets out exactly
what lies inside and what lies outside.

The asymptotics in \(n\), which Linusson leaves as an open problem, are not attempted.
\section{Notation}
\label{sec-notation}
The notation is collected here, before anything below makes use of it. Each line of the table reads
across its two columns as a sentence, and points to the statement in which the symbol is actually
defined. The companion article carries a table in the same spirit, and where the two papers share a
symbol they also share its meaning.

\begingroup\small
\begin{center}
\begin{tabular}{@{}ll@{}}
  \toprule
  \multicolumn{2}{@{}l}{\emph{The ambient objects}}\\
  \(k\)                    & is a field, arbitrary and fixed throughout; nothing depends on it.\\
  \(V\)                    & is an \(n\)-dimensional \(k\)-vector space.\\
  \(\Ext\)                 & is the exterior algebra on \(V\) (Definition~\ref{def-setup}); the letter \(E\) is not used for it.\\
  \(\wedge^jV\)         & is its degree-\(j\) component, of dimension \(\binom nj\).\\
  \(I_\Delta\)             & is the exterior face ideal of \(\Delta\), inside \(\Ext\) (Equation~\ref{eq:exterior-face-ideal}).\\
  \addlinespace
  \multicolumn{2}{@{}l}{\emph{Complexes and their \(f\)-vectors}}\\
  \(\mathcal P([n])\)      & is the lattice of all subsets of \([n]\).\\
  \(\Delta\)               & is an order ideal of \(\mathcal P([n])\), and the void one is allowed.\\
  \(f_j(\Delta)\)          & is the number of members of \(\Delta\) having exactly \(j\) elements.\\
  \(\mathcal F_n\)         & is the set of \(f\)-vectors of all order ideals (Definition~\ref{def-fn}).\\
  \(S_n\)                  & is the set of \emph{proper} ones, those with \(s(0)=1\) and \(s(1)=n\).\\
  \(\shift^d\)             & re-indexes a vector by \(d\), prepending \(d\) zeros when \(d\ge0\).\\
  \addlinespace
  \multicolumn{2}{@{}l}{\emph{The two statistics}}\\
  \(\tp(\vek s)\)          & is the largest \(j\) with \(s(j)\neq0\) (Definition~\ref{def-stats}).\\
  \(\indeg(\vek s)\)       & is the least \(j\) at which \(s(j)\) falls below \(\binom nj\).\\
  \(\indeg(I)\)            & is the initial degree of a monomial ideal (Definition~\ref{def-setup}).\\
  \(A_n[\tau][\rho]\)        & is the joint distribution of \((\tp,\indeg)\) (Lemma~\ref{lem-joint}).\\
  \addlinespace
  \multicolumn{2}{@{}l}{\emph{Sumsets and the count}}\\
  \(\vek d=(d_1,\dots,d_r)\) & is the shift vector, always sorted, with gaps \(\delta_i=d_i-d_{i-1}\).\\
  \(\mathrm{FF}(n;\vek d)\) & is the sumset of the shifted copies of \(\mathcal F_n\) (Definition~\ref{def-fn}).\\
  \(\mathrm{FF}_{\mathrm{pr}}(n;\vek d)\) & is the proper one, built from \(S_n\) in its place.\\
  \(0^r\)                   & is the shift vector \((0,\dots,0)\) of length \(r\), all gaps zero.\\
  \(\mathrm{CT}(n;\vek d)\)   & is the set of chained tuples (Equation~\ref{eq:chained-set}), which \(\Sigma_{\vek d}\) sends onto \(\mathrm{FF}(n;\vek d)\).\\
  \(\mathrm{CT}_{\mathrm{pr}}(n;\vek d)\) & is the proper one, with every entry in \(S_n\).\\
  \(\Sigma_{\vek d}\), \(\Phi_{\vek d}\) & are the summation map (Equation~\ref{eq:sumset-map}) and its inverse (Equation~\ref{eq:water-filling}).\\
  \(F^p(n)\), \(E^p(n,k)\)  & are Linusson's two counts (Definition~\ref{def-linusson}).\\
  \(\vek u_n\), \(T_n(\delta)\), \(\vek 1_n\) & are a row vector, a transfer matrix and a column of ones (Definition~\ref{def-uT}).\\
  \(A'_n\), \(\vek u'_n\), \(T'_n(\delta)\) & are their proper versions, built over \(S_n\) (Theorem~\ref{thm-proper}).\\
  \(B_n\), \(P_n\)          & are the base matrix \(T_n(0)\) and the saturating shift (Proposition~\ref{prop-normalform}).\\
  \addlinespace
  \multicolumn{2}{@{}l}{\emph{Alexander duality}}\\
  \(\Delta^{\vee}\)        & is the Alexander dual of \(\Delta\) (Definition~\ref{def-alexander}).\\
  \(\mathcal B_n\)         & is the binomial box (Equation~\ref{eq:binomial-box}), which contains \(\mathcal F_n\).\\
  \(D_n\)                  & is the duality map (Equation~\ref{eq:def-D}) on \(\mathcal B_n\).\\
  \bottomrule
\end{tabular}
\captionof{table}{The notation of this paper, with the statement in which each symbol is defined.}
\label{tab-notation}
\end{center}
\endgroup

The companion article handles some of these objects under a different convention, so three points of
contact need stating.

First, \textbf{it writes \(E\) where this paper writes \(\Ext\)}. The two are synonyms, and the divergence
is deliberate: the letter \(E\) is spent here on Linusson's \(E^p(n,k)\), which occurs more often in
this paper than the algebra does, while the companion article does not use that function at all.

Second, \textbf{that article distinguishes three conventions} for ``simplicial complex on the ground set
\([n]\)'': (i) all order ideals of \(\mathcal P([n])\), (ii) the non-empty ones, (iii) those containing
\(\emptyset\) and every singleton. \textbf{This paper works throughout in convention (i)}: the void family
is a member of \(\mathcal F_n\), and it has to be, since \(\indeg\) attains its minimum \(0\) there and
that is what fixes the threshold of Proposition \ref{prop-threshold}. Convention (iii) is what \(S_n\)
picks out, and the two papers agree on the symbol \(S_n\).

Third, the companion article writes \(S_n^{+}\) for the \(f\)-vectors of the \textbf{non-empty} order
ideals, convention (ii). In the notation here that is
\begin{align}
  \mathcal F_n &\;=\; S_n^{+}\cup\lbrace\vek 0\rbrace,\label{eq:Fn-vs-Splus}\\
  \lvert\mathcal F_n\rvert &\;=\; \lvert S_n^{+}\rvert+1 .\label{eq:Fn-vs-Splus-card}
\end{align}
We keep \(\mathcal F_n\) rather than \(S_n^{+}\cup\lbrace\vek0\rbrace\) because every matrix below is
indexed by it.
\section{Setup}
\label{sec-setup}
We recall only what is needed; the companion paper has the full development.

\begin{definition}
Throughout, \(k\) is a field. It is arbitrary and fixed, and nothing below depends on it: every
object counted here is a Hilbert function, so only dimensions over \(k\) ever appear, and no
characteristic assumption is used anywhere. Let \(V\) have basis \(e_1,\dots,e_n\) and put
\begin{equation}\label{eq:def-ext}
  \Ext\;=\;\frac{k\langle e_1,\dots,e_n\rangle}{\bigl(e_ie_j+e_je_i,\ e_i^2\bigr)},
\end{equation}
graded with degree-\(j\) component \(\wedge^jV\) of dimension \(\binom nj\). Write
\(\mathfrak m=(e_1,\dots,e_n)\), and for \(\sigma\subseteq[n]\) let \(e_\sigma\) be the corresponding
monomial. For \(r\ge1\) and \(d_1\le\cdots\le d_r\) in \(\mathbb Z\), let
\(F=\bigoplus_{i=1}^r(\Ext)g_i\) with \(\deg g_i=d_i\). A \textbf{monomial submodule} is one generated by
monomials \(e_\sigma g_i\); such an \(M\) decomposes uniquely as \(M=\bigoplus_iI_ig_i\) with each
\(I_i\subseteq \Ext\) a monomial ideal \autocite[Definition 2.1]{AmataCrupi2020}. It is \textbf{proper} if no \(I_i\) has
a minimal generator of degree \(\le1\).

For a monomial ideal \(I\subseteq \Ext\) write
\begin{equation}\label{eq:indeg-ideal}
  \indeg(I)=\min\lbrace j:I_j\ne0\rbrace ,
\end{equation}
its \textbf{initial degree}, with \(\indeg(0)=n+1\) since the minimum is then over an empty set.
Equivalently it is the least degree of a minimal generator of \(I\): if \(j_0\) is least with
\(I_{j_0}\ne0\) then \((\mathfrak mI)_{j_0}\) is built from \(I_{j_0-1}=0\), so every non-zero
element of \(I_{j_0}\) is a minimal generator, and no minimal generator has smaller degree.
\label{def-setup}
\end{definition}

\begin{definition}
An \textbf{order ideal} of \(\mathcal P([n])\), the lattice of subsets of \([n]\), is a family
\(\Delta\subseteq\mathcal P([n])\) closed under taking subsets, with
\(f_j(\Delta)=|\lbrace\sigma\in\Delta:|\sigma|=j\rbrace|\). Put
\begin{align}
  \mathcal F_n &=\lbrace(f_0(\Delta),\dots,f_n(\Delta)):\Delta\ \text{an order ideal of }\mathcal P([n])\rbrace,
    \label{eq:def-Fn}\\
  S_n &=\lbrace \vek s\in\mathcal F_n\;:\;s(0)=1,\ s(1)=n\rbrace .
    \label{eq:def-Sn}
\end{align}
In \eqref{eq:def-Fn} the void family is included and singletons are not required. The elements of
\eqref{eq:def-Sn} are the \textbf{proper} ones, the \(f\)-vectors of simplicial complexes using the full
vertex set.

For the shift map it is convenient to regard an \(f\)-vector as a function on all of \(\mathbb Z\)
that happens to vanish outside \(0\le j\le n\). Write \(\mathbb Z^{(\mathbb Z)}\) for the group of
finitely supported functions \(\mathbb Z\to\mathbb Z\), so that
\(\mathcal F_n\subseteq\mathbb Z^{(\mathbb Z)}\). Then for every \(d\in\mathbb Z\) the shift map
is the group automorphism
\begin{equation}\label{eq:def-shift}
\begin{split}
  \shift^d\colon\mathbb Z^{(\mathbb Z)} &\longrightarrow \mathbb Z^{(\mathbb Z)},\\
  \shift^d(\vek s)(j) &= s(j-d),
\end{split}
\end{equation}
which moves the entry in position \(j\) to position \(j+d\); for \(d\ge0\) it simply prepends
\(d\) zeros, and for \(d<0\) it moves the first \(\lvert d\rvert\) entries into negative
positions, which is what makes negative generator degrees harmless
(Remark \ref{rem-shift-only-gaps}). Finally
\begin{align}
  \mathrm{FF}(n;\vek d)&=\Bigl\lbrace\sum_i\shift^{d_i}(\vek s_i)\;:\;\vek s_i\in\mathcal F_n\Bigr\rbrace,
    \label{eq:def-FF}\\
  \mathrm{FF}_{\mathrm{pr}}(n;\vek d)&=\Bigl\lbrace\sum_i\shift^{d_i}(\vek s_i)\;:\;\vek s_i\in S_n\Bigr\rbrace .
    \label{eq:def-FFpr}
\end{align}

Both parameters are always written out, in \(\mathrm{FF}\) and \(\mathrm{FF}_{\mathrm{pr}}\)
alike, and the same holds for the chained-tuple sets \(\mathrm{CT}\) and
\(\mathrm{CT}_{\mathrm{pr}}\) introduced below, even where \(n\) is fixed for the whole of a
section. The semicolon is what separates \(n\) from the shift vector and is never omitted, so
\(\mathrm{FF}(3;0,0)\) is the sumset at \(n=3\) for \(\vek d=(0,0)\) and not a sumset for a
three-term \(\vek d\).
\label{def-fn}
\end{definition}
\subsection{Notation: which quantities are tuples}
\label{sec:orgea4d433}

\textbf{Bold marks a variable whose value is a tuple; a plain letter is a scalar.} The companion
article draws the same distinction, marking it with an arrow rather than with bold. It needs stating
in full here because two different kinds of index
occur, and the difference between them matters.

\begin{itemize}
\item \(\vek s\in\mathcal F_n\) is a whole \(f\)-vector; \(\vek d=(d_1,\dots,d_r)\) is the shift vector;
\(\vek u_n\) and \(\vek 1_n\) are the row and column vectors of Theorem \ref{thm-main}. Each is a tuple, and
each is set in bold.
\item \textbf{A plain subscript on a bold letter selects a member of the \(r\)-tuple}, and that member may
itself be a vector or a scalar: \(\vek s_i\) is the \(i\)-th \(f\)-vector of an \(r\)-tuple, and so
is still bold; \(d_i\) is the \(i\)-th generator degree, a scalar, and is not.
\textbf{The one exception is the letter \(n\)}, which never indexes a tuple: a subscript \(n\) always
names the member of an \(n\)-indexed family, as in \(\mathcal F_n\) and \(S_n\), and that is what
it does in \(\vek u_n\) and \(\vek 1_n\) as well. Since \(r\)-tuples are indexed by \(i\) and
\(\vek d\) is the only tuple of which \(n\) could plausibly be an index, nothing is ambiguous.
\item \textbf{A parenthesised argument on a vector-valued letter selects a coordinate}, and is
always a scalar: \(s(j)\) is the number of \(j\)-subsets counted by \(\vek s\),
\(\HF(F/M)(j)\) is a dimension, \(x(j)\) a single entry of \(\vek x\). So \(s_i(j)\) (bold
letter, tuple subscript, coordinate argument) is the \(j\)-th coordinate of the \(i\)-th
\(f\)-vector, and is a scalar. \textbf{On a capital or an operator name that is not bold a parenthesised
argument is a parameter instead}, and the result is whatever kind of object the name denotes: \(T_n(\delta)\)
is a matrix, \(\mathrm{FF}(n;\vek d)\) a set, \(F^p(n)\) and \(E^p(n,k)\) integers.
\item Matrices are capitals and are never bold (\(A_n\), \(T_n(\delta)\)); their entries are written
\(A_n[\tau][\rho]\), \(u_n[\tau]\), also scalars. \textbf{Square brackets are only ever entry selection}
(apart from \([n]=\lbrace1,\dots,n\rbrace\)), never a parameter, which is why the \(n\) of \(A_n\) is a subscript and not a third bracket.
\item \textbf{Classical notation is kept where it is standard}, although it does not follow the rules above:
the face numbers \(f_j(\Delta)\) of a complex \(\Delta\), the exterior monomials \(e_\sigma\), and
the standard basis vectors \(e_\tau\) of Proposition \ref{prop-normalform} and after, which are
not bold.
\item \textbf{The usual exception:} a subscript on a module or a graded object still means its
homogeneous component: \(\wedge^jV\), \(F_j\), \(I_j\), \(\mathrm{Mon}_j(L)\) are subspaces, not
coordinates of a vector, and follow the standard convention rather than this one.
\end{itemize}

The distinction does real work in exactly one place, Definition \ref{def-fn} against
Theorem \ref{thm-chaining}: without it, \(s_1\) would mean both ``the first coordinate of \(\vek s\)''
(where it equals \(n\) for a proper vector) and ``the first \(f\)-vector of the tuple''. Under the
convention those are \(s(1)\) and \(\vek s_1\), and cannot be confused.

The dictionary between the algebra and the combinatorics is exact:
\begin{equation}\label{eq:dictionary}
  I\;\longmapsto\;\Delta(I)=\lbrace\sigma\subseteq[n]\;:\;e_\sigma\notin I\rbrace
\end{equation}
is a bijection from the monomial ideals of \(\Ext\) onto the order ideals of \(\mathcal P([n])\), inverse
to \eqref{eq:exterior-face-ideal}; \(\HF\bigl((\Ext)/I\bigr)\) is the \(f\)-vector of \(\Delta(I)\); and both
degenerate cases are included, \(I=\Ext\) going to the void family and \(I=0\) to the full simplex.
Consequently
\(\mathrm{FF}(n;\vek d)=\lbrace\HF(F/M):M\ \text{monomial}\rbrace\) and
\(\mathrm{FF}_{\mathrm{pr}}(n;\vek d)=\lbrace\HF(F/M):M\ \text{proper monomial}\rbrace\).

\begin{remark}
Both sets depend on \(\vek d\) only through the gaps \(\delta_i=d_i-d_{i-1}\ge0\): a global shift
re-indexes \(\mathrm{FF}(n;\vek d)\) bijectively. Negative generator degrees are therefore allowed
throughout, as they are in Amata and Crupi. Sortedness of \(\vek d\) is not optional; see
Remark \ref{rem-sorted}.
\label{rem-shift-only-gaps}
\end{remark}
\section{Linusson's recursion}
\label{sec-linusson}
\begin{definition}
Linusson~\cite[\S3]{Linusson1999} Fix \(p\ge0\). For \(n\ge-1\), let \(F^p(n)\) be the number of
\(f\)-vectors of simplicial complexes on at most \(p\) vertices, no face of which has more than
\(n\) elements. In coordinates those two bounds read \(f_1\le p\) and \(f_j=0\) for \(j>n\). Both
degenerate \(f\)-vectors are counted: \((0,0,\dots)\), of the empty complex, and
\((1,0,0,\dots)\), of the complex whose only face is \(\emptyset\). For \(k\ge-1\), let
\(E^p(n,k)\) count those among them that contain every face of at most \(k\)
elements but not every face of \(k+1\) elements, that is \(f_i=\binom pi\) for \(i\le k\) but
\(f_{k+1}<\binom p{k+1}\).

We also set
\begin{equation}\label{eq:E-minus-two}
  E^p(-2,k)=0\qquad\text{for every }k .
\end{equation}
This is a convention, and it is needed because the clause \(f_j=0\) for \(j>n\) stops doing its
work at \(n=-2\). That clause constrains only the \(f_j\) with \(j\ge0\), so at \(n=-2\) it says
exactly what it says at \(n=-1\), and the empty complex still satisfies it; taken literally the
definition would therefore give \(E^p(-2,-1)=1\) rather than \(0\). Setting
\eqref{eq:E-minus-two} restores the reading the clause has for every \(n\ge-1\), which is what
Lemma \ref{lem-joint} uses it for. It is used nowhere else.
\label{def-linusson}
\end{definition}

The two parameters play different roles, and the notation \(F^p(n)\) keeps both: \(p\) bounds the
vertex set, \(n\) bounds the size of a face. They coincide at \(p=n\), which is the case this paper
runs on, and there \(F^n(n)=\lvert\mathcal F_n\rvert\). \textbf{Both letters are Linusson's own}, kept so
that this definition and Theorem \ref{thm-linusson} can be read against his paper as printed. The
cost is that the \(n\) of this one definition is not the \(n\) of the rest of this paper, which
fixes \(\dim V=n\) throughout; wherever the two would collide, the face bound is written \(m\) or
\(\tau\) instead, as in Lemma \ref{lem-joint}.

\begin{theorem}
Linusson~\cite[Theorem 3.2]{Linusson1999} For \(0\le k\) and \(1\le n\le p\),
\begin{equation}\label{eq:auto-8}
  F^p(n)=\begin{cases}
    1+\sum_{i=0}^{n}E^{p-1}(n,i)\,F^{p-1}(i-1), & n<p,\\
    1+F^p(p-1), & n=p,
  \end{cases}
\end{equation}
\begin{equation}\label{eq:auto-9}
  E^p(n,k)=\begin{cases}
    \sum_{i=k}^{n}E^{p-1}(n,i)\,E^{p-1}(i-1,k-1), & k\le n<p,\\
    E^p(p-1,k), & k<n=p,\\
    1, & k=n=p,
  \end{cases}
\end{equation}
with \(F^p(-1)=1\), \(F^p(0)=2\), \(F^0(n)=2\), \(E^p(n,-1)=1\), \(E^p(n,n)=1\), and
\(E^p(n,k)=0\) for \(k>n\). Moreover \(F^p(n)=\sum_kE^p(n,k)\).
\label{thm-linusson}
\end{theorem}

\begin{remark}
The transcription above was checked entry by entry against the appendix of Linusson's preprint
\autocite{Linusson1999preprint}: all \(100\)
printed entries of the preprint's Table 5, of \(F^p(n)\), and all \(99\) of its Table 6, of
\(E^p(p,k)\), are
reproduced exactly, and an
independent brute-force enumeration of order ideals agrees for every \(F^p(n)\) and
\(E^p(n,k)\) with \(p\le5\). \textbf{The appendix is in the preprint only}: the published version is
shorter, carries no appendix, and prints a single table, of \(M^p(n)\). Linusson's tables are
therefore always named as his below, since this paper has tables of its own and they are numbered
in the same series; every numbered \textbf{result} cited is the published one. One wrinkle: at \(k=n=p\) the clause
``\(f_{k+1}<\binom p{k+1}\)'' read literally is vacuously false, since \(\binom p{p+1}=0\); but
Linusson's stated boundary value \(E^p(n,n)=1\) and his Table 6 diagonal both require the clause to
be dropped there. We follow Linusson.
\label{rem-transcription}
\end{remark}

We use two consequences repeatedly. Setting \(p=n\) gives \(F^n(n)=|\mathcal F_n|\), and the
proper count is a difference of two such.

\begin{remark}
\(|S_n|=F^n(n)-F^{n-1}(n-1)\), giving
\begin{equation}\label{eq:auto-10}
  \lvert S_n\rvert = 1,\,2,\,5,\,16,\,70,\,457,\,4908,\,95248,\,3617645,\,286007155
  \qquad(n=1,\dots,10),
\end{equation}
whose first five terms agree with direct enumeration. The sequence \(F^p(p)\) is
OEIS A007695 \autocite{OEIS}. There is also a second identity,
\begin{equation}\label{eq:Sn-as-Fnn2}
  \lvert S_n\rvert=F^n(n-2),
\end{equation}
for instance \(|S_9|=F^9(7)=3617645\), a value printed in Linusson's Table 5. That identity is
Proposition \ref{prop-fnn2} below, proved by Alexander duality; equivalently, via
\(F^n(n)=1+F^n(n-1)\), it says
\begin{equation}\label{eq:Fnn2-equivalent}
  F^n(n-1)-F^n(n-2)=F^{n-1}(n-1)-1 .
\end{equation}
Nothing later depends on it.
\label{obs-fnn2}
\end{remark}
\section{Two statistics, and the joint distribution}
\label{sec-statistics}
\begin{definition}
For \(\vek s\in\mathcal F_n\) put
\begin{align}
  \indeg(\vek s) &= \min\lbrace j:s(j)<\tbinom nj\rbrace,\label{eq:def-indeg}\\
  \tp(\vek s) &= \max\lbrace j:s(j)>0\rbrace,\label{eq:def-top}
\end{align}
with \(\indeg(\vek s)=n+1\) for the full simplex and \(\tp(\vek s)=-1\) for the void family. The
symbol \(\indeg\) is the one of \eqref{eq:indeg-ideal}: this is its combinatorial shadow, and
Lemma \ref{lem-dictionary}(2) is the statement that the two agree.
\label{def-stats}
\end{definition}

\begin{definition}
(\emph{The joint-distribution matrix.}) Let \(A_n\) be the matrix whose entry
\begin{equation}\label{eq:def-A}
  A_n[\tau][\rho]=\#\bigl\lbrace \vek s\in\mathcal F_n:\tp(\vek s)=\tau,\ \indeg(\vek s)=\rho\bigr\rbrace
\end{equation}
counts the elements of \(\mathcal F_n\) with the prescribed pair of statistics. Rows are indexed by
\(\tp\) and columns by \(\indeg\), over
\begin{equation}\label{eq:def-A-range}
  \tau=-1,0,1,\dots,n,\qquad \rho=0,1,\dots,n+1 ,
\end{equation}
so \(A_n\) is \((n+2)\times(n+2)\). Two features of the range are forced by the conventions of
Definition \ref{def-stats} and are used repeatedly below: \(A_n\) is supported on \(\rho\le\tau+1\), and
the two extreme rows are the single entries \(A_n[-1][0]=1\) (the void family) and \(A_n[n][n+1]=1\)
(the full simplex). Every row sums to the number of \(\vek s\) with that value of \(\tp\), and the
grand total is \(\lvert\mathcal F_n\rvert\).
\label{def-A}
\end{definition}

Note \(\indeg(\vek s)\le\tp(\vek s)+1\) always. For \(\tp(\vek s)<n\) this is because
\begin{equation}\label{eq:indeg-le-top}
  s\bigl(\tp(\vek s)+1\bigr)=0<\tbinom n{\tp(\vek s)+1},
\end{equation}
so the minimum defining \(\indeg\) is attained no later
than \(\tp(\vek s)+1\); and for \(\tp(\vek s)=n\) it holds because \(\indeg(\vek s)\le n+1\) outright, by
the convention of Definition \ref{def-stats}. The bound is therefore unconditional, the binomial
coefficient \(\binom n{\tp(\vek s)+1}\) vanishing in exactly the case the second clause covers. So
\(A_n\) is supported on \(\rho\le\tau+1\).

\begin{lemma}
With the convention \eqref{eq:E-minus-two} of Definition \ref{def-linusson},
\begin{equation}\label{eq:auto-12}
  A_n[\tau][\rho]\;=\;E^n(\tau,\rho-1)-E^n(\tau-1,\rho-1).
\end{equation}
\label{lem-joint}
\end{lemma}

\begin{proof}
Work at \(p=n\). Every \(\vek s\in\mathcal F_n\) has \(f_1\le n\); conversely every \(f\)-vector counted by
\(E^n(m,k)\) has \(f_1\le n\), hence belongs to a complex on at most \(n\) vertices, which embeds in
\([n]\). So the objects counted by \(E^n(m,k)\) are exactly certain elements of \(\mathcal F_n\).
Under that identification Linusson's clause ``\(f_j=0\) for \(j>m\)'' says precisely \(\tp(\vek s)\le m\),
and his clause ``\(f_i=\binom ni\) for \(i\le k\) but \(f_{k+1}<\binom n{k+1}\)'' (with the boundary
convention of Remark \ref{rem-transcription}) says precisely \(\indeg(\vek s)=k+1\). Hence
\[
  E^n(m,k)=\#\lbrace \vek s\in\mathcal F_n:\tp(\vek s)\le m,\ \indeg(\vek s)=k+1\rbrace ,
\]
and subtracting the count at \(m=\tau-1\) from the count at \(m=\tau\) leaves \(\tp(\vek s)=\tau\)
exactly. For \(\tau=-1\) the subtrahend would count \(f\)-vectors with \(\tp\le-2\), of which there
are none (even the void \(f\)-vector has \(\tp=-1\)), which is the stated convention.
\end{proof}

So one run of Theorem \ref{thm-linusson} at \(p=n\) delivers the entire joint distribution. Its
\(\indeg\) marginal telescopes to \(E^n(n,\rho-1)\), which is Linusson's Table 6.

\begin{mexample}
At \(n=3\) the set \(\mathcal F_3\) has ten elements. Listing them with their statistics, first the
\(f\)-vector and then \((\tp,\indeg)\), gives
\[
\begin{array}{llll}
 (0,0,0,0)\colon(-1,0) & (1,0,0,0)\colon(0,1) & (1,1,0,0)\colon(1,1) & (1,2,0,0)\colon(1,1)\\
 (1,2,1,0)\colon(2,1)  & (1,3,0,0)\colon(1,2) & (1,3,1,0)\colon(2,2) & (1,3,2,0)\colon(2,2)\\
 (1,3,3,0)\colon(2,3)  & (1,3,3,1)\colon(3,4) & &
\end{array}
\]
and collecting them into \eqref{eq:def-A} (rows \(\tau=-1,\dots,3\) downward, columns
\(\rho=0,\dots,4\) rightward)
\[
  A_3=\begin{pmatrix}
    1&0&0&0&0\\ 0&1&0&0&0\\ 0&2&1&0&0\\ 0&1&2&1&0\\ 0&0&0&0&1
  \end{pmatrix},
  \qquad \text{row sums } (1,1,3,4,1),\ \text{total }10=\lvert\mathcal F_3\rvert .
\]
The support condition \(\rho\le\tau+1\) is the vanishing above the superdiagonal; the two corner
entries are the void family and the full simplex; and the antidiagonal symmetry of
Corollary \ref{cor-persymmetry} is already visible, the \(2\) at \((\tau,\rho)=(1,1)\) matching the
\(2\) at \((2,2)\).

The listing also locates the two degenerate \(f\)-vectors that Definition \ref{def-linusson} counts
along with the rest, since it is easy to guess the wrong pair. They are the first two entries:
\((0,0,0,0)\), the \(f\)-vector of the void family, which has no faces at all; and \((1,0,0,0)\), the
\(f\)-vector of \(\lbrace\emptyset\rbrace\), the complex whose only face is the empty one. They are
the two with \(\tp\le0\), they occupy the two extreme rows of the top-left corner of \(A_3\), and they
are exactly the two members of \(\mathcal F_3\) that no non-empty complex on a non-empty vertex set
produces. The full simplex \((1,3,3,1)\) is \textbf{not} one of them: it is a perfectly ordinary complex, and
it sits at the opposite corner of \(A_3\) as the unique element with \(\tp=3\).
\label{ex-A}
\end{mexample}

\begin{remark}
Lemma \ref{lem-joint} was also checked entrywise, for \(n\le6\) and with no mismatches, against
direct enumeration of the order ideals themselves. At \(n=6\), with rows \(\tau=-1,\dots,6\) and
columns \(\rho=0,\dots,7\), it is
\begin{equation}\label{eq:auto-13}
\begin{pmatrix}
1&0&0&0&0&0&0&0\\
0&1&0&0&0&0&0&0\\
0&5&1&0&0&0&0&0\\
0&20&14&1&0&0&0&0\\
0&43&84&19&1&0&0&0\\
0&25&143&84&14&1&0&0\\
0&1&25&43&20&5&1&0\\
0&0&0&0&0&0&0&1
\end{pmatrix}
\end{equation}
\label{rem-A-verified}
\end{remark}
\subsection{Alexander duality, and the identity of Remark \ref{obs-fnn2}}
\label{sec:org819984b}

The matrix \eqref{eq:auto-13} is visibly symmetric about its antidiagonal. Alexander duality
explains the symmetry, and proves the identity left open in Section \ref{sec-linusson} as well.

\begin{definition}
(\emph{The Alexander dual, the binomial box, and the duality map.}) For an order ideal \(\Delta\) of
\(\mathcal P([n])\), the \textbf{Alexander dual} is
\begin{equation}\label{eq:def-alexander-dual}
  \Delta^{\vee}=\bigl\lbrace\sigma\subseteq[n]\;:\;[n]\setminus\sigma\notin\Delta\bigr\rbrace .
\end{equation}
The \textbf{binomial box} is the set of integer vectors obeying the entrywise bound,
\begin{equation}\label{eq:binomial-box}
  \mathcal B_n=\bigl\lbrace\vek s\in\mathbb Z_{\ge0}^{n+1}\;:\;s(j)\le\tbinom nj\ \text{ for all }j\bigr\rbrace
  \;\supseteq\;\mathcal F_n ,
\end{equation}
and the \textbf{duality map} \(D_n\) is the self-map of \(\mathcal B_n\) given coordinatewise by
\begin{equation}\label{eq:def-D}
  D_n(\vek s)(j)=\tbinom nj-s(n-j) \qquad (0\le j\le n).
\end{equation}
Per the convention of Section \ref{sec-notation}, the parenthesised argument in
\eqref{eq:def-D} selects a coordinate, so \(s(m)\) and \(D_n(\vek s)(j)\) are integers while
\(\vek s\) and \(D_n(\vek s)\) are vectors. The subscript on \(D_n\) and \(\mathcal B_n\) is the
ground-set size, as in \(\mathcal F_n\) and \(S_n\); \(\Delta^{\vee}\) carries no subscript because the
complement in \eqref{eq:def-alexander-dual} is taken inside the \([n]\) that \(\Delta\) already
names.
\label{def-alexander}
\end{definition}

\begin{lemma}
(\emph{Alexander duality on \(f\)-vectors.}) Let \(\Delta\) be an order ideal of \(\mathcal P([n])\),
and let \(\Delta^{\vee}\), \(\mathcal B_n\) and \(D_n\) be as in Definition \ref{def-alexander}.
Then \(\Delta^{\vee}\) is an order ideal, \((\Delta^{\vee})^{\vee}=\Delta\), and
\begin{equation}\label{eq:alexander-f}
  f_j(\Delta^{\vee})=\tbinom nj-f_{n-j}(\Delta).
\end{equation}
Consequently \(D_n\) restricts to an involution of \(\mathcal F_n\), and it conjugates the two
statistics into each other:
\begin{align}
  \indeg(D_n(\vek s)) &= n-\tp(\vek s),\label{eq:alexander-indeg}\\
  \tp(D_n(\vek s)) &= n-\indeg(\vek s).\label{eq:alexander-top}
\end{align}
Of these assertions, only one uses the combinatorics: that \(\Delta^{\vee}\) is an order ideal,
equivalently that \(D_n\) maps \(\mathcal F_n\) into itself. On the box \(\mathcal B_n\) itself
\(D_n\) is already an involution and already satisfies
\eqref{eq:alexander-indeg}--\eqref{eq:alexander-top}. The involution identity is formal and holds
for every integer vector; the entrywise bound of \eqref{eq:binomial-box} is what makes \(D_n\) a
self-map of \(\mathcal B_n\), and it is what the two conjugations need, since \(\tp\) asks whether a
coordinate is non-zero and \eqref{eq:D-above} only says whether it is positive.
\label{lem-alexander}
\end{lemma}

\begin{proof}
If \(\sigma\in\Delta^{\vee}\) and \(\tau\subseteq\sigma\) then
\[
  [n]\setminus\sigma\;\subseteq\;[n]\setminus\tau ;
\]
were \([n]\setminus\tau\in\Delta\), downward closure would put
\([n]\setminus\sigma\in\Delta\), which it is not. So \(\Delta^{\vee}\) is an order ideal. Double
complementation gives \((\Delta^{\vee})^{\vee}=\Delta\). For \eqref{eq:alexander-f}, complementation
\(\sigma\mapsto[n]\setminus\sigma\) is a bijection from the \(j\)-subsets to the \((n-j)\)-subsets,
under which the \(j\)-members of \(\Delta^{\vee}\) correspond to the \((n-j)\)-subsets \textbf{not} in
\(\Delta\). The formula shows \(D_n(\vek s)\) depends only on \(\vek s\), so \(D_n\) is a well-defined
involution of \(\mathcal F_n\). Finally \(D_n(\vek s)(j)<\binom nj\) if and only if \(s(n-j)>0\), which
is \eqref{eq:alexander-indeg}--\eqref{eq:alexander-top}; the conventions \(\tp=-1\) for the void family and
\(\indeg=n+1\) for the full simplex are exactly what make the two extremes correspond.

For the last paragraph of the statement, observe that this final step used no property of \(\vek s\)
beyond \(0\le s(n-j)\le\binom n{n-j}\), which makes every coordinate of \(D_n(\vek s)\)
non-negative, so that ``non-zero'' and ``positive'' agree. Both equivalences read off the definition
of \(D_n\),
\begin{align}
  D_n(\vek s)(j)<\tbinom nj \;&\iff\; s(n-j)>0 ,\label{eq:D-below}\\
  D_n(\vek s)(j)>0 \;&\iff\; s(n-j)<\tbinom n{n-j} ,\label{eq:D-above}
\end{align}
and, since \(\binom nj=\binom n{n-j}\), with no bound needed at all,
\begin{equation}\label{eq:D-involution}
  D_n(D_n(\vek s))(j)=\tbinom nj-\Bigl(\tbinom n{n-j}-s(j)\Bigr)=s(j) ,
\end{equation}
so \(D_n\), which maps \(\mathcal B_n\) into itself by the bound, is an involution of
\(\mathcal B_n\). What \eqref{eq:alexander-f} adds is that this involution of the box restricts to one of the
subset \(\mathcal F_n\), and that is where downward closure is used.
\end{proof}

\begin{mexample}
(\emph{The involution on \(\mathcal F_3\).}) Continuing Example \ref{ex-A}. With \(n=3\) the formula
\(D_3(\vek s)(j)=\binom 3j-s(3-j)\) reads \(D_3(\vek s)=\bigl(1-s(3),\,3-s(2),\,3-s(1),\,1-s(0)\bigr)\):
reverse the vector and subtract it from Pascal's row \((1,3,3,1)\). Applied to the ten elements listed
in Example \ref{ex-A}, with \((\tp,\indeg)\) beside each:
\begingroup\small
\begin{center}
\begin{tabular}{@{}llcll@{}}
  \toprule
  \(\vek s\) & \((\tp,\indeg)\) & & \(D_3(\vek s)\) & \((\tp,\indeg)\)\\
  \midrule
  \((0,0,0,0)\) & \((-1,0)\) & \(\longleftrightarrow\) & \((1,3,3,1)\) & \((3,4)\)\\
  \((1,0,0,0)\) & \((0,1)\)  & \(\longleftrightarrow\) & \((1,3,3,0)\) & \((2,3)\)\\
  \((1,1,0,0)\) & \((1,1)\)  & \(\longleftrightarrow\) & \((1,3,2,0)\) & \((2,2)\)\\
  \((1,2,0,0)\) & \((1,1)\)  & \(\longleftrightarrow\) & \((1,3,1,0)\) & \((2,2)\)\\
  \addlinespace
  \((1,2,1,0)\) & \((2,1)\)  & \(=\)                   & \((1,2,1,0)\) & \((2,1)\)\\
  \((1,3,0,0)\) & \((1,2)\)  & \(=\)                   & \((1,3,0,0)\) & \((1,2)\)\\
  \bottomrule
\end{tabular}
\captionof{table}{The duality \(D_3\) on \(\mathcal F_3\), each \(f\)-vector with its \((\tp,\indeg)\).}
\label{tab-duality-n3}
\end{center}
\endgroup
Four two-cycles and two fixed points, accounting for all ten elements. Every row satisfies
\eqref{eq:alexander-indeg}--\eqref{eq:alexander-top}: in the first, \(\indeg(D_3(\vek s))=4=3-(-1)=n-\tp(\vek s)\) and
\(\tp(D_3(\vek s))=3=3-0=n-\indeg(\vek s)\), and so on down. The two fixed points are the
self-dual complexes, and note that a fixed point must have \(\tp+\indeg=n\), which is why there are
any: the antidiagonal \(\tau+\rho=3\) of \(A_3\) is where the persymmetry of
Corollary \ref{cor-persymmetry} has its fixed entries.

The first and last rows exchange the two extremes, the void family against the full simplex, and the
second row exchanges \(\lbrace\emptyset\rbrace\) against the boundary of the simplex. Both pairs are
cases where leaving a degenerate \(f\)-vector out of \(\mathcal F_n\) would break the involution, which is
the point made after Proposition \ref{prop-fnn2} below.

Finally, reading the table for the five proper vectors, \(D_3\) carries
\[
  S_3=\bigl\lbrace(1,3,0,0),(1,3,1,0),(1,3,2,0),(1,3,3,0),(1,3,3,1)\bigr\rbrace
\]
onto \(\lbrace(1,3,0,0),(1,2,0,0),(1,1,0,0),(1,0,0,0),(0,0,0,0)\rbrace\), which is exactly
\(\lbrace\vek g\in\mathcal F_3:\tp(\vek g)\le1\rbrace\) with \(1=n-2\). So \(|S_3|=F^3(1)=5\), which
is Proposition \ref{prop-fnn2} at \(n=3\).
\label{ex-alexander-f3}
\end{mexample}

\begin{corollary}
\(A_n[\tau][\rho]=A_n[n-\rho][n-\tau]\) for all \(\tau,\rho\): the matrix \(A_n\) is \textbf{persymmetric}, which
is to say that it is symmetric about its antidiagonal (the diagonal running from the top-right entry
to the bottom-left one) rather than about the main diagonal. With
\(A_n\) read as in \eqref{eq:def-A} for arbitrary integers \(\tau,\rho\) (so that it vanishes
outside the range \eqref{eq:def-A-range}), the identity holds on all of \(\mathbb Z^2\), and no
check that the reflected indices stay in range is needed when it is applied inside a sum.
\label{cor-persymmetry}
\end{corollary}

\begin{proof}
By \eqref{eq:alexander-indeg}--\eqref{eq:alexander-top}, \(D_n\) is a bijection of \(\mathcal F_n\) carrying the class
\((\tp,\indeg)=(\tau,\rho)\) onto the class \((n-\rho,n-\tau)\). For the unrestricted form, note that
\eqref{eq:def-A} counts an empty set whenever \(\tau\) or \(\rho\) leaves the range
\eqref{eq:def-A-range}, and that the reflection
\begin{equation}\label{eq:persymmetry-reflection}
  (\tau,\rho)\;\longmapsto\;(n-\rho,\,n-\tau)
\end{equation}
maps that range onto itself, so both sides vanish together outside it.
\end{proof}

\begin{proposition}
\(\lvert S_n\rvert=F^n(n-2)\) for every \(n\ge1\).
\label{prop-fnn2}
\end{proposition}

\begin{proof}
By Definition \ref{def-linusson} at \(p=n\), where the condition \(f_1\le n\) is automatic,
\(F^n(m)\) is the number of \(\vek s\in\mathcal F_n\) with \(\tp(\vek s)\le m\). Now
\begin{align}
  D_n(\vek s)(n)   &= \tbinom nn-s(0) = 1-s(0) ,\label{eq:D-top}\\
  D_n(\vek s)(n-1) &= \tbinom n{n-1}-s(1) = n-s(1) ,\label{eq:D-next}
\end{align}
so
\begin{align*}
  \vek s\in S_n
  &\iff s(0)=1\ \text{and}\ s(1)=n\\
  &\iff D_n(\vek s)(n)=D_n(\vek s)(n-1)=0
   \iff \tp(D_n(\vek s))\le n-2 .
\end{align*}
As \(D_n\) is an involution of \(\mathcal F_n\) by Lemma \ref{lem-alexander}, it restricts to a bijection
\begin{equation}\label{eq:Sn-bijection}
  D_n\colon S_n\;\xrightarrow{\ \sim\ }\;\bigl\lbrace\vek g\in\mathcal F_n\;:\;\tp(\vek g)\le n-2\bigr\rbrace ,
\end{equation}
a set of cardinality \(F^n(n-2)\).
\end{proof}

\begin{remark}
The hypothesis \(n\ge1\) is sharp rather than cosmetic. At \(n=0\) the left side of
Proposition \ref{prop-fnn2} is \(1\), since \(\lbrace\emptyset\rbrace\) is the one order ideal of
\(\mathcal P(\emptyset)\) meeting condition (iii) of Section \ref{sec-notation} and its \(f\)-vector is
\((1)\), while the right side is \(F^0(-2)=0\).
\label{rem-fnn2-sharp}
\end{remark}

Example \ref{ex-alexander-f3} above exhibits the involution, the persymmetry and this proposition
together on the ten elements of \(\mathcal F_3\).

In words: \textbf{the proper \(f\)-vectors are Alexander-dual to the \(f\)-vectors of order ideals of
dimension at most \(n-3\)}, and properness, two conditions at the bottom of \(\vek s\), becomes
the vanishing of the top two entries of \(D_n(\vek s)\). Including both degenerate \(f\)-vectors in
\(\mathcal F_n\), which Definition \ref{def-fn} does and which convention (i) of Section
\ref{sec-notation} requires, is what makes the duality an involution of the whole set rather than a
partial map: the void family and the full simplex are exchanged, as are \(\lbrace\emptyset\rbrace\)
and the boundary of the simplex.

\begin{definition}
(\emph{Chained tuples.}) Fix \(n\ge1\), \(r\ge1\) and a sorted \(\vek d\), with gaps
\(\delta_i=d_i-d_{i-1}\). An \(r\)-tuple \((\vek s_1,\dots,\vek s_r)\in\mathcal F_n^{\,r}\) is
\textbf{chained} for \(\vek d\) if
\begin{equation}\label{eq:chaining-condition}
  \tp(\vek s_{i-1})\;<\;\indeg(\vek s_i)+\delta_i \qquad (i=2,\dots,r) ,
\end{equation}
and the set of all of them is written
\begin{equation}\label{eq:chained-set}
  \mathrm{CT}(n;\vek d)=\bigl\lbrace(\vek s_1,\dots,\vek s_r)\in\mathcal F_n^{\,r}\;:\;
    \text{the tuple is chained}\bigr\rbrace ,
\end{equation}
the letters being the initials of \textbf{chained tuples}. The tuple is \textbf{proper-chained} if in addition
every entry lies in \(S_n\), and the set of those is written \(\mathrm{CT}_{\mathrm{pr}}(n;\vek d)\).

Theorem \ref{thm-chaining} below shows that \(\lvert\mathrm{CT}(n;\vek d)\rvert=\lvert\mathrm{FF}(n;\vek
d)\rvert\) and Theorem \ref{thm-proper} the proper analogue. Algebraically the chained tuples are the
component tuples of the lex submodules, and the proper-chained ones those of the \textbf{proper} lex
submodules; that is how both theorems are proved.
\label{def-chained}
\end{definition}

\begin{remark}
The duality has one further consequence, and it constrains any future answer to the bijective
question of Section \ref{sec-verification}. Applying \(D_n\) to every entry and reversing the order, the map
\begin{equation}\label{eq:gap-reversal-map}
  (\vek s_1,\dots,\vek s_r)\longmapsto\bigl(D_n(\vek s_r),\dots,D_n(\vek s_1)\bigr)
\end{equation}
carries tuples chained for \((\delta_2,\dots,\delta_r)\) to tuples chained for the reversed gap
vector \((\delta_r,\dots,\delta_2)\), by \eqref{eq:alexander-indeg}--\eqref{eq:alexander-top}. So
\(|\mathrm{FF}(n;\vek d)|\) is
invariant under reversing the gaps, which was verified for \(n\le6\) over six gap vectors,
including non-palindromic ones. This is an ingredient for a bijective proof of
Theorem \ref{thm-chaining}, and a symmetry that any such proof must respect; it does not solve that
problem.

Two things delimit the symmetry sharply, both settled by computation, and the second of them is the
subject of Remark \ref{rem-proper-not-reversible} below.

\textbf{Nothing more than reversal is true:} \(\lvert\mathrm{FF}(n;\vek d)\rvert\) is not invariant under arbitrary
permutation of the gaps. At \(n=2\), \(r=4\), the three arrangements of the multiset
\(\lbrace0,0,1\rbrace\) give
\begin{equation}\label{eq:no-permutation-invariance}
  (0,0,1)\mapsto96,\qquad (0,1,0)\mapsto111,\qquad (1,0,0)\mapsto96 ,
\end{equation}
and the two that agree are precisely the two that are reverses of each other. So the involution of
this remark is the whole of the symmetry.
\label{rem-gap-reversal}
\end{remark}

\begin{remark}
(\emph{Gap reversal fails for the proper count.}) The second delimitation is that the proper count does
not inherit the symmetry of Remark \ref{rem-gap-reversal}. It is also one reason the proper variant of
Section \ref{sec-proper} costs more than imposing a restriction on the improper theory does.

The argument there cannot be run for \(\mathrm{FF}_{\mathrm{pr}}(n;\vek d)\), because \(D_n\) does not
preserve \(S_n\): by Proposition \ref{prop-fnn2} it carries \(S_n\) onto
\(\lbrace\vek g\in\mathcal F_n:\tp(\vek g)\le n-2\rbrace\), a different subset of \(\mathcal F_n\), so
the reversal map of Remark \ref{rem-gap-reversal} does not carry \(\mathrm{CT}_{\mathrm{pr}}(n;\vek d)\)
into itself. It does send a proper-chained tuple to a chained one; what it need not preserve is the
entrywise properness.

That much would leave open whether the invariance nonetheless holds for some other reason. It does
not: the conclusion itself is false, not merely unproved. At \(n=3\) the gap vectors
\((\delta_2,\delta_3)=(0,1)\) and \((1,0)\) give
\begin{align}
  \lvert\mathrm{FF}_{\mathrm{pr}}(3;0,0,1)\rvert &= 45,\label{eq:proper-not-reversible}\\
  \lvert\mathrm{FF}_{\mathrm{pr}}(3;0,1,1)\rvert &= 51,\label{eq:proper-not-reversible-b}
\end{align}
while the corresponding improper counts agree at \(236\), exactly as Remark
\ref{rem-gap-reversal} predicts. So \(45\ne51\) is a counterexample: any attempt to transport the
reversal symmetry to the proper count must fail.

The asymmetry is worth keeping in view when reading Section \ref{sec-proper}. Properness is two
conditions at the bottom of \(\vek s\), and Alexander duality moves the bottom of a vector to its
top; a symmetry built out of that duality therefore has no reason to respect a condition pinned to
one end, and here it does not.
\label{rem-proper-not-reversible}
\end{remark}
\section{The chaining reformulation}
\label{sec-chaining}
This section replaces a sumset cardinality by a count of tuples with a nearest-neighbour constraint.
The mechanism is that Amata and Crupi's lex submodules give a canonical representative for each
achievable Hilbert function, and that being a lex submodule is a condition on consecutive components
only.

\begin{lemma}
Let \(I\subseteq \Ext\) be a monomial ideal and \(\vek s=\HF\bigl((\Ext)/I\bigr)\in\mathcal F_n\).
\begin{enumerate}
\item For every integer \(t\ge0\): \(\mathfrak m^t\subseteq I\) if and only if \(\tp(\vek s)<t\).
\item \(\indeg(I)=\indeg(\vek s)\).
\end{enumerate}
\label{lem-dictionary}
\end{lemma}

\begin{proof}
(1) \(\mathfrak m^t\) is spanned by the monomials of degree \(\ge t\), so
\(\mathfrak m^t\subseteq I\) iff \(I_j=\wedge^jV\) for all \(j\ge t\), iff \(s(j)=0\) for all \(j\ge t\),
iff \(\tp(\vek s)<t\). This covers \(t=0\) (then \(\mathfrak m^0=\Ext\subseteq I\) iff \(I=\Ext\) iff
\(\tp(\vek s)=-1<0\)) and \(t>n\) (then \(\mathfrak m^t=0\) and \(\tp(\vek s)\le n<t\), both sides true). For
\(t<0\) the equivalence fails, the left side being vacuously true and the right side impossible,
which is exactly why sortedness of \(\vek d\) is needed in Theorem \ref{thm-chaining}.
(2) \(\dim I_j=\binom nj-s(j)\), so the least \(j\) with \(I_j\ne0\) is the least \(j\) with
\(s(j)<\binom nj\). For \(I=0\) both sides are \(n+1\) by convention; any value exceeding \(n\) gives
the same truth value in (1), so the convention is harmless.
\end{proof}

\begin{theorem}
Fix \(n\ge1\), \(r\ge1\) and integers \(d_1\le d_2\le\cdots\le d_r\), with
\(\delta_i=d_i-d_{i-1}\). With \(\mathrm{CT}(n;\vek d)\) the set of chained tuples of
Definition \ref{def-chained},
\begin{equation}\label{eq:auto-14}
  \lvert\mathrm{FF}(n;\vek d)\rvert\;=\;\lvert\mathrm{CT}(n;\vek d)\rvert ,
\end{equation}
\label{thm-chaining}
\end{theorem}

\begin{proof}
(a) By Definition \ref{def-setup} and the dictionary, \(\mathrm{FF}(n;\vek d)\) is the set of \(\HF(F/M)\)
for \(M\) a monomial submodule. Indeed, for \(M=\bigoplus_iI_ig_i\),
\[
  \HF\Bigl(\frac FM\Bigr)(j)=\sum_i\dim\Bigl(\frac{\Ext}{I_i}\Bigr)_{j-d_i}
    =\sum_i\shift^{d_i}(\vek s_i)(j) ,
\]
and as \((I_1,\dots,I_r)\) ranges over all tuples of monomial ideals, unconstrained by one another,
\((\vek s_1,\dots,\vek s_r)\) ranges over all of \(\mathcal F_n^{\,r}\).

(b) Restricted to \textbf{lex} ideals, \(I\mapsto\HF\bigl((\Ext)/I\bigr)\) is a bijection onto \(\mathcal F_n\). It is well
defined (a lex ideal is monomial); injective, since a lex ideal is degreewise the lex-initial segment
of \(\mathrm{Mon}_j(\Ext)\) of size \(\binom nj-s(j)\) and is therefore determined by \(\vek s\); and
surjective by Kruskal--Katona in its exterior form \autocite[Theorem 4.1]{AramovaHerzogHibi1997}, every \(f\)-vector being attained by a compressed
family, whose complementary ideal is lex.

(c) By the equivalence of Definitions 2.6 and 2.7 of \autocite{AmataCrupi2020KK}, a lex submodule of \(F\) is exactly
an \(L=\bigoplus_iL_ig_i\) with every \(L_i\) a lex ideal and
\[
  \mathfrak m^{\,\rho_i+d_i-d_{i-1}}\;\subseteq\;L_{i-1}
  \qquad (i=2,\dots,r),\qquad \rho_i=\indeg(L_i) .
\]
The exponent is \(\ge0\) because \(\rho_i\ge0\) and \(d_i\ge d_{i-1}\), so
Lemma \ref{lem-dictionary} applies and turns the side condition into
\(\tp(\vek s_{i-1})<\indeg(\vek s_i)+\delta_i\). With (b) componentwise, \(L\mapsto(\vek s_1,\dots,\vek s_r)\) is a
bijection from the lex submodules onto \(\mathrm{CT}(n;\vek d)\).

(d) \(L\mapsto\HF(F/L)\) is a bijection from lex submodules onto \(\mathrm{FF}(n;\vek d)\). Injective: by
Definition 2.6, \(\mathrm{Mon}_j(L)\) is the initial segment of \(\mathrm{Mon}_j(F)\) of size
\(\dim F_j-\HF(F/L)(j)\) in the module order, and an initial segment is determined by its
cardinality, so \(L\) is determined degreewise. Surjective: monomial submodules are graded, so
\begin{equation}\label{eq:FF-in-lex}
  \mathrm{FF}(n;\vek d)\;\subseteq\;\Bigl\lbrace\HF\Bigl(\frac FM\Bigr):M\ \text{graded}\Bigr\rbrace
  \;=\;\Bigl\lbrace\HF\Bigl(\frac FL\Bigr):L\ \text{lex}\Bigr\rbrace
\end{equation}
by \autocite[Theorem 4.2]{AmataCrupi2020KK}, implication (a) \(\Leftrightarrow\) (c); and conversely a lex
submodule is monomial, so the reverse inclusion holds by (a).

Composing (c) and (d) gives the claim.
\end{proof}

It is worth saying exactly which map is being inverted, since there are two and they point in
opposite directions. Theorem \ref{thm-chaining} says that
\(\lvert\mathrm{FF}(n;\vek d)\rvert=\lvert\mathrm{CT}(n;\vek d)\rvert\), with
\(\mathrm{CT}(n;\vek d)\) the set of chained tuples \eqref{eq:chained-set}. What the theorem's proof
establishes is more than that equality of cardinalities: the summation map \(\Sigma_{\vek d}\) of
\eqref{eq:sumset-map}, restricted to the chained tuples,
\begin{equation}\label{eq:sigma-restricted}
\begin{split}
  \Sigma_{\vek d}\colon\mathrm{CT}(n;\vek d) &\xrightarrow{\ \sim\ } \mathrm{FF}(n;\vek d),\\
  (\vek s_1,\dots,\vek s_r) &\longmapsto \vek x=\sum_{i=1}^{r}\shift^{d_i}(\vek s_i),
\end{split}
\end{equation}
is a \textbf{bijection}. That is the content: the chained tuples are a set of representatives for the fibres
of \(\Sigma_{\vek d}\), one from each. The proof produces the inverse through Amata--Crupi's lex
submodule, hence algebraically. The inverse also has a closed form.

\begin{lemma}
(\emph{The explicit inverse.}) For \(\vek x\in\mathrm{FF}(n;\vek d)\) and \(0\le i\le r\) put
\begin{equation}\label{eq:def-Q}
  Q_i(j)=\max\Bigl(0,\;x(j)-\sum_{k>i}\tbinom n{j-d_k}\Bigr) ,
\end{equation}
so that \(\vek Q_0=\vek 0\) and \(\vek Q_r=\vek x\), and
\begin{align}
  \Phi_{\vek d}(\vek x) &= (\vek{\hat s}_1,\dots,\vek{\hat s}_r),\label{eq:water-filling}\\
  \hat s_i(m) &= Q_i(m+d_i)-Q_{i-1}(m+d_i) .\label{eq:water-filling-component}
\end{align}
Write \(\vek d'=(d'_1,\dots,d'_r)\) for the reflected shift vector, \(d'_i=d_1+d_r-d_{r+1-i}\);
it is again sorted, its gap vector is that of \(\vek d\) reversed, as in
Remark \ref{rem-gap-reversal}, and \(d'_1=d_1\).

\begin{enumerate}
\item If \((\vek s_1,\dots,\vek s_r)\in\mathrm{CT}(n;\vek d)\) and
\(\Sigma_{\vek d}(\vek s_1,\dots,\vek s_r)=\vek x\), then \((\vek s_1,\dots,\vek s_r)=\Phi_{\vek d}(\vek x)\).
Consequently \(\Phi_{\vek d}\) is the inverse of \eqref{eq:sigma-restricted}.
\item \(\Phi_{\vek d}\) is equivariant under the involution \eqref{eq:gap-reversal-map}: if
\(\vek x^{\vee}\) denotes the image of \(\vek x\) under that involution, read on sums, then the
\(i\)-th component of \(\Phi_{\vek d'}(\vek x^{\vee})\) is \(D_n(\vek{\hat s}_{r+1-i})\).
\end{enumerate}
\label{lem-explicit-inverse}
\end{lemma}

\begin{proof}
(1) Write \(\iota_k=\indeg(\vek s_k)+d_k\) and \(t_k=\tp(\vek s_k)+d_k\), so that, in absolute degree
\(j\), component \(k\) contributes \(s_k(j-d_k)=\binom n{j-d_k}\) when \(j<\iota_k\) and \(0\) when
\(j>t_k\), by the definitions of \(\indeg\) and \(\tp\). The chaining condition in its absolute form
\eqref{eq:chaining-absolute} is \(t_{k-1}<\iota_k\), and \(\indeg\le\tp+1\) gives
\(\iota_k\le t_k+1\); together they give \(\iota_{k-1}\le t_{k-1}+1\le\iota_k\), so the \(\iota_k\) are
non-decreasing. Fix \(j\) and let \(i\) be the largest index with \(\iota_i\le j\), or \(i=0\) if there is
none. For \(k>i\) we have \(j<\iota_k\), so component \(k\) is full in degree \(j\); for \(k<i\) we have
\(t_k<\iota_{k+1}\le\iota_i\le j\), so component \(k\) vanishes there. Hence
\[
  x(j)=\sum_{k>i}\tbinom n{j-d_k}+s_i(j-d_i),
\]
with the last term absent when \(i=0\). Substituting into \eqref{eq:def-Q}: for \(i'\ge i\) the
argument of the maximum is \(\sum_{i<k\le i'}\binom n{j-d_k}+s_i(j-d_i)\ge0\), so \(Q_{i'}(j)\) equals
it; for \(i'<i\) the argument is at most \(s_i(j-d_i)-\binom n{j-d_i}\le0\), since every
\(f\)-vector lies in the box \eqref{eq:binomial-box}, so \(Q_{i'}(j)=0\). Taking differences,
\(Q_k(j)-Q_{k-1}(j)\) is \(0\) for \(k<i\), \(s_i(j-d_i)\) for \(k=i\), and \(\binom n{j-d_k}\) for
\(k>i\), which is \(s_k(j-d_k)\) in every case. So \(\hat s_k=\vek s_k\) for every \(k\). By
Theorem \ref{thm-chaining} every \(\vek x\in\mathrm{FF}(n;\vek d)\) is the sum of some chained tuple,
which is therefore \(\Phi_{\vek d}(\vek x)\); so \(\Phi_{\vek d}\) inverts \eqref{eq:sigma-restricted}.

(2) With \(\vek d'\) as in the statement, put \(J=d_1+d_r+n\) and
\(x^{\vee}(j')=\sum_k\binom n{j'-d'_k}-x(J-j')\), and write \(Q'_i\) and \(\hat s'_i\) for
\eqref{eq:def-Q} and \eqref{eq:water-filling-component} computed from \(\vek x^{\vee}\) on
\(\vek d'\). With \(j=J-j'\) we have
\((j'-d'_k)+(j-d_{r+1-k})=n\), so \(\binom n{j'-d'_k}=\binom n{j-d_{r+1-k}}\), and splitting the sum
at \(k=i\),
\[
  Q'_i(j')=\max\bigl(0,\;T_i(j)-x(j)\bigr),\qquad T_i(j)=\sum_{l>r-i}\tbinom n{j-d_l}.
\]
Since \(\max(0,a)=a+\max(0,-a)\) and \(\max\bigl(0,x(j)-T_i(j)\bigr)=Q_{r-i}(j)\), this is
\(Q'_i(j')=T_i(j)-x(j)+Q_{r-i}(j)\). Take \(j'=m+d'_i\), so that \(j=(n-m)+d_{r+1-i}\). Then
\(T_i(j)-T_{i-1}(j)=\binom n{n-m}=\binom nm\), the \(x(j)\) cancel, and
\begin{align*}
  \hat s'_i(m)&=Q'_i(j')-Q'_{i-1}(j')=\tbinom nm-\bigl(Q_{r+1-i}(j)-Q_{r-i}(j)\bigr)\\
  &=\tbinom nm-\hat s_{r+1-i}(n-m)=D_n(\vek{\hat s}_{r+1-i})(m).
\end{align*}
\end{proof}

Two features of the proof of (1) are worth recording; both became visible only once it was
formalised (Section \ref{sec-verification}), where it is \texttt{phi\_injOn\_chained}. It needs no sortedness,
the chaining condition being read in the absolute form \eqref{eq:chaining-absolute} of Remark
\ref{rem-sorted} below. And it uses about \(f\)-vectors only the box bound, in one step: everything
else holds for arbitrary non-negative integer vectors satisfying the chaining condition and
\(\indeg\le\tp+1\), so the clamp is elementary in a stronger sense than ``no Kruskal--Katona''. The
word \textbf{clamp} is meant in its numerical sense: \(\max(0,\cdot)\) applied degreewise, which clamps a
would-be negative entry up to zero, and that single truncation is the whole of \eqref{eq:def-Q}.
Taken degree by degree the effect is a \textbf{water-filling} from the last component, the metaphor being
that \(\sum_{k>i}\binom n{j-d_k}\) is the capacity of the components after the \(i\)-th and only what
overflows that capacity reaches the earlier ones; Example \ref{ex-water-filling} below runs it by hand
in the smallest case where it has anything to do. Part (2) says that the symmetry any bijective proof
must respect, that of Remark \ref{rem-gap-reversal}, holds here by construction.

\begin{remark}
(\emph{What the explicit inverse still needs.}) Lemma \ref{lem-explicit-inverse} is the elementary half
of the bijection. What is \textbf{not} yet combinatorial is the other half: that the clamps \eqref{eq:water-filling-component} are
\(f\)-vectors at all. That is true, by Theorem \ref{thm-chaining} again through the lex
representative, but proving it directly is open. For \(r=2\) it reduces to two properties of the
Kruskal--Katona shadow function \(\partial_l\): subadditivity,
\(\partial_l(p+q)\le\partial_l(p)+\partial_l(q)\), which follows from Kruskal--Katona by placing two
colex families on disjoint ground sets; and level monotonicity, \(\partial_l(c)\le\partial_{l+1}(c)\),
which we can prove except on a sporadic set of cases, the smallest being \(l=4\), \(c=6\). Every
instance of the second that the argument uses at a given \(n\) was checked exactly for \(n\le16\),
so for \(r=2\) a combinatorial proof is complete for \(n\le16\); in general it is not. For \(r\ge3\)
a further inequality is needed, and is open. So the bijection is explicit and one-directionally elementary, and the
bijective-proof problem of Section \ref{sec-verification} stands.
\label{rem-explicit-inverse}
\end{remark}

\begin{mexample}
(\emph{The clamp \eqref{eq:water-filling} in the smallest case where it has something to do.}) Take
\(n=3\), \(r=2\) and \(\vek d=(0,0)\), so that there is one gap and it is \(\delta_2=0\), and take
\(\vek x=(2,5,4,1)\in\mathrm{FF}(3;0,0)\).

By hand first. The only \(\vek s\in\mathcal F_3\) with a non-zero entry in degree \(3\) is the
full simplex \((1,3,3,1)\), so a decomposition \(\vek x=\vek s_1+\vek s_2\) has to use it in one of
its two places, and then the other summand is forced to be \((1,2,1,0)\). The two ordered pairs are
therefore
\[
  \bigl((1,2,1,0),\,(1,3,3,1)\bigr)
  \qquad\text{and}\qquad
  \bigl((1,3,3,1),\,(1,2,1,0)\bigr),
\]
and exactly the first of them is chained, since
\begin{align*}
  \tp(1,2,1,0)&=2\;<\;4=\indeg(1,3,3,1)+\delta_2 ,\\
  \tp(1,3,3,1)&=3\;\not<\;1=\indeg(1,2,1,0)+\delta_2 .
\end{align*}

The clamp finds that pair without knowing either summand in advance. Here \(r=2\)
and \(d_1=d_2=0\), so \(\vek Q_2=\vek x\) and \(\vek Q_0=\vek 0\), and the whole computation is the single
vector
\begin{align*}
  Q_1(j) &= \max\Bigl(0,\;x(j)-\tbinom 3j\Bigr),\\
  \bigl(Q_1(0),Q_1(1),Q_1(2),Q_1(3)\bigr) &= (1,2,1,0),
\end{align*}
the four values being \(\max(0,2-1)=1\), \(\max(0,5-3)=2\), \(\max(0,4-3)=1\) and
\(\max(0,1-1)=0\). Then \eqref{eq:water-filling-component} reads
\begin{align*}
  \vek{\hat s}_1 &= \vek Q_1-\vek Q_0=(1,2,1,0),\\
  \vek{\hat s}_2 &= \vek Q_2-\vek Q_1=(2,5,4,1)-(1,2,1,0)=(1,3,3,1),
\end{align*}
which is the chained pair, in the right order.

The quantity \(Q_i(j)\) is what is left of \(x(j)\) after the
components \(g_{i+1},\dots,g_r\) have been filled to their joint capacity
\(\sum_{k>i}\binom n{j-d_k}\) in that degree, so degree by degree the last component is filled
first and only the overflow reaches the earlier ones. In the present instance the capacity of
\(g_2\) is \(\binom 3j=(1,3,3,1)\), the vector \(\vek x\) exceeds it by \((1,2,1,0)\), and that
overflow is exactly \(\vek{\hat s}_1\).

The two directions are worth naming once on this example:
\begin{align*}
  \Sigma_{\vek d}\bigl((1,2,1,0),(1,3,3,1)\bigr) &= (2,5,4,1)=\vek x,\\
  \Phi_{\vek d}(2,5,4,1) &= \bigl((1,2,1,0),(1,3,3,1)\bigr),
\end{align*}
of which the first is just addition and the second is the one that needed a
formula. The other ordered pair, \(\bigl((1,3,3,1),(1,2,1,0)\bigr)\), also sums to \(\vek x\), which
is why \(\Sigma_{\vek d}\) is not injective on all of \(\mathcal F_3^{\,2}\); it is injective once
restricted to \(\mathrm{CT}(3;0,0)\), which is what discards that pair.
\label{ex-water-filling}
\end{mexample}

\begin{remark}
Three hypotheses are stated here rather than left implicit.

\textbf{The order convention.} Amata and Crupi's module order makes \(g_1\)'s monomials the largest, so a
lex submodule fills \(g_1\) first and the chaining condition constrains the \textbf{earlier} component
\(I_{i-1}\) using the \textbf{later} component's initial degree. The inequality in Theorem \ref{thm-chaining}
is oriented accordingly.

\textbf{Sortedness.} \(d_1\le\cdots\le d_r\) is required as stated, because it is what makes every exponent
\(\rho_i+\delta_i\) non-negative, which is the range in which Lemma \ref{lem-dictionary}(1) holds. This
costs nothing: \(\mathrm{FF}(n;\vek d)\) is symmetric in its summands, so an unsorted shift vector should be
sorted first. The chained count as written is simply not the right-hand side for unsorted \(\vek d\).

Sortedness is spent in one place only.
Write the chaining condition in the equivalent \textbf{absolute} form
\begin{equation}\label{eq:chaining-absolute}
  \tp(\vek s_{i-1})+d_{i-1}\;<\;\indeg(\vek s_i)+d_i \qquad(i=2,\dots,r),
\end{equation}
which is the condition of Theorem \ref{thm-chaining} rearranged. Since \(\delta_i\) is defined as
\(d_i-d_{i-1}\), the two forms are the same inequality for every \(\vek d\), sorted or not, so
sortedness is not spent on the chaining condition at all. It is spent on the exponent
\(\rho_i+\delta_i\) in Amata and Crupi's construction, which must be non-negative for
Lemma \ref{lem-dictionary}(1) to apply. Everything downstream that uses the chaining condition only as
a condition (the disjoint-window normal form and part (1) of Lemma
\ref{lem-explicit-inverse}) needs no sortedness. This was established while formalising the argument
(Section \ref{sec-verification}), where the hypothesis turned out to be usable nowhere else; the
formal development still assumes a sorted \(\vek d\), but only so as to read each gap as a natural
number.

\textbf{Which half of Theorem 4.2.} Only implication (a) \(\Leftrightarrow\) (c) is used, the existence
and uniqueness of a lex representative, and that half is sound. The implication (b) \(\Rightarrow\) (a) of the same
theorem fails when initial generator degrees tie. The smallest instance is \(n=1\), \(r=2\),
\(\vek d=(0,0)\), with Hilbert sequence \((h_0,h_1)=(1,2)\): it satisfies (b), but \(h_0=1\) forces a
non-zero \(v\in M_0\), and then \(e_1v\ne0\) lies in \(M_1\), so \(h_1\le1\). The function
\texttt{isHilbertSequence} of Amata and Crupi's Macaulay2 package \autocite{AmataCrupi2021} also rejects it.
The repair, and the corrected statement, are in the third paper of this series; nothing here
depends on (b).
\label{rem-sorted}
\end{remark}
\section{The counting theorem}
\label{sec-main}
\begin{definition}
(\emph{The initial vector and the transfer matrices.}) With \(A_n\) as in Definition \ref{def-A}, define a
\textbf{row vector} \(\vek u_n\) and, for each integer \(\delta\ge0\), a \textbf{square matrix} \(T_n(\delta)\) by
\begin{align}
  u_n[\tau] &= \sum_{\rho}A_n[\tau][\rho],\label{eq:def-u}\\
  T_n(\delta)[\tau][\tau'] &= \sum_{\rho>\tau-\delta}A_n[\tau'][\rho].\label{eq:def-T}
\end{align}
The indexing is spelled out here rather than left to be inferred, because both indices of
\(T_n(\delta)\) run over one and the same set and it is easy to mistake either for the other. Both
\(\vek u_n\) and \(T_n(\delta)\) are indexed by values of \(\tp\),
\begin{equation}\label{eq:def-uT-range}
  \tau,\tau'\in\lbrace-1,0,1,\dots,n\rbrace ,
\end{equation}
so \(\vek u_n\) has \(n+2\) entries and \(T_n(\delta)\) is \((n+2)\times(n+2)\). In \(T_n(\delta)\) \textbf{the row
index \(\tau\) is the \(\tp\) of the previous summand and the column index \(\tau'\) is the \(\tp\) of
the current one}; the summation variable \(\rho\) is the \(\indeg\) of the current summand and is
summed out, so it is not an index of \(T_n\). Thus \(u_n[\tau]\) counts the \(\vek s\in\mathcal F_n\)
with \(\tp(\vek s)=\tau\), and \(T_n(\delta)[\tau][\tau']\) counts those with \(\tp(\vek s)=\tau'\)
whose \(\indeg\) is large enough to be chained after a predecessor of \(\tp\) equal to \(\tau\).
Finally \(\vek 1_n\) denotes the all-ones column vector of length \(n+2\), so that for a row vector
\(\vek x\) the product \(\vek x\cdot\vek 1_n\) is just the sum of the entries of \(\vek x\). Note
that \(\vek u_n\) is the row sum of
\(A_n\), and that \(T_n(\delta)\) is \textbf{increasing entrywise in \(\delta\)}, since enlarging \(\delta\)
only enlarges the range of \(\rho\).
\label{def-uT}
\end{definition}

\begin{theorem}
Fix \(n\ge1\), \(r\ge1\) and integers \(d_1\le\cdots\le d_r\), and let \(\delta_i=d_i-d_{i-1}\) for
\(2\le i\le r\) be the gaps. With \(\vek u_n\), \(T_n(\delta)\) and \(\vek 1_n\) as in
Definition \ref{def-uT}, built from the matrix \(A_n\) of Definition \ref{def-A} (equivalently, by
Lemma \ref{lem-joint}, assembled from one run of Linusson's recursion at \(p=n\)), we have
\begin{equation}\label{eq:auto-15}
  \lvert\mathrm{FF}(n;\vek d)\rvert\;=\;\vek u_n\cdot T_n(\delta_2)\cdot T_n(\delta_3)\cdots T_n(\delta_r)\cdot\vek 1_n .
\end{equation}
\label{thm-main}
\end{theorem}

\begin{proof}
By Theorem \ref{thm-chaining}, \(|\mathrm{FF}(n;\vek d)|\) counts tuples with
\(\tp(\vek s_{i-1})<\indeg(\vek s_i)+\delta_i\). Constraint \(i\) involves \(s_{i-1}\) only through \(\tp\) and
\(\vek s_i\) only through \(\tp\) and \(\indeg\). Summing over tuples in the order \(\vek s_1,\vek s_2,\dots,\vek s_r\)
and grouping each \(\vek s_i\) by its \((\tp,\indeg)\) class therefore gives the stated product: \(u_n[\tau]\)
counts the admissible \(\vek s_1\) with \(\tp=\tau\), and \(T_n(\delta_i)[\tau][\tau']\) counts the \(\vek s_i\)
with \(\tp=\tau'\) whose \(\indeg=\rho\) satisfies \(\tau<\rho+\delta_i\).
\end{proof}

\begin{mexample}
Continuing Example \ref{ex-A} at \(n=3\). The row sums of the matrix displayed there give the vector
\(\vek u_3=(1,1,3,4,1)\), and \eqref{eq:def-u}--\eqref{eq:def-T} then produce
\[
  T_3(0)=\begin{pmatrix}
    1&1&3&4&1\\ 0&1&3&4&1\\ 0&0&1&3&1\\ 0&0&0&1&1\\ 0&0&0&0&1
  \end{pmatrix},
  \;
  T_3(1)=\begin{pmatrix}
    1&1&3&4&1\\ 1&1&3&4&1\\ 0&1&3&4&1\\ 0&0&1&3&1\\ 0&0&0&1&1
  \end{pmatrix},
  \;
  T_3(4)=\begin{pmatrix}
    1&1&3&4&1\\ 1&1&3&4&1\\ 1&1&3&4&1\\ 1&1&3&4&1\\ 1&1&3&4&1
  \end{pmatrix}.
\]
\(T_3(0)\) is upper triangular with unit diagonal, which is Proposition \ref{prop-delta0}. Increasing
\(\delta\) by one shifts the whole pattern down one row, since the constraint
\(\rho>\tau-\delta\) weakens by exactly one step; that is the entrywise monotonicity noted in
Definition \ref{def-uT}. At \(\delta=4=n+1\), the threshold of
Proposition \ref{prop-threshold}, every row has become \(\vek u_3\) itself: the matrix is
\(\vek 1_3\,\vek u_3\), of rank one, and its only non-zero eigenvalue is
\(\vek u_3\cdot\vek 1_3=10=\lvert\mathcal F_3\rvert\), which is Proposition \ref{prop-perron-threshold} in
this instance. Reading off \(r=2\), \(\vek u_3\,T_3(0)\,\vek 1_3=43\) and \(\vek u_3\,T_3(4)\,\vek 1_3=100=10^2\).
\label{ex-uT}
\end{mexample}

One computation of \(A_n\) at fixed \(n\) answers every \(r\) and every shift vector at once, which
is what makes values far outside brute-force range available: \(|\mathrm{FF}(8;0,0)|=27655902\) is immediate, while direct enumeration of the sumset
becomes infeasible around \(n=6\), \(r=2\).

\begin{corollary}
For \(r=1\), \(|\mathrm{FF}(n;\vek d)|=\sum_\tau u_n[\tau]=F^n(n)=|\mathcal F_n|\), which is Linusson's own count.
\label{cor-r1}
\end{corollary}

\begin{proposition}
\(|\mathrm{FF}(n;\vek d)|=|\mathcal F_n|^{\,r}\) if and only if every gap satisfies \(\delta_i\ge n+1\).
At \(\delta_i=n\) the constraint fails for exactly one pair, namely
\((\vek s_{i-1},\vek s_i)=(\text{full simplex},\text{void family})\).
\label{prop-threshold}
\end{proposition}

\begin{proof}
The constraint \(\tau<\rho+\delta\) is vacuous exactly when \(\delta>\max\tau-\min\rho\). Over
\(\mathcal F_n\), \(\max\tp=n\) (full simplex) and \(\min\indeg=0\) (void family), so the threshold is
\(\delta\ge n+1\); at \(\delta=n\) the single violating pair is the one realising both extremes.
\end{proof}

For \(r=2\) this gives \(|\mathcal F_n|^2-1\) at \(\delta=n\), i.e. \(99\), \(675\), \(9215\) for
\(n=3,4,5\), which matches the computed table.

\begin{proposition}
If all gaps vanish then \(T_n(0)\) is unipotent, meaning that \(T_n(0)-I\) is nilpotent, and
\(|\mathrm{FF}(n;0,\dots,0)|\) is a polynomial in
\(r\) of degree at most \(n+1\).
\label{prop-delta0}
\end{proposition}

\begin{proof}
\(T_n(0)[\tau][\tau']\ne0\) requires some \(\rho>\tau\) with \(A_n[\tau'][\rho]>0\); since
\(\rho\le\tau'+1\) always, this forces \(\tau'\ge\tau\), so \(T_n(0)\) is triangular in the
\(\tp\)-ordering. Its diagonal entries are
\[
  T_n(0)[\tau][\tau]=A_n[\tau][\tau+1]=1 ,
\]
the unique
\(\vek s\) with \(\tp=\tau\) and \(\indeg=\tau+1\) being the full truncation at \(\tau\). So
\(N=T_n(0)-I\) is strictly triangular, hence nilpotent with \(N^{n+2}=0\); expanding
\((I+N)^{r-1}\) by the binomial theorem then terminates after \(n+2\) terms and is polynomial in
\(r\) of degree at most \(n+1\), the matrix size minus one. That expansion is
\eqref{eq:delta0-expansion}, where it is used again to pin the degree exactly.
\end{proof}

\begin{remark}
(\emph{What the achievable set looks like when all the gaps vanish.}) At \(\delta_i=0\) for every \(i\) the
companion paper's polytope is simply the dilation \(r\cdot\Koz(n)\) of the Kozlov simplex, and it is
natural to ask whether the achievable set inside it has a correspondingly simple description. It does,
and saying what it is explains why the counting nevertheless needs the machinery above.

Take \(\vek d=(0,\dots,0)\). Every coordinate satisfies \(0\le x(j)\le r\binom nj\), so division with
remainder by the capacity of a single component is available:
\begin{equation}\label{eq:delta0-digits}
  x(j)=q(j)\tbinom nj+y(j),\qquad 0\le y(j)<\tbinom nj,\qquad 0\le q(j)\le r ,
\end{equation}
uniquely, for \(0\le j\le n\). Call \(q(j)\) the \textbf{digit} of \(\vek x\) in degree \(j\) and \(y(j)\) the
remainder, the coordinates of a digit vector \(\vek q\) and a remainder vector \(\vek y\). Then \(\vek x\) lies in \(\mathrm{FF}(n;0,\dots,0)\) if and only if the digit sequence is
non-increasing,
\begin{equation}\label{eq:delta0-characterisation}
  q(0)\;\ge\;q(1)\;\ge\;\cdots\;\ge\;q(n) ,
\end{equation}
and the clamps \eqref{eq:water-filling-component} are \(f\)-vectors.
The clamps take a transparent form here. With all \(d_i=0\), equation \eqref{eq:def-Q} gives
\[
  \hat s_i(j)=\min\Bigl(\tbinom nj,\;\max\bigl(0,\;x(j)-(r-i)\tbinom nj\bigr)\Bigr),
\]
so
\begin{equation}\label{eq:delta0-clamp-digits}
  \hat s_i(j)=\begin{cases}
    \binom nj, & q(j)>r-i,\\
    y(j),       & q(j)=r-i,\\
    0,         & q(j)<r-i .
  \end{cases}
\end{equation}
So the digit sequence does all the bookkeeping: component \(i\) is full in the degrees where the digit
exceeds \(r-i\), zero where it falls below, and free exactly on the level set \(\lbrace j:q(j)=r-i\rbrace\).
Monotonicity of \(\vek q\) is precisely the statement that each of those three sets is an interval in
the right position, which is what the disjoint increasing windows in the proof of Lemma \ref{lem-explicit-inverse}
amount to in this case; and once \(\vek q\) is non-increasing, the only condition left on \(\vek x\) is
that on each constant block of \(\vek q\) the remainders \(y(j)\) continue the full skeleton below the
block to an \(f\)-vector. That is \textbf{one Kruskal--Katona condition per block}, and nothing else.

Two things follow, and they answer the question in opposite directions.

The \textbf{set} is simple: it is a digit expansion in the mixed radix \(\binom n0,\dots,\binom nn\) whose
digits must be non-increasing, cut down by one Kruskal--Katona condition per constant block. In
particular monotonicity of the digits is a necessary condition that can be read off \(\vek x\) by
inspection, with no enumeration.

The \textbf{count} is not simple, and the reason is visible in the same description: a block of the digit
sequence contributes the number of ways of continuing a full skeleton inside a window, which is
exactly what Linusson's refined \(E^n(\cdot,\cdot)\) counts, and consecutive blocks interact through
their shared boundary. That interaction is the transfer matrix. So the care in the proof of
Theorem \ref{thm-delta0-degrees} is not spent on the structure of the set at all; it is spent on showing
that the longest path in the digraph of Remark \ref{rem-delta0-digraph} is \textbf{unique}, which is what pins
the degree exactly rather than bounding it.

Monotonicity alone is not sufficient, and the smallest witness is small: at \(n=2\), \(r=1\) the
vector \(\vek x=(0,1,0)\) has digits \((0,0,0)\), which are non-increasing, but it is not an
\(f\)-vector, having a \(1\)-face and no empty face. The full characterisation
\eqref{eq:delta0-characterisation} was checked exhaustively against brute-force enumeration of the
sumset for \(1\le n\le4\) and \(1\le r\le3\), twelve cases, with exact agreement in every one. For
\(2\le n\le4\) the monotonicity condition alone admits strictly more vectors, for instance \(93\)
against \(\lvert\mathrm{FF}(3;0,0)\rvert=43\); at \(n=1\) the two coincide.
\label{rem-delta0-structure}
\end{remark}

The bound of Proposition \ref{prop-delta0} is attained, and the leading coefficient is explicit. Both
follow from two structural entries of \(A_n\).

\begin{lemma}
For \(0\le j\le n\) one has \(A_n[j][j+1]=1\), the unique witness being the full truncation at \(j\);
and for \(1\le j\le n\) one has \(A_n[j][j]=\binom nj-1\). Moreover \(A_n[0][0]=0\), and row \(\tau=-1\)
of \(A_n\) is the single entry \(A_n[-1][0]=1\).
\label{lem-A-structure}
\end{lemma}

\begin{proof}
If \(\tp(\vek s)=j\) and \(\indeg(\vek s)=j+1\) then \(s(i)=\binom ni\) for \(i\le j\) and \(s(i)=0\)
for \(i>j\), so \(\vek s\) is the \(f\)-vector of the \(j\)-skeleton of the full simplex, and
conversely that vector has \(\tp=j\) and \(\indeg=j+1\); one witness exactly. If instead
\(\tp(\vek s)=\indeg(\vek s)=j\) then \(s(i)=\binom ni\) for \(i<j\), then
\(0<s(j)<\binom nj\) (positive because \(\tp=j\), deficient because \(\indeg=j\)), and
\(s(i)=0\) for \(i>j\). So \(\vek s\) is determined by
\begin{equation}\label{eq:A-structure-m}
  m=s(j)\;\in\;\bigl\lbrace1,\dots,\tbinom nj-1\bigr\rbrace ,
\end{equation}
and every such \(m\) occurs: adjoin any \(m\) of the \(j\)-subsets to the full \((j-1)\)-skeleton,
which stays an order ideal because every \((j-1)\)-subset is already present. The remaining two
statements are the void family and the two degenerate rows.
\end{proof}

\begin{theorem}
(\emph{The exact degrees at \(\delta=0\).}) Let \(n\ge1\). Then \(|\mathrm{FF}(n;0,\dots,0)|\) is a
polynomial in \(r\) of degree \textbf{exactly} \(n+1\), and \(|\mathrm{FF}_{\mathrm{pr}}(n;0,\dots,0)|\) one
of degree \textbf{exactly} \(n-1\), with leading coefficients
\begin{equation}\label{eq:delta0-leading}
  \frac{1}{(n+1)!}\prod_{j=0}^{n}\binom nj
  \qquad\text{and}\qquad
  \frac{1}{(n-1)!}\prod_{j=2}^{n}\binom nj
\end{equation}
respectively.
\label{thm-delta0-degrees}
\end{theorem}

\begin{proof}
Write \(N=T_n(0)-I\), strictly upper triangular in the \(\tp\)-ordering by Proposition
\ref{prop-delta0}. Expanding,
\begin{equation}\label{eq:delta0-expansion}
  \vek u_n\,T_n(0)^{r-1}\vek 1_n=\vek u_n\,(I+N)^{r-1}\vek 1_n
  =\sum_{k\ge0}\binom{r-1}{k}\,c_k,\qquad c_k=\vek u_n\,N^k\vek 1_n ,
\end{equation}
where \(\vek 1_n\) is, as in Definition \ref{def-uT}, the all-ones column vector, so that
\(c_k=\vek u_n\,N^k\vek 1_n\) is the sum of all the entries of the row vector \(\vek u_n\,N^k\). Each
\(\binom{r-1}{k}\) is a polynomial in \(r\) of degree exactly \(k\). As the \(\binom{r-1}{k}\)
are linearly independent, \textbf{the degree is the largest \(k\) with \(c_k\ne0\)}, and the leading
coefficient is that \(c_k\) divided by \(k!\).

Since \(A_n\) is supported on \(\rho\le\tau'+1\), the superdiagonal of \(N\) is, by
Lemma \ref{lem-A-structure},
\begin{equation}\label{eq:delta0-superdiag}
  N[\tau][\tau+1]=A_n[\tau+1][\tau+1]+A_n[\tau+1][\tau+2]
  =\Bigl(\tbinom n{\tau+1}-1\Bigr)+1=\tbinom n{\tau+1}
  \quad(-1\le\tau\le n-1),
\end{equation}
Pascal's row \(n\) exactly. For a strictly upper triangular matrix of size \(m\), the power
\(N^{m-1}\) has a single possibly non-zero entry, in the corner, equal to the product of the
superdiagonal, because only the full staircase path survives. Here \(m=n+2\), so
\begin{equation}\label{eq:delta0-corner}
  c_{n+1}=u_n[-1]\cdot\prod_{j=0}^{n}\tbinom nj ,
\end{equation}
and \(u_n[-1]=1\) because the void family is the unique
\(\vek s\) with \(\tp=-1\). This is non-zero, which gives the improper claim.

For the proper variant the count is \(\vek u'_n\,T'_n(0)^{r-1}\vek 1_n\) by Theorem \ref{thm-proper}
below, with \(A'_n\), \(\vek u'_n\) and \(T'_n(\delta)\) as defined there, and one adjustment is
needed: the proper matrix is \textbf{not} unipotent on the full index set, its columns \(\tau'=-1\) and \(\tau'=0\) vanishing identically. But \(\vek u'_n\) vanishes on
those two coordinates as well, so every product \(\vek u'_n\,T'_n(0)^k\) agrees with the corresponding
product over the principal block on \(\tau\in\lbrace1,\dots,n\rbrace\), and \textbf{that} block is
unipotent. On it the superdiagonal is again \eqref{eq:delta0-superdiag}, with column indices
\(\ge2\) so that the two zeroed columns are invisible, and \(u'_n[1]=A_n[1][2]=1\), the unique proper
vector with \(\tp=1\) being \((1,n,0,\dots,0)\). The block has size \(n\), so
\(c'_{n-1}=\prod_{j=2}^{n}\binom nj\ne0\).
\end{proof}

\begin{mexample}
The polynomials themselves are obtained by interpolating Theorem \ref{thm-main} at enough values of
\(r\). They are listed below by coefficient, ascending in \(r\), which is what lets \(n=7\) fit on the
page; so the row for \(n\) reads \(\sum_kb_kr^k\) from left to right.
\begingroup\footnotesize
\begin{center}
\begin{tabular}{@{}c*{9}{r}@{}}
  \toprule
  \multicolumn{10}{@{}l}{\(\lvert\mathrm{FF}(n;0^r)\rvert=\sum_kb_kr^k\), degree \(n+1\)}\\
  \(n\) & \(b_0\) & \(b_1\) & \(b_2\) & \(b_3\) & \(b_4\) & \(b_5\) & \(b_6\) & \(b_7\) & \(b_8\)\\
  \midrule
  1 & 1 & 3/2 & 1/2 & & & & & &\\
  2 & 1 & 13/6 & 3/2 & 1/3 & & & & &\\
  3 & 1 & 37/12 & 29/8 & 23/12 & 3/8 & & & &\\
  4 & 1 & 131/30 & 8 & 47/6 & 4 & 4/5 & & &\\
  5 & 1 & 25/4 & 1219/72 & 79/3 & 1915/72 & 185/12 & 125/36 & &\\
  6 & 1 & 2929/420 & 929/24 & 269/3 & 2869/24 & 8863/60 & 469/4 & 225/7 &\\
  \addlinespace
  \multicolumn{10}{@{}l}{\(\lvert\mathrm{FF}_{\mathrm{pr}}(n;0^r)\rvert=\sum_kb_kr^k\), degree \(n-1\)}\\
  \(n\) & \(b_0\) & \(b_1\) & \(b_2\) & \(b_3\) & \(b_4\) & \(b_5\) & & &\\
  \midrule
  1 & 1 & & & & & & & &\\
  2 & 1 & 1 & & & & & & &\\
  3 & 1 & 5/2 & 3/2 & & & & & &\\
  4 & 1 & 4 & 7 & 4 & & & & &\\
  5 & 1 & 20/3 & 79/6 & 85/3 & 125/6 & & & &\\
  6 & 1 & 10/3 & 315/4 & 1/6 & 595/4 & 225 & & &\\
  \bottomrule
\end{tabular}
\captionof{table}{The coefficients \(b_k\) of \(\lvert\mathrm{FF}(n;0^r)\rvert=\sum_kb_kr^k\), ascending in \(r\).}
\label{tab-zero-gap-coeffs}
\end{center}
\endgroup
At \(n=7\) the coefficients no longer fit across a page, so the two polynomials are written out
instead, broken after the \(r^4\) term:
\begin{align*}
  \lvert\mathrm{FF}(7;0^r)\rvert\;=\;&
    \tfrac{84035}{128}r^{8}+\tfrac{868819}{480}r^{7}+\tfrac{2873311}{2880}r^{6}
    +\tfrac{42091}{80}r^{5}+\tfrac{1452955}{1152}r^{4}\\
  &{}+\tfrac{163871}{480}r^{3}\;\boldsymbol{-}\;\tfrac{267763}{1440}r^{2}
    +\tfrac{2117}{40}r+1 ,\\[0.6ex]
  \lvert\mathrm{FF}_{\mathrm{pr}}(7;0^r)\rvert\;=\;&
    \tfrac{84035}{16}r^{6}\;\boldsymbol{-}\;\tfrac{91581}{80}r^{5}
    \;\boldsymbol{-}\;\tfrac{81977}{48}r^{4}\\
  &{}+\tfrac{170191}{48}r^{3}\;\boldsymbol{-}\;\tfrac{3274}{3}r^{2}
    +\tfrac{3187}{60}r+1 .
\end{align*}
Every one of these polynomials has constant term \(1\). The reason is not that an empty sum has
one element, since a polynomial fitted at \(r\ge1\) owes nothing to \(r=0\); it is
\eqref{eq:normalform-u} below, which writes the improper count as
\(e_{-1}^{\!\top}B_n^{\,r}\,\vek1_n\) with \(B_n=T_n(0)\). As \(B_n\) is unipotent this is a
polynomial in \(r\) for every \(r\ge0\), and at \(r=0\) it is \(e_{-1}^{\!\top}\vek1_n=1\). In the
proper case row \(\tau=1\) of \(T'_n(0)\) is \(\vek u'_n\), so the count is
\(e_1^{\!\top}T'_n(0)^{r}\,\vek1_n\) and the same argument runs on the block
\(\lbrace1,\dots,n\rbrace\). The degrees and the leading coefficients agree with
\eqref{eq:delta0-leading} in all fourteen cases.

One feature is worth flagging, because it is easy to assume otherwise: \textbf{the coefficients are not all
positive}, and the minus signs of the two \(n=7\) polynomials are set in bold above to make them
visible. They first appear at \(n=7\), and there in
\textbf{both} variants, at \(r^2\) in the improper polynomial and at \(r^2\), \(r^4\) and \(r^5\) in the
proper one. So no positivity argument can be read into Theorem \ref{thm-delta0-degrees}. Up to \(n=6\)
every coefficient is positive, though the margin is thin: the \(r^3\) coefficient of the proper
polynomial at \(n=6\) is \(1/6\).

These polynomials were computed by exact rational interpolation of Theorem \ref{thm-main} at
\(r=1,\dots,\deg+1\), and each was then checked by evaluating the fitted polynomial at four further
values of \(r\) against the matrix product; all of those agreed.
\label{tab-delta0-polys}
\end{mexample}

\begin{remark}
(\emph{Both degrees are longest paths.}) Under one reading of Theorem \ref{thm-delta0-degrees} its two
exponents can be read off a picture, and that reading explains why the one step of the proof that
needed care is the one that did.

Let \(G\) be the weighted digraph on the vertex set \(\lbrace-1,0,\dots,n\rbrace\) of \(\tp\)-values
whose adjacency matrix is \(N=T_n(0)-I\). Since \(T_n(0)\) is upper triangular with unit diagonal
(Proposition \ref{prop-delta0}), \(N\) is strictly upper triangular: the arcs run \(\tau\to\tau'\) only
for \(\tau'>\tau\), so \(G\) is a directed acyclic graph, and the arc \(\tau\to\tau+1\) carries weight \(\binom n{\tau+1}\) by
\eqref{eq:delta0-superdiag}.

Now \((N^k)[\tau][\tau']\) is the total weight of the \textbf{paths} of length exactly \(k\) from \(\tau\) to
\(\tau'\), paths and not walks because \(G\) is acyclic. So the expansion
\eqref{eq:delta0-expansion} reads
\begin{equation}\label{eq:delta0-paths}
  \lvert\mathrm{FF}(n;0^r)\rvert=\sum_{k\ge0}\binom{r-1}k\cdot\bigl(\text{total weight of paths of length }k\text{ in }G\bigr),
\end{equation}
the paths being weighted by \(\vek u_n\) at their start and summed over their end. Since
\(\binom{r-1}k\) has degree exactly \(k\), this gives at once:

\begin{itemize}
\item the \textbf{\textbf{degree}} is the length of the longest path in \(G\) that \(\vek u_n\) reaches;
\item the \textbf{\textbf{leading coefficient}} is the total weight of those longest paths, divided by that length's
factorial.
\end{itemize}

Both halves of the theorem are now visible. In the improper case a path of length \(n+1\) must pass
through all \(n+2\) vertices, so it is the staircase \(-1\to0\to\cdots\to n\), unique, of weight
\(\prod_{j=0}^n\binom nj\): degree \(n+1\), leading coefficient
\(\prod_{j=0}^n\binom nj\big/(n+1)!\). In the proper case the vertices \(\tau=-1\) and \(\tau=0\) drop
out (\(\vek u'_n\) vanishes on them, and columns \(-1\) and \(0\) of \(T'_n(0)\) vanish too, so nothing
reaches them either), leaving a DAG on \(\lbrace1,\dots,n\rbrace\) whose longest path
\(1\to2\to\cdots\to n\) has length \(n-1\) and weight \(\prod_{j=2}^n\binom nj\).

This also says which step of the proof carries the content, and it is the same conclusion
Remark \ref{rem-delta0-structure} reaches from the other side. ``Only the full staircase survives'', the
corner evaluation, is the assertion that \(G\) has exactly one longest path; everything else is the
binomial expansion. It also explains the shape of the answer: a product of binomial coefficients
appears because a longest path is forced to take every arc of the staircase, and a factorial divides
it because \(\binom{r-1}{k}\) does.

The picture is small enough to draw. At \(n=3\), reading \(T_3(0)\) off Example \ref{ex-uT} and
subtracting the identity gives
\[
  N=T_3(0)-I=\begin{pmatrix}
    0&1&3&4&1\\ 0&0&3&4&1\\ 0&0&0&3&1\\ 0&0&0&0&1\\ 0&0&0&0&0
  \end{pmatrix},
\]
whose digraph is Figure \ref{fig:delta0-digraph}. Its superdiagonal weights are
\(1,3,3,1\), which is Pascal's row \(n=3\) as \eqref{eq:delta0-superdiag} says. The longest path
in the whole graph is the staircase \(-1\to0\to1\to2\to3\), of length \(4=n+1\) and weight
\(1\cdot3\cdot3\cdot1=9\); it is the only path of that length, since a path of length \(4\) on
five vertices must visit all of them. So the improper polynomial has degree \(4\) and leading
coefficient \(9/4!=3/8\). Deleting the two vertices \(-1\) and \(0\), which the proper case
does, leaves the staircase \(1\to2\to3\) as the longest path, of length \(2=n-1\) and weight
\(3\cdot1=3\), so the proper polynomial has degree \(2\) and leading coefficient \(3/2!=3/2\).
Both agree with the \(n=3\) lines of Example \ref{tab-delta0-polys}, where the polynomials are
\(\tfrac38r^4+\cdots\) and \(\tfrac32r^2+\cdots\).

\begin{figure}[htbp]
\centering
\begin{tikzpicture}[shorten >=1pt,
  every node/.style={circle,draw,minimum size=7mm,inner sep=0pt,font=\small},
  wt/.style={draw=none,shape=rectangle,inner sep=1pt,font=\footnotesize}]
  \node (v0) at (0,0)   {$-1$};
  \node (v1) at (2.1,0) {$0$};
  \node (v2) at (4.2,0) {$1$};
  \node (v3) at (6.3,0) {$2$};
  \node (v4) at (8.4,0) {$3$};
  \draw[->,very thick] (v0) -- node[wt,above] {$1$}  (v1);
  \draw[->,very thick] (v1) -- node[wt,above] {$3$}  (v2);
  \draw[->,very thick] (v2) -- node[wt,above] {$3$}  (v3);
  \draw[->,very thick] (v3) -- node[wt,above] {$1$}  (v4);
  \draw[->] (v0) to[bend right=22] node[wt,below] {$3$} (v2);
  \draw[->] (v0) to[bend right=30] node[wt,below] {$4$} (v3);
  \draw[->] (v0) to[bend right=36] node[wt,below] {$1$} (v4);
  \draw[->] (v1) to[bend left=22]  node[wt,above] {$4$} (v3);
  \draw[->] (v1) to[bend left=30]  node[wt,above] {$1$} (v4);
  \draw[->] (v2) to[bend right=22] node[wt,below] {$1$} (v4);
\end{tikzpicture}
\caption{The weighted acyclic digraph \(G\) of Remark~\ref{rem-delta0-digraph} at \(n=3\), with
adjacency matrix \(N=T_3(0)-I\). Nodes are values of \(\tp\); the arc \(\tau\to\tau'\) carries
weight \(N[\tau][\tau']\). The staircase arcs, drawn heavy, carry Pascal's row \(1,3,3,1\) and
form the unique longest path: its length \(4\) is the degree of
\(\lvert\mathrm{FF}(3;0^r)\rvert\) in \(r\) and its weight \(9\), divided by \(4!\), is the
leading coefficient \(3/8\). The proper count lives on the sub-digraph spanned by
\(\lbrace1,2,3\rbrace\).}
\label{fig:delta0-digraph}
\end{figure}
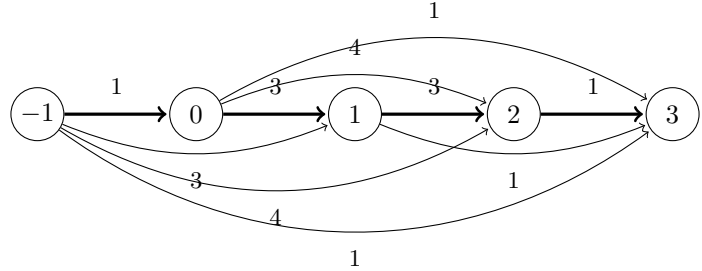
\label{rem-delta0-digraph}
\end{remark}

\begin{remark}
The improper leading coefficient is \(\prod_{j=0}^{n}\binom nj\big/(n+1)!\), the numerator being the
product of a whole row of Pascal's triangle (OEIS A001142). At \(n=2\) the two coefficients are
\(2/3!=1/3\) and \(1/1!=1\), matching the improper values \(14,30,55,91\) for \(r=2,\dots,5\), a
cubic with third difference \(2\), and the proper values
\(|\mathrm{FF}_{\mathrm{pr}}(2;0^r)|=r+1\). Theorem \ref{thm-delta0-degrees} was verified against
independent computation of both counts for \(n=2,\dots,8\) and \(r\) up to \(n+4\); see
Section \ref{sec-verification}.
\label{rem-delta0-degrees}
\end{remark}
\section{The proper variant}
\label{sec-proper}
The companion paper's polytope is built from \(S_n\), not \(\mathcal F_n\): its Hilbert functions come
from \textbf{proper} monomial submodules. The count restricts accordingly, but the restriction is not free:
it costs work in the proof below, and it costs a symmetry outright (Remark
\ref{rem-proper-not-reversible}).
The argument of Theorem \ref{thm-chaining} goes through the lex representative, so what must be checked
is that the lex representative of a properly achievable Hilbert function is itself proper.

\begin{lemma}
Let \(\vek x=\sum_i\shift^{d_i}(\vek s_i)\) with every \(\vek s_i\in S_n\). Then the unique lex submodule
\(L\) with \(\HF(F/L)=\vek x\) is proper.
\label{lem-lexproper}
\end{lemma}

\begin{proof}
By Definition 2.6, \(\mathrm{Mon}_j(L)\) is the initial segment of \(\mathrm{Mon}_j(F)\) of size
\(C_j=\dim F_j-x(j)\), and the segment fills the slots \(g_1,g_2,\dots\) in order. So \(L\) is proper
exactly when, for every \(i\), the segment does not reach \(g_i\)'s slot in the two module degrees
\(j=d_i\) and \(j=d_i+1\), that is, when \(C_j\le\sum_{k<i}\binom n{j-d_k}\) there. Writing
\(\binom nm=0\) and \(s_k(m)=0\) outside \(0\le m\le n\),
\[
  C_j=\sum_k\Bigl(\tbinom n{j-d_k}-s_k(j-d_k)\Bigr).
\]
For \(k\ge i\) and \(j\le d_i+1\) we have \(j-d_k\le d_i+1-d_k\le1\), and a proper \(\vek s_k\) has
\(s_k(0)=1=\binom n0\) and \(s_k(1)=n=\binom n1\), while both vanish in negative degrees; so every
such summand is zero. Hence
\[
  C_j\;=\;\sum_{k<i}\Bigl(\tbinom n{j-d_k}-s_k(j-d_k)\Bigr)\;\le\;\sum_{k<i}\tbinom n{j-d_k} ,
\]
which is exactly the capacity of the slots preceding \(g_i\).
\end{proof}

\begin{theorem}
Fix \(n\ge1\), \(r\ge1\) and a sorted \(\vek d\). Let \(A'_n\) be \(A_n\) with the columns \(\rho=0\) and
\(\rho=1\) set to zero, \(\vek u'_n\) its row sums, and \(T'_n(\delta)\) built from \(A'_n\) as in
Definition \ref{def-uT}. Then
\begin{equation}\label{eq:auto-16}
  \lvert\mathrm{FF}_{\mathrm{pr}}(n;\vek d)\rvert
  \;=\;\lvert\mathrm{CT}_{\mathrm{pr}}(n;\vek d)\rvert
  \;=\;\vek u'_n\cdot T'_n(\delta_2)\cdots T'_n(\delta_r)\cdot\vek 1_n .
\end{equation}
\label{thm-proper}
\end{theorem}

\begin{proof}
\(S_n=\lbrace \vek s\in\mathcal F_n:\indeg(\vek s)\ge2\rbrace\), which is what deleting the two columns
implements. As in Theorem \ref{thm-chaining}(d), \(L\mapsto\HF(F/L)\) is injective on lex submodules; it
carries proper lex submodules into \(\mathrm{FF}_{\mathrm{pr}}(n;\vek d)\), since a proper lex submodule is a
proper monomial submodule, and onto it: given \(\vek x\in\mathrm{FF}_{\mathrm{pr}}(n;\vek d)\), Theorem 4.2
(a) \(\Leftrightarrow\) (c) supplies the lex representative and Lemma \ref{lem-lexproper} shows it is
proper. Definition 2.6 \(\Leftrightarrow\) 2.7 with Lemma \ref{lem-dictionary}, restricted componentwise
to lex ideals of initial degree \(\ge2\), identifies the proper lex submodules with the proper-chained
tuples. The factorisation is then the rearrangement of Theorem \ref{thm-main} verbatim.
\end{proof}

\begin{proposition}
For \(n\ge1\), \(|\mathrm{FF}_{\mathrm{pr}}(n;\vek d)|=|S_n|^{\,r}\) if and only if every gap satisfies
\(\delta_i\ge n-1\).
\label{prop-proper-threshold}
\end{proposition}

\begin{proof}
As in Proposition \ref{prop-threshold}, with \(\max\tp=n\) and \(\min\indeg=2\) over \(S_n\).
\end{proof}

\begin{remark}
Two sequences recorded in earlier work on this project are reproduced exactly by
Theorem \ref{thm-proper} and extended. For \(r=2\), \(n=2,\dots,9\):
\begin{equation}\label{eq:auto-17}
  \delta=0:\quad 3,\ 12,\ 69,\ 627,\ 9904,\ 296289,\ 18009389,\ 2346218172,
\end{equation}
\begin{equation}\label{eq:auto-18}
  \delta=1:\quad 4,\ 22,\ 162,\ 1858,\ 38232,\ 1571774,\ 138881518,\ 27635025278 .
\end{equation}
The prefixes through \(n=6\) are the independently enumerated ones. Neither sequence, nor \(|S_n|\)
itself, appears to be in the OEIS (searched October 2026).
\label{rem-proper-sequences}
\end{remark}
\section{Growth}
\label{sec-growth}
\begin{lemma}
(\emph{The support of the transfer matrix.}) Let \(n\ge1\) and \(\delta\ge0\). For all \(\tau,\tau'\) in
the range \eqref{eq:def-uT-range},
\begin{equation}\label{eq:transfer-support}
  T_n(\delta)[\tau][\tau']>0\quad\Longleftrightarrow\quad\tau'\ge\tau-\delta .
\end{equation}
In \(T'_n(\delta)\) the two columns \(\tau'=-1\) and \(\tau'=0\) vanish identically, and
\eqref{eq:transfer-support} holds as stated for every \(\tau'\in\lbrace1,\dots,n\rbrace\).
\label{lem-transfer-support}
\end{lemma}

\begin{proof}
By \eqref{eq:def-T} the entry \(T_n(\delta)[\tau][\tau']=\sum_{\rho>\tau-\delta}A_n[\tau'][\rho]\)
is a sum of non-negative terms, so it is positive exactly when some \(\rho>\tau-\delta\) has
\(A_n[\tau'][\rho]>0\). The largest \(\rho\) with \(A_n[\tau'][\rho]>0\) is \(\rho=\tau'+1\).
Nothing larger occurs, since \(A_n\) is supported on \(\rho\le\tau+1\) by Definition \ref{def-A},
which in row \(\tau'\) reads \(\rho\le\tau'+1\); and \(\tau'+1\) does occur, being attained by the
full truncation at \(\tau'\), the unique \(\vek s\) with \(s(j)=\binom nj\) for \(j\le\tau'\) and
\(s(j)=0\) beyond, which is the void family at \(\tau'=-1\) and the full simplex at \(\tau'=n\).
The sum is therefore positive if and only if \(\tau'+1>\tau-\delta\), which for integers is
\(\tau'\ge\tau-\delta\).

For the proper matrix, \(A'_n\) is \(A_n\) with the columns \(\rho=0\) and \(\rho=1\) deleted
(Theorem \ref{thm-proper}). A row \(\tau'\le0\) of \(A_n\) is supported on \(\rho\le\tau'+1\le1\) and
is therefore deleted entirely, which is the vanishing of those two columns of \(T'_n(\delta)\); and
for \(\tau'\ge1\) the entry \(A_n[\tau'][\tau'+1]=1\) survives the deletion, since \(\tau'+1\ge2\),
so the argument above applies to \(A'_n\) word for word.
\end{proof}

Read \eqref{eq:transfer-support} as a statement about the digraph on the \(\tp\)-values, in the
manner of Remark \ref{rem-delta0-digraph}: the arcs of \(T_n(\delta)\) are exactly the pairs
\(\tau\to\tau'\) with \(\tau'\ge\tau-\delta\), so one may move up without limit and down by at most
\(\delta\). Three readings of that are used below.

\begin{itemize}
\item At \(\delta=0\) the condition is \(\tau'\ge\tau\), so \(T_n(0)\) is upper triangular. That is the
first step of the proof of Proposition \ref{prop-delta0} again, reached here from the general
statement instead of separately.
\item For \(\delta\ge1\) the digraph is strongly connected, so \(T_n(\delta)\) is \textbf{irreducible}: every
\(\tau\) has an arc to \(n\), since \(n\ge\tau-\delta\), and from \(n\) one may descend to any
\(\tau'\ge n-\delta\), so repeating reaches every value down to \(-1\) in at most
\(\lceil(n+1)/\delta\rceil\) steps.
\item For every \(\delta\ge0\) there is a loop at each vertex, since \(\tau\ge\tau-\delta\). Together
with irreducibility that makes \(T_n(\delta)\) \textbf{aperiodic} for \(\delta\ge1\), by criterion (iii)
of \autocite[Note V.31]{FlajoletSedgewick2009}.
\end{itemize}

The same three readings hold for \(T'_n(\delta)\) on the principal block
\(\tau,\tau'\in\lbrace1,\dots,n\rbrace\), with \(\lceil(n-1)/\delta\rceil\) in place of
\(\lceil(n+1)/\delta\rceil\). Off that block \(T'_n(\delta)\) is reducible for every \(\delta\),
the two vanishing columns meaning that no arc enters \(\tau'=-1\) or \(\tau'=0\) at all.

\begin{proposition}
For a constant gap \(\delta\), put
\begin{equation}\label{eq:def-ar}
  a_r\;:=\;\lvert\mathrm{FF}(n;0,\delta,2\delta,\dots,(r-1)\delta)\rvert
  \;=\;\vek u_n\cdot T_n(\delta)^{r-1}\cdot\vek 1_n .
\end{equation}
Then \(a_r\) satisfies the linear recurrence whose characteristic polynomial is
that of \(T_n(\delta)\), and its exponential growth rate is the Perron root \(\lambda\) of
\(T_n(\delta)\), in the sense that \(a_r^{1/r}\to\lambda\). The same holds for the proper count with
\(T'_n(\delta)\).
\label{prop-recurrence}
\end{proposition}

\begin{proof}
The recurrence is the Cayley--Hamilton theorem applied to \(T_n(\delta)\). That theorem says that the
characteristic polynomial annihilates the matrix, and the vectors \(\vek u_n\) and \(\vek 1_n\) are
fixed once and for all, so the same polynomial annihilates the scalar sequence \(a_r\). For the
growth rate, write \(\Vert M\Vert\) for the sum of the entries of a non-negative matrix \(M\); it is
a norm on non-negative matrices, so \(\Vert T_n(\delta)^k\Vert^{1/k}\) tends to the spectral radius
of \(T_n(\delta)\) \autocite[Theorem 5.7.10]{HornJohnson2013}, and that radius is itself an
eigenvalue, the Perron root \(\lambda\), because \(T_n(\delta)\) is non-negative
\autocite[Theorem 8.3.1]{HornJohnson2013}. Neither step needs irreducibility, which \(T_n(0)\) does
not have. Every entry of
\(\vek u_n\) is positive, since each value \(\tau=-1,\dots,n\) of \(\tp\) is attained, so
\[
  \bigl(\min_\tau u_n[\tau]\bigr)\,\Vert T_n(\delta)^{r-1}\Vert\;\le\;a_r\;\le\;
  \bigl(\max_\tau u_n[\tau]\bigr)\,\Vert T_n(\delta)^{r-1}\Vert ,
\]
and taking \(r\)-th roots gives \(a_r^{1/r}\to\lambda\). The recurrence argument applies verbatim to
\(T'_n(\delta)\). For the growth rate \(\vek u'_n\) is not positive, since it vanishes at \(\tau=-1,0\);
but so do the columns \(\tau'=-1,0\) of \(T'_n(\delta)\), so the proper count is the same product
taken over the block \(\lbrace1,\dots,n\rbrace\), where \(\vek u'_n\) is positive, and that block
carries every non-zero eigenvalue of \(T'_n(\delta)\). The argument above then runs on the block.
\end{proof}

\begin{remark}
It is worth saying what Proposition \ref{prop-recurrence} does \textbf{not} claim, since the natural stronger
reading of it is false. There is in general no constant \(K\) with \(a_r\sim K\lambda^{\,r}\). At
\(\delta=0\) the matrix is unipotent, so \(\lambda=1\), while \(a_r\) is a polynomial in \(r\) of
degree exactly \(n+1\) by Theorem \ref{thm-delta0-degrees}; the ratio \(a_r/\lambda^{\,r}\) is then
\(a_r\) itself, and is unbounded. The hypothesis that fails there is irreducibility, and
\(\delta=0\) is the only place it fails: for every \(\delta\ge1\) the constant does exist, and
Proposition \ref{prop-growth-constant} produces it, together with an error term. The Perron roots of
Table \ref{tab:perron-roots} are the \(\lambda\) of that proposition.
\label{rem-recurrence-not-asymptotic}
\end{remark}

\begin{proposition}
(\emph{The constant, for a positive gap.}) Fix \(n\ge1\) and a constant gap \(\delta\ge1\), let \(a_r\) be
as in \eqref{eq:def-ar}, and let \(\lambda\) be the Perron root of \(T_n(\delta)\). Then there are
constants \(K>0\) and \(\varrho\) with \(0\le\varrho<\lambda\) such that
\begin{equation}\label{eq:growth-constant}
  a_r\;=\;K\lambda^{\,r}+O(\varrho^{\,r})\qquad(r\to\infty) ,
\end{equation}
and in particular \(a_r\sim K\lambda^{\,r}\). The same holds for the proper count, with
\(T'_n(\delta)\), its own Perron root and its own constants.
\label{prop-growth-constant}
\end{proposition}

\begin{proof}
By Lemma \ref{lem-transfer-support}, and the second and third of the readings drawn from it above,
\(T_n(\delta)\) is irreducible and aperiodic for \(\delta\ge1\). Put \(F(z)=(I-zT_n(\delta))^{-1}\), so that
\([z^m]F_{\tau,\tau'}(z)=\bigl(T_n(\delta)^m\bigr)[\tau][\tau']\). Then
\autocite[Theorem V.7]{FlajoletSedgewick2009} applies and gives
\begin{equation}\label{eq:fs-v7}
  \bigl(T_n(\delta)^m\bigr)[\tau][\tau']\;=\;\varphi_{\tau,\tau'}\lambda^{\,m}+O(\varrho^{\,m}),
  \qquad 0\le\varrho<\lambda ,
\end{equation}
with constants \(\varphi_{\tau,\tau'}>0\), the dominant eigenvalue \(\lambda\) there being the
largest positive eigenvalue of \(T_n(\delta)\), which is its Perron root. Taking \(m=r-1\) in
\eqref{eq:fs-v7} and summing against \(\vek u_n\) and \(\vek 1_n\), a sum of \((n+2)^2\) terms and
so a finite one, \eqref{eq:def-ar} gives
\[
  a_r\;=\;\Bigl(\sum_{\tau,\tau'}u_n[\tau]\varphi_{\tau,\tau'}\Bigr)\lambda^{\,r-1}+O(\varrho^{\,r-1}) .
\]
Every \(u_n[\tau]\) is positive, as in the proof of Proposition \ref{prop-recurrence}, so the bracket
is positive; put \(K=\lambda^{-1}\sum_{\tau,\tau'}u_n[\tau]\varphi_{\tau,\tau'}>0\). Finally
\(O(\varrho^{\,r-1})=O(\varrho^{\,r})\), the two differing only by the constant factor
\(\varrho^{-1}\) when \(\varrho>0\) and both being zero for \(r\ge2\) when \(\varrho=0\). That is
\eqref{eq:growth-constant}, and \(K>0\) with \(\varrho<\lambda\) gives the asymptotic equivalence.

For the proper count, \(T'_n(\delta)\) is not irreducible on the whole index set, by
Lemma \ref{lem-transfer-support}. But the proof of Proposition \ref{prop-recurrence} has already
reduced that count to the same product taken over the principal block \(\lbrace1,\dots,n\rbrace\),
where \(\vek u'_n\) is positive and which carries every non-zero eigenvalue of \(T'_n(\delta)\); and
on that block \(T'_n(\delta)\) is irreducible and aperiodic for \(\delta\ge1\), again by
Lemma \ref{lem-transfer-support}. The argument above then runs on the block verbatim.
\end{proof}

\begin{mexample}
At \(n=2\), \(\delta=1\), the characteristic polynomial of \(T_2(1)\) and the recurrence it gives are
\[
  \lambda(\lambda-1)(\lambda^2-4\lambda+1),\qquad a_r=5a_{r-1}-5a_{r-2}+a_{r-3};
\]
the values are
\(5,20,76,285,1065\) for \(r=1,\dots,5\), and indeed \(5\cdot285-5\cdot76+20=1065\). The Perron root
is \(2+\sqrt3\approx3.732\). At \(n=3\), \(\delta=1\) they are
\[
  \lambda(\lambda^4-9\lambda^3+21\lambda^2-15\lambda+3),\qquad
  a_r=9a_{r-1}-21a_{r-2}+15a_{r-3}-3a_{r-4},
\]
checked against \(10,69,426,2532,14847\).

In both cases a factor of \(\lambda\) can be dropped before reading off the recurrence, and it is
worth saying why rather than leaving it to be noticed. \(T_n(\delta)\) is singular for
\(\delta\ge1\): rows \(\tau=-1\) and \(\tau=0\) both equal \(\vek u_n\), since
\(\rho>-1-\delta\) and \(\rho>-\delta\) admit the same \(\rho\ge0\). A root at \(0\)
contributes a term \(c\cdot0^{\,r}\), which vanishes for \(r\ge1\), so the quotient annihilates
the sequence just as the characteristic polynomial does and gives a shorter recurrence. The proper
matrix shows the same phenomenon doubled: \(T'_n(0)\) has characteristic polynomial
\(\lambda^2(\lambda-1)^n\), the two factors of \(\lambda\) being the two identically zero
columns that the proof of Theorem \ref{thm-delta0-degrees} turns on.

Across \(n\) and \(\delta\) the Perron roots are as collected in Table \ref{tab:perron-roots},
with \(\lvert\mathcal F_n\rvert\) and \(\lvert S_n\rvert\) in the last two columns for
comparison; an entry in bold has reached its limit exactly, which is
Proposition \ref{prop-perron-threshold}.
\begingroup\small
\begin{center}
\begin{tabular}{@{}cc rr cc@{}}
  \toprule
  \(n\) & \(\delta\) & \(\lambda(\delta)\) & \(\lambda_{\mathrm{pr}}(\delta)\)
        & \(\lvert\mathcal F_n\rvert\) & \(\lvert S_n\rvert\)\\
  \midrule
  2 & 0 & 1 & 1 & 5 & 2\\
  2 & 1 & 3.7321 & \textbf{2} & &\\
  2 & 2 & 4.7913 & \textbf{2} & &\\
  2 & 3 & \textbf{5} & \textbf{2} & &\\
  \addlinespace
  3 & 1 & 5.8188 & 4.3028 & 10 & 5\\
  3 & 2 & 9.0138 & \textbf{5} & &\\
  3 & 3 & 9.8990 & \textbf{5} & &\\
  3 & 4 & \textbf{10} & \textbf{5} & &\\
  \addlinespace
  4 & 1 & 9.8557 & 8.7566 & 26 & 16\\
  4 & 2 & 20.5859 & 15.3485 & &\\
  4 & 3 & 25.2491 & \textbf{16} & &\\
  4 & 4 & 25.9615 & \textbf{16} & &\\
  4 & 5 & \textbf{26} & \textbf{16} & &\\
  \addlinespace
  5 & 1 & 17.7573 & 17.1251 & 96 & 70\\
  5 & 2 & 57.8411 & 52.1804 & &\\
  5 & 3 & 87.4360 & 69.3802 & &\\
  5 & 4 & 95.4659 & \textbf{70} & &\\
  5 & 5 & 95.9896 & \textbf{70} & &\\
  5 & 6 & \textbf{96} & \textbf{70} & &\\
  \bottomrule
\end{tabular}
\captionof{table}{The Perron root \(\lambda(\delta)\) of \(T_n(\delta)\) and the Perron root
\(\lambda_{\mathrm{pr}}(\delta)\) of \(T'_n(\delta)\), for \(2\le n\le5\) and small constant
gaps \(\delta\). A bold entry has reached its limiting value exactly, which happens from
\(\delta=n+1\) on for the improper root and from \(\delta=n-1\) on for the proper one
(Proposition~\ref{prop-perron-threshold}).}
\label{tab:perron-roots}
\end{center}
\endgroup
At \(\delta=0\) both roots are \(1\), the unipotent case of Proposition \ref{prop-delta0}, where the
growth is polynomial rather than exponential. The proper root reaches its limit two steps before the
improper one, and in every case computed (\(n\le6\)) the last improper step is the narrowest.
\label{mex-recurrences}
\end{mexample}

\begin{proposition}
(\emph{Normal form: all the gap-dependence sits in one shift.}) Two matrices are needed, and only two. The
\textbf{base matrix} is the transfer matrix at gap zero,
\begin{equation}\label{eq:def-B}
  B_n\;=\;T_n(0),\qquad B_n[\tau][\tau']=\sum_{\rho>\tau}A_n[\tau'][\rho] ,
\end{equation}
which by Proposition \ref{prop-delta0} is unipotent and carries the Pascal superdiagonal of
Lemma \ref{lem-A-structure}. The \textbf{saturating one-step shift} on the \(\tp\)-index is the
\((n+2)\times(n+2)\) matrix
\begin{equation}\label{eq:def-P}
  P_n[\tau][\tau']=\begin{cases}
    1, & \tau'=\max(\tau-1,-1),\\
    0, & \text{otherwise},
  \end{cases}
  \qquad \tau,\tau'\in\lbrace-1,0,\dots,n\rbrace .
\end{equation}
Both are indexed as in Definition \ref{def-uT}. Let \(e_{-1}\) be the standard basis vector supported at
the index \(-1\). Then for every \(\delta\ge0\),
\begin{align}
  T_n(\delta) &= P_n^{\delta}B_n ,\label{eq:normalform}\\
  \vek u_n    &= e_{-1}^{\!\top}B_n ,\label{eq:normalform-u}
\end{align}
and consequently
\begin{equation}\label{eq:normalform-product}
  \lvert\mathrm{FF}(n;\vek d)\rvert
   =e_{-1}^{\!\top}B_n\,P_n^{\delta_2}B_n\,P_n^{\delta_3}\cdots P_n^{\delta_r}B_n\,\vek 1_n .
\end{equation}
\label{prop-normalform}
\end{proposition}

\begin{proof}
By Definition \ref{def-uT} the entries of the transfer matrix are
\begin{equation*}
  T_n(\delta)[\tau][\tau']=\sum_{\rho>\tau-\delta}A_n[\tau'][\rho] ,
\end{equation*}
which depend on \(\tau\) and \(\delta\) only through the difference \(\tau-\delta\), and which
therefore coincide with row \(\tau-\delta\) of \(T_n(0)\) whenever \(\tau-\delta\ge-1\). For
\(\tau-\delta<-1\) the condition \(\rho>\tau-\delta\) is vacuous, and so is the condition
\(\rho>-1\) that defines row \(-1\) of \(T_n(0)\), so instead of running off the index set the row
saturates at row \(-1\). Shifting every row down by one and saturating at \(-1\) is exactly left
multiplication by the matrix \(P_n\) of \eqref{eq:def-P}, which is \eqref{eq:normalform}. For
\eqref{eq:normalform-u}, read off row \(-1\) of \(T_n(0)\):
\begin{equation*}
  T_n(0)[-1][\tau']=\sum_{\rho>-1}A_n[\tau'][\rho]=\sum_{\rho\ge0}A_n[\tau'][\rho]=u_n[\tau'] .
\end{equation*}
Substituting \eqref{eq:normalform} and \eqref{eq:normalform-u} into Theorem \ref{thm-main} gives
\eqref{eq:normalform-product}.
\end{proof}

\begin{remark}
The normal form has two consequences. First, the count depends on each gap \textbf{only} through
\(\min(\delta_i,n+1)\), because \(P_n\) saturates: \(P_n^{\delta}=P_n^{n+1}=\vek1_n\,e_{-1}^{\!\top}\) for
every \(\delta\ge n+1\). Second, that same identity gives a second and shorter proof of
Propositions \ref{prop-threshold} and \ref{prop-perron-threshold}: above the threshold
\(T_n(\delta)=\vek1_n\,e_{-1}^{\!\top}B_n=\vek1_n\,\vek u_n\), manifestly of rank one with non-zero eigenvalue
\(\vek u_n\cdot\vek1_n=\lvert\mathcal F_n\rvert\). The pair \((B_n,P_n)\), one unipotent
matrix carrying the Pascal superdiagonal of Lemma \ref{lem-A-structure} and one \(0/1\) shift, encodes
every gap vector at once.

It also explains why the matrices do not commute while the improper count is nevertheless
reversal-symmetric
(Example \ref{ex-factors} and Remark \ref{rem-gap-reversal}): the non-commutation is visible in
\eqref{eq:normalform-product}, since \(P_n\) and \(B_n\) do not commute, whereas the reversal symmetry is the
persymmetry of \(A_n\) applied factor by factor inside the bilinear form, a property the individual
\(T_n(\delta)\) cannot express once \(\rho\) has been summed out.
\label{rem-normalform}
\end{remark}

\begin{proposition}
(\emph{The Perron root above the threshold.}) Let \(\lambda(\delta)\) be the Perron root of \(T_n(\delta)\)
and \(\lambda_{\mathrm{pr}}(\delta)\) that of its proper analogue. Then
\begin{align}
  \lambda(\delta) &= \lvert\mathcal F_n\rvert \quad\text{for every }\delta\ge n+1,\label{eq:perron-threshold}\\
  \lambda_{\mathrm{pr}}(\delta) &= \lvert S_n\rvert \quad\text{for every }\delta\ge n-1 .\label{eq:perron-threshold-pr}
\end{align}
In both cases the value is attained, not merely approached, and the thresholds are exactly those of
Propositions \ref{prop-threshold} and \ref{prop-proper-threshold}.
\label{prop-perron-threshold}
\end{proposition}

\begin{proof}
By Definition \ref{def-uT},
\[
  T_n(\delta)[\tau][\tau']=\sum_{\rho>\tau-\delta}A_n[\tau'][\rho] .
\]
The
condition \(\rho>\tau-\delta\) is vacuous exactly when \(\delta>\max\tp-\min\indeg\), which over
\(\mathcal F_n\) is \(\delta\ge n+1\), the computation already made in
Proposition \ref{prop-threshold}. When it is vacuous the inner sum no longer depends on \(\tau\):
every row of \(T_n(\delta)\) is the same row vector \(\vek u_n\), whose entry
\(u_n[\tau']=\sum_\rho A_n[\tau'][\rho]\) is the number of \(\vek s\in\mathcal F_n\) with
\(\tp(\vek s)=\tau'\). So \(T_n(\delta)=\vek1_n\,\vek u_n\) has rank one, its non-zero eigenvalue is
\(\vek u_n\cdot\vek1_n=\sum_{\tau'}u_n[\tau']=|\mathcal F_n|\), and
every other eigenvalue is \(0\). The proper statement is the same computation with \(\min\indeg=2\)
over \(S_n\), giving the threshold \(\delta\ge n-1\) and the root \(|S_n|\).

That the thresholds are exact, that is, that the root is strictly smaller below them, follows from
two results of Perron--Frobenius theory. For \(0\le M\le M'\) entrywise the spectral radius of
\(M\) is at most that of \(M'\) \autocite[Corollary 8.1.19]{HornJohnson2013}; and when \(M'\) is
irreducible the two are equal only if \(M=\zeta DM'D^{-1}\) for some diagonal unitary \(D\) and
some \(\zeta\) of modulus one \autocite[Theorem 8.4.5]{HornJohnson2013}, which for non-negative
\(M\) and \(M'\) forces \(M=M'\) entrywise, since it makes each \(M[i][j]\) equal to
\(|M'[i][j]|\). So \(M\ne M'\) makes the inequality strict.
Here \(M'=T_n(n+1)=\vek1_n\,\vek u_n\) has every entry positive,
because \(\vek u_n\) does, and so is irreducible; for \(\delta\le n\) the matrix \(T_n(\delta)\) lies
below it by the monotonicity in Definition \ref{def-uT}, and differs from it, since the full simplex
followed by the void family is not chained at \(\delta=n\). The proper case is the same argument on
the block \(\lbrace1,\dots,n\rbrace\), where \(T'_n(n-1)\) is positive and, for \(\delta\le n-2\), a
summand with \(\tp=n\) followed by one with \(\indeg=2\) is not chained.
\end{proof}

\begin{remark}
Proposition \ref{prop-perron-threshold} is Propositions \ref{prop-threshold} and
\ref{prop-proper-threshold} stated about the matrix instead of about the sumset. Those propositions say that above
the threshold the chaining constraint is vacuous, so the sumset is a full Cartesian product and
\(|\mathrm{FF}(n;\vek d)|\) is \(|\mathcal F_n|^{\,r}\) exactly. One level down, the constraint is what
couples consecutive summands, and with it gone \(T_n(\delta)\) forgets its row index entirely. A rank-one matrix \(\vek1_n\,\vek u_n\)
has a single non-zero eigenvalue, and it is the total count; multiplying \(r-1\) copies of it
reproduces \(|\mathcal F_n|^{\,r}\) directly.

Figure \ref{fig:perron-vs-delta} shows both roots at \(n=5\). The shape is the content of
Proposition \ref{prop-perron-threshold}: each climbs and then stops dead, the proper one two steps
earlier than the improper, and after that they are constant rather than asymptotic.

\begin{figure}[htbp]
\centering
\includegraphics[width=0.86\linewidth]{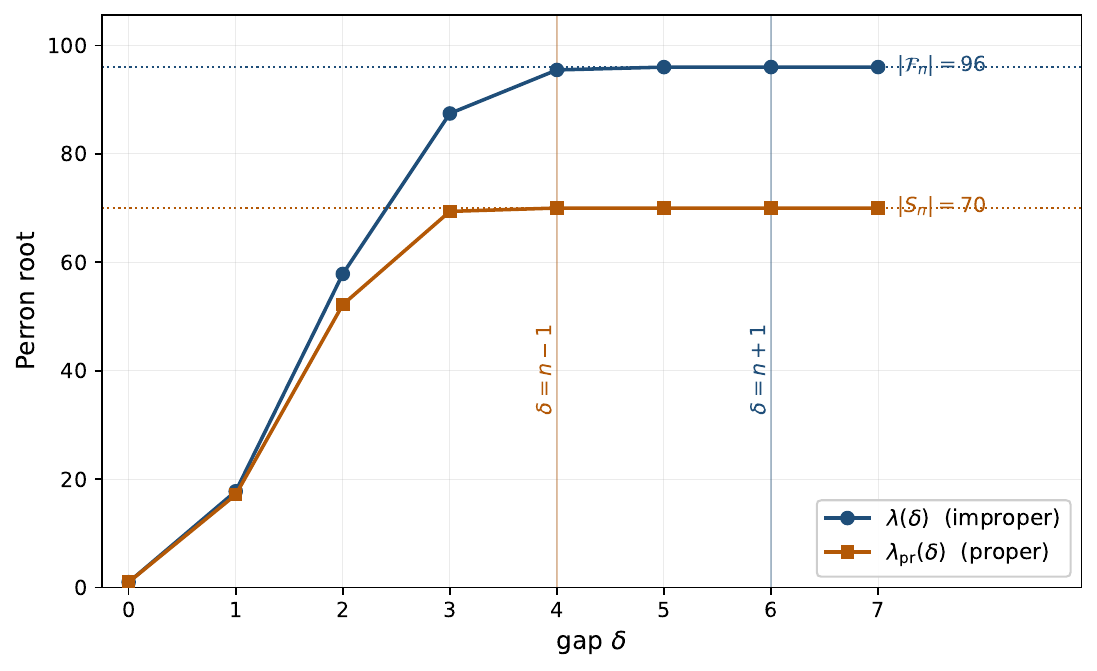}
\caption{The Perron roots \(\lambda(\delta)\) and \(\lambda_{\mathrm{pr}}(\delta)\) of
Definition~\ref{def-uT}'s transfer matrices at \(n=5\), against the gap \(\delta\). Dotted lines mark
\(\lvert\mathcal F_5\rvert=96\) and \(\lvert S_5\rvert=70\); the vertical rules mark the two
thresholds of Proposition~\ref{prop-perron-threshold}, \(\delta=n-1\) for the proper root and
\(\delta=n+1\) for the improper. Each root is *equal* to its limit from its own threshold onward, not
merely close to it.}
\label{fig:perron-vs-delta}
\end{figure}

Below the threshold the root is strictly smaller, and it climbs. For \(n=3\), where
\(|\mathcal F_3|=10\) and \(|S_3|=5\), the improper roots at \(\delta=1,2,3,4\) are
\begin{equation}\label{eq:auto-19}
  5.8188\ldots,\quad 9.0138\ldots,\quad 9.8990\ldots,\quad 10,
\end{equation}
reaching \(|\mathcal F_3|\) exactly at \(\delta=4=n+1\) and staying there; the proper roots at
\(\delta=1,2\) are \(4.3028\ldots\) and \(5\), reaching \(|S_3|\) exactly at \(\delta=2=n-1\). The
approach is from below in both cases, as it must be, since \(T_n(\delta)\) increases entrywise with
\(\delta\). Note how narrow the last step is: at \(\delta=n\) the improper root is already
\(9.8990\ldots\) against \(10\), and at \(n=6\) it is \(552.998\ldots\) against \(553\), which is
Proposition \ref{prop-threshold}'s ``the constraint fails for exactly one pair'' made quantitative.
\label{rem-perron}
\end{remark}

The remaining growth question is in the number of vertices. The count \(|\mathrm{FF}(n;\vek d)|\) itself
cannot be polynomial in \(n\), since it dominates \(|\mathcal F_n|=F^n(n)\). The right generalisation
is Linusson's own: keep the dimension bound and free the vertex bound.

\begin{definition}
For \(p\ge n\), let \(\mathcal F_n^{(p)}\) be the set of \(f\)-vectors with \(f_1\le p\) and
\(f_j=0\) for \(j>n\), and \(\mathrm{FF}^{(p)}(n;\vek d)\) the corresponding sumset.

From here to the end of this section the exterior algebra is on \(p\) generators, so \(p\) is the
number of vertices, and \(n\) bounds the size of a face. The subscript \(n\) on
\(\mathcal F_n^{(p)}\), \(A^{(p)}_n\) and \(T^{(p)}_n\) is that bound, not the size of a ground set.
\label{def-twoparam}
\end{definition}

\begin{theorem}
Fix \(n\ge1\) and \(r\ge1\); let \(d_1\le\cdots\le d_r\) be integers and put
\(\delta_i=d_i-d_{i-1}\) for \(2\le i\le r\), so that \(\vek d\) is sorted and the gaps are
non-negative. Let \(p\ge n\), let \(\mathcal F_n^{(p)}\) be as in Definition \ref{def-twoparam}, and set
\begin{equation}\label{eq:Ap}
  A^{(p)}_n[\tau][\rho]=E^p(\tau,\rho-1)-E^p(\tau-1,\rho-1),
  \qquad \tau=-1,\dots,n,\quad\rho=0,\dots,n+1,
\end{equation}
with \(\vek u^{(p)}_n\) its vector of row sums and \(T^{(p)}_n(\delta)\) built from \(A^{(p)}_n\) by
\eqref{eq:def-u}--\eqref{eq:def-T}, both indexed as in Definition \ref{def-uT}. Then
\begin{equation}\label{eq:auto-20}
  \lvert\mathrm{FF}^{(p)}(n;\vek d)\rvert\;=\;\vek u^{(p)}_n\cdot T^{(p)}_n(\delta_2)\cdots T^{(p)}_n(\delta_r)\cdot\vek 1_n ,
\end{equation}
and for fixed \(n\), \(r\) and \(\vek d\) this is a polynomial in \(p\).
\label{thm-poly-p}
\end{theorem}

\begin{proof}
Lemma \ref{lem-joint} and Theorems \ref{thm-chaining} and \ref{thm-main} go through verbatim over the exterior
algebra on \(p\) generators, with one extra ingredient: the lex representative must respect the
truncation, i.e. \(\mathfrak m^{n+1}F\subseteq M\) must force \(L\supseteq\mathfrak m^{n+1}F\). This
is the mirror image of the capacity argument of Lemma \ref{lem-lexproper}. In module degree
\(j=d_i+m\) with \(m\ge n+1\), every component \(k\le i\) has \(s_k(j-d_k)=0\) because
\(\tp(s_k)\le n\); hence
\[
  C_j\;\ge\;\sum_{k\le i}\binom p{j-d_k},
\]
with \(C_j\) as in that lemma's proof, and the initial segment covers
\(g_i\)'s slot. Polynomiality in \(p\) is then immediate from Theorem \ref{thm-main}'s formula, each
entry \(E^p(\tau,\rho-1)\) being a polynomial in \(p\) for fixed \(\tau,\rho\) by Linusson's
Theorem 3.3.
\end{proof}

\textbf{Status.} The reduction above is proved; the polynomiality of the entries is Linusson's Theorem 3.3,
which we cite rather than reprove, so Theorem \ref{thm-poly-p} is proved modulo that citation. It was
verified against brute-force truncated sumsets for \(n\le3\), \(p\le5\), \(r\le3\), with no
mismatches.

\begin{mexample}
(\emph{The polynomials in \(p\).}) Nine configurations, chosen to span the regimes: all gaps zero, a constant gap
\(\delta=1\), a saturating gap \(\delta=4>n\) where the summands have come apart, and two
non-constant gap vectors that are reverses of each other.
\begingroup\footnotesize
\begin{center}
\begin{tabular}{@{}ccl@{}}
  \toprule
  \(\vek d\) & gaps & \(\lvert\mathrm{FF}^{(p)}(n;\vek d)\rvert\) \\
  \midrule
  \multicolumn{3}{@{}l}{\(n=2\)}\\
  \((0,0)\)   & \((0)\)    & \(p^{3}-\tfrac12p^{2}+\tfrac52p+3\)\\
  \((0,1)\)   & \((1)\)    & \(\tfrac1{12}p^{5}+\tfrac1{12}p^{4}+\tfrac5{12}p^{3}+\tfrac{11}{12}p^{2}+\tfrac52p+4\)\\
  \((0,0,0)\) & \((0,0)\)  & \(3p^{3}-2p^{2}+5p+4\)\\
  \((0,1,2)\) & \((1,1)\)  & \(\tfrac1{24}p^{7}+\tfrac1{36}p^{6}+\tfrac13p^{5}+\tfrac{19}{36}p^{4}+\tfrac{43}{24}p^{3}+\tfrac{31}9p^{2}+\tfrac{41}6p+8\)\\
  \addlinespace
  \multicolumn{3}{@{}l}{\(n=3\)}\\
  \((0,0)\)   & \((0)\)    & \(\tfrac7{90}p^{6}-\tfrac3{10}p^{5}+\tfrac{29}{72}p^{4}+\tfrac34p^{3}-\tfrac{173}{360}p^{2}+\tfrac{51}{20}p+3\)\\
  \((0,1)\)   & \((1)\)    & \(\tfrac1{1080}p^{9}-\tfrac7{720}p^{7}+\tfrac3{80}p^{6}+\tfrac1{48}p^{5}+\tfrac5{48}p^{4}+\tfrac{473}{1080}p^{3}+\tfrac{103}{120}p^{2}+\tfrac{51}{20}p+4\)\\
  \((0,4)\)   & \((4)\)    & \(\tfrac1{32400}p^{12}-\tfrac1{5400}p^{11}+\cdots+\tfrac{2321}{3600}p^{2}+\tfrac{17}5p+4\)\\
  \((0,1,3)\) & \((1,2)\)  & \(\tfrac1{32400}p^{14}-\tfrac1{6480}p^{13}+\cdots+\tfrac{1563}{400}p^{2}+\tfrac{527}{60}p+8\)\\
  \((0,2,3)\) & \((2,1)\)  & \(\tfrac1{32400}p^{14}-\tfrac1{4050}p^{13}+\cdots+\tfrac{14299}{3600}p^{2}+\tfrac{527}{60}p+8\)\\
  \bottomrule
\end{tabular}
\captionof{table}{The counts \(\lvert\mathrm{FF}^{(p)}(n;\vek d)\rvert\) as polynomials in \(p\), for nine shift vectors.}
\label{tab-poly-in-p}
\end{center}
\endgroup
Each polynomial was obtained by exact rational interpolation at \(p=n,\dots,n+\deg+1\) and then
checked by evaluating it at three further values of \(p\) against the matrix product of
Theorem \ref{thm-poly-p}; the two longest are abbreviated in the middle, their full coefficient lists
being unenlightening. Two features are immediate. The coefficients are \textbf{not} positive, unlike the
\(r\)-direction up to \(n=6\); and the two non-constant gap vectors \((1,2)\) and \((2,1)\), reverses
of each other, give \textbf{different} polynomials. The second of those is taken up in
Remark \ref{rem-reversal-only-at-p-n} below.
\label{ex-poly-p-tables}
\end{mexample}

\begin{remark}
(\emph{Gap reversal is a \(p=n\) phenomenon.}) The last two rows of Example \ref{ex-poly-p-tables} give
different polynomials for gap vectors that are reverses of each other, while
Remark \ref{rem-gap-reversal} proves that \(\lvert\mathrm{FF}(n;\vek d)\rvert\) \textbf{is} invariant under
reversing the gaps. Both are correct, and the apparent tension locates the hypothesis. At \(p=n\) the
two polynomials take the same value, \(\lvert\mathrm{FF}(3;0,1,3)\rvert=\lvert\mathrm{FF}(3;0,2,3)\rvert=607\),
which is the symmetry; for \(p>n\) they part company, \(6696\) against \(6592\) at \(p=4\) and
\(94402\) against \(87165\) at \(p=5\). That is as it should be. The symmetry comes from the Alexander
duality of Lemma \ref{lem-alexander}, which complements inside \([n]\), and in the truncated setting of
Definition \ref{def-twoparam} the vertex bound \(p\) and the dimension bound \(n\) differ, so there is
no involution of \(\mathcal F_n^{(p)}\) with which to run the argument. The invariance is special to
the untruncated case, and no reversal symmetry should be expected of the polynomials themselves.
\label{rem-reversal-only-at-p-n}
\end{remark}

\begin{mexample}
The factors, written out. Take \(n=3\), \(r=3\) and \(\vek d=(0,1,3)\), so the gaps are
\((\delta_2,\delta_3)=(1,2)\) and the product of Theorem \ref{thm-main} is \(\vek u_3\,T_3(1)\,T_3(2)\,\vek1_3\)
with the matrices of Example \ref{ex-uT}. Multiplying left to right,
\[
  \vek u_3=(1,1,3,4,1)
  \ \xrightarrow{\ T_3(1)\ }\ (2,5,19,33,10)
  \ \xrightarrow{\ T_3(2)\ }\ (26,59,187,266,69),
\]
whose entries sum to \(\lvert\mathrm{FF}(3;0,1,3)\rvert=607\). Each application of a transfer matrix
advances one summand: the vector after \(k\) factors records, for each value of \(\tp\), how many
chained \((k+1)\)-tuples end in a summand with that \(\tp\).

Now reverse the gaps, \(\vek d=(0,2,3)\) and \((\delta_2,\delta_3)=(2,1)\). The intermediate vector is
different, \(\vek u_3\,T_3(2)=(5,9,28,39,10)\), but the total is again \(607\), which is the
gap-reversal symmetry of Remark \ref{rem-gap-reversal}. It is worth seeing that this is \textbf{not} because
the factors commute: at \(n=3\),
\[
  T_3(1)T_3(2)\ \ne\ T_3(2)T_3(1),
\]
already in the first row, \((5,9,28,39,10)\) against \((2,5,19,33,10)\). The symmetry belongs to the
bilinear form \(\vek u_3\,(\cdot)\,\vek1_3\) and to the duality behind it, not to the matrices
themselves.
\label{ex-factors}
\end{mexample}

Linusson's Theorem 3.3 and Proposition 2.8 give the degree and the leading coefficient of his own
count as a polynomial in \(p\). They transport through the product of Theorem \ref{thm-poly-p}, and the
result is clean. For each gap put \(\kappa_i=\max(0,n-\delta_i)\).

\begin{theorem}
(\emph{Degree and leading coefficient in \(p\).}) Fix \(n\ge1\), \(r\ge1\) and a sorted \(\vek d\) with
gaps \(\delta_i=d_i-d_{i-1}\). Then
\begin{equation}\label{eq:degree-in-p}
  \deg_p\lvert\mathrm{FF}^{(p)}(n;\vek d)\rvert
   \;=\; r\binom{n+1}{2}\;-\;\sum_{i=2}^{r}\binom{\kappa_i+1}{2},
\end{equation}
which is strictly increasing in each \(\delta_i\) up to \(\delta_i=n\) and constant thereafter. If
moreover every gap satisfies \(\delta_i\ge1\), then the summands decouple to leading order:
\begin{equation}\label{eq:lc-in-p}
  \lvert\mathrm{FF}^{(p)}(n;\vek d)\rvert\;\sim\;F^p(n)\prod_{i=2}^{r}E^p(n,\kappa_i)
  \qquad(p\to\infty).
\end{equation}
\label{thm-asymptotics-p}
\end{theorem}

\begin{proof}
By Lemma \ref{lem-joint} the entries of \(A^{(p)}_n\) are the differences
\[
  E^p(\tau,\rho-1)-E^p(\tau-1,\rho-1),
\]
and Linusson's Theorem 3.3 gives
\[
  \deg_pE^p(m,k)=\binom{m+1}{2}-\binom{k+1}{2},
\]
so the subtrahend has degree
strictly smaller (by exactly \(\tau\)) and cannot cancel the leading term. Hence
\[
  \deg_pA^{(p)}_n[\tau][\rho]=\binom{\tau+1}{2}-\binom{\rho}{2},
\]
with leading coefficient the positive quantity
of his Proposition 2.8.

Expanding the matrix product of Theorem \ref{thm-poly-p} as a sum over paths
\(\tau_1,\dots,\tau_r\), every term is a product of such entries, all with positive leading
coefficients, so again nothing cancels and the degree is the maximum over paths. The path degree
\textbf{separates position by position},
\begin{equation}\label{eq:path-degree}
  \sum_{j}h_{\delta_{j+1}}(\tau_j)+\binom{\tau_r+1}{2},
  \qquad h_\delta(t)=\binom{t+1}{2}-\binom{\max(1,t-\delta+1)}{2},
\end{equation}
so the maximisation separates into one independent choice per position. For \(\delta\ge1\)
the function \(h_\delta\) is strictly increasing on \(\lbrace0,\dots,n\rbrace\), so the maximiser is
unique (every \(\tau_j=n\)), which gives \eqref{eq:degree-in-p} and, since a unique maximising
path contributes a single product of leading coefficients, the product form \eqref{eq:lc-in-p}.
\end{proof}

\begin{remark}
At \(\delta_i=0\) the function \(h_0\) of \eqref{eq:path-degree} is identically zero, so the
maximiser is no longer unique and the leading coefficient becomes a \textbf{sum} over the degree-maximising
paths, which are pinned to \(\tau=n\) wherever a positive gap follows and free non-decreasing
across runs of zero gaps. The degree formula \eqref{eq:degree-in-p} is unaffected. The product form
\eqref{eq:lc-in-p} genuinely fails: at \(n=2\) with the single gap \(\delta_2=0\) the leading
coefficient is \(1\), against \(1/6\) for the product. This is the same phenomenon as
Proposition \ref{prop-delta0}, where vanishing gaps collapse the transfer matrix to a unipotent one,
seen now in the \(p\)-direction rather than the \(r\)-direction.

Two smaller observations follow. The degree and the leading coefficient are already those of the
saturated case at \(\delta_i=n\), one step before the threshold \(\delta_i\ge n+1\) of
Proposition \ref{prop-threshold}; the two regimes first differ in lower-order terms, the single
violating pair recorded there surviving at every \(p\). The proper variant transports as well,
which is Theorem \ref{thm-asymptotics-p-proper} below.
\label{rem-asymptotics-zero-gaps}
\end{remark}

The transport of the proper count is the same argument run on the principal block, and the outcome
is a uniform shift by one. Put
\begin{align}
  S_n^{(p)} &= \bigl\lbrace\vek s\in\mathcal F_n^{(p)}\;:\;\indeg(\vek s)\ge2\bigr\rbrace,\label{eq:def-Snp}\\
  \lvert S_n^{(p)}\rvert &= \sum_{k=1}^{n}E^p(n,k)=F^p(n)-F^{p-1}(n) ,\label{eq:card-Snp}
\end{align}
let \(\mathrm{FF}^{(p)}_{\mathrm{pr}}(n;\vek d)\) be the corresponding sumset, and write
\(\kappa_i'=\max(1,n-\delta_i)\).

\begin{theorem}
(\emph{The proper analogue.}) Let \(A'^{(p)}_n\) be the matrix \(A^{(p)}_n\) with its columns \(\rho=0\)
and \(\rho=1\) set to zero, let \(\vek u'^{(p)}_n\) be the vector of its row sums, and let
\(T'^{(p)}_n(\delta)\) be built from \(A'^{(p)}_n\) by \eqref{eq:def-u}--\eqref{eq:def-T}. Then, for every
\(p\ge n\),
\begin{equation}\label{eq:proper-p-product}
  \lvert\mathrm{FF}^{(p)}_{\mathrm{pr}}(n;\vek d)\rvert
  \;=\;\vek u'^{(p)}_n\cdot T'^{(p)}_n(\delta_2)\cdots T'^{(p)}_n(\delta_r)\cdot\vek 1_n ,
\end{equation}
and this count is a polynomial in \(p\), whose degree is
\begin{equation}\label{eq:proper-degree-in-p}
  \deg_p\lvert\mathrm{FF}^{(p)}_{\mathrm{pr}}(n;\vek d)\rvert
   \;=\; r\binom{n+1}{2}\;-\;1\;-\;\sum_{i=2}^{r}\binom{\kappa_i'+1}{2},
\end{equation}
The degree is strictly increasing in each \(\delta_i\) up to \(\delta_i=n-1\), and it is constant
thereafter. If moreover every gap satisfies \(\delta_i\ge1\), then
\begin{equation}\label{eq:proper-lc-in-p}
  \lvert\mathrm{FF}^{(p)}_{\mathrm{pr}}(n;\vek d)\rvert\;\sim\;E^p(n,1)\prod_{i=2}^{r}E^p(n,\kappa_i')
  \qquad(p\to\infty),
\end{equation}
where \(E^p(n,1)\) may be replaced by \(\lvert S_n^{(p)}\rvert\), of which it is the dominant term.
\label{thm-asymptotics-p-proper}
\end{theorem}

\begin{proof}
For the formula, the proofs of Theorems \ref{thm-proper} and \ref{thm-poly-p} compose. Over \(p\)
generators the capacity argument of Lemma \ref{lem-lexproper} reads \(s_k(0)=1=\binom p0\) and
\(s_k(1)=p=\binom p1\) and goes through verbatim, so the lex representative of a properly
achievable Hilbert function is again proper, and deleting the \(\indeg\le1\) columns implements
the restriction, as in Theorem \ref{thm-proper}.

For the asymptotics, observe first that row \(\tau=-1\) of \(A^{(p)}_n\) is \(e_0\) and row
\(\tau=0\) is \(e_1\), so that both of them die with the deleted columns. Hence \(\vek u'^{(p)}_n\) vanishes at \(\tau=-1,0\) and the columns
\(\tau'=-1,0\) of every \(T'^{(p)}_n(\delta)\) vanish, confining the whole product to the principal
block \(\tau\in\lbrace1,\dots,n\rbrace\), the restriction that already repaired the proper half of
Theorem \ref{thm-delta0-degrees}, now at every gap simultaneously. On the block every surviving
entry is an untouched entry of \(A^{(p)}_n\), with the degrees and positive leading coefficients of
Theorem \ref{thm-asymptotics-p}'s proof; the inner sums now start at \(\max(2,\tau-\delta+1)\)
rather than \(\max(1,\tau-\delta+1)\), so nothing cancels, and the path degree separates as in
\eqref{eq:path-degree} with \(h_\delta\) replaced by
\begin{equation}\label{eq:h-prime}
  h'_\delta(t)=\binom{t+1}{2}-\binom{\max(2,\,t-\delta+1)}{2} ,
\end{equation}
and a constant \(-\binom22=-1\), the same for every path, contributed by the first factor. On \(\lbrace1,\dots,n\rbrace\), \(h'_0\equiv0\) and \(h'_\delta\)
is strictly increasing for \(\delta\ge1\), so the unique maximiser is again the all-\(n\) path;
its factors have degrees \(\binom{n+1}{2}-1\) and \(\binom{n+1}{2}-\binom{\kappa_i'+1}{2}\), with leading
coefficients those of \(E^p(n,1)\) and \(E^p(n,\kappa_i')\) by Linusson's Proposition 2.8, which
gives \eqref{eq:proper-degree-in-p} and \eqref{eq:proper-lc-in-p}.
\end{proof}

\begin{remark}
At \(\delta_i=0\) the leading coefficient again becomes a sum over the degree-maximising paths,
now confined to the principal block. At \(n=2\) with all \(r-1\) gaps zero that sum equals
\(r/2\), because there are \(r\) weakly increasing chains
\(1\le\tau_1\le\dots\le\tau_{r-1}\le2\) and each of them carries a single factor
\(\tfrac12\).

The one-step window that the improper count leaves between saturation and threshold closes here.
Over \(S_n^{(p)}\) the statistics satisfy \(\tp\le n\) and \(\indeg\ge2\) at every \(p\), so
\(\lvert\mathrm{FF}^{(p)}_{\mathrm{pr}}(n;\vek d)\rvert=\lvert S_n^{(p)}\rvert^{\,r}\) holds \textbf{exactly}
when every \(\delta_i\ge n-1\), which is the \(p\)-analogue of
Proposition \ref{prop-proper-threshold} and is also the point at which \(\kappa_i'\), and hence the
degree, saturates. The reason the two thresholds coincide here and not in the improper case is that
the improper violating pair at \(\delta_i=n\) was the full truncation against the void family,
and the void family is not proper.

The verification was carried out in the same style as for Theorem \ref{thm-asymptotics-p}. The degree
and the leading coefficient were checked by exact rational interpolation over \(132\)
configurations with \(n\le5\) and \(r\le4\), with gaps \(0\) to \(n+2\) at \(r=2\) and \(0\) to
\(n+1\) beyond, reaching degrees as high as \(37\); and the transfer formula \eqref{eq:proper-p-product} was checked against brute-force
proper truncated sumsets for \(n\le3\), \(p\le5\) and \(r\le3\). No mismatch was found.
\label{rem-proper-p}
\end{remark}

\begin{mexample}
(\emph{The two degree formulas on the configurations of Example \ref{ex-poly-p-tables}.}) With
\(\kappa_i=\max(0,n-\delta_i)\) and \(\kappa_i'=\max(1,n-\delta_i)\) as in the two theorems, the
degrees predicted by \eqref{eq:degree-in-p} and \eqref{eq:proper-degree-in-p} against the degrees of
the interpolated polynomials:
\begingroup\small
\begin{center}
\begin{tabular}{@{}cccrrrr@{}}
  \toprule
  \(n\) & \(\vek d\) & gaps & \(\deg_p\lvert\mathrm{FF}^{(p)}(n;\vek d)\rvert\) & \eqref{eq:degree-in-p}
        & \(\deg_p\lvert\mathrm{FF}^{(p)}_{\mathrm{pr}}(n;\vek d)\rvert\) & \eqref{eq:proper-degree-in-p}\\
  \midrule
  2 & \((0,0)\)   & \((0)\)   & 3  & 3  & 2  & 2\\
  2 & \((0,1)\)   & \((1)\)   & 5  & 5  & 4  & 4\\
  2 & \((0,0,0)\) & \((0,0)\) & 3  & 3  & 2  & 2\\
  2 & \((0,1,2)\) & \((1,1)\) & 7  & 7  & 6  & 6\\
  \addlinespace
  3 & \((0,0)\)   & \((0)\)   & 6  & 6  & 5  & 5\\
  3 & \((0,1)\)   & \((1)\)   & 9  & 9  & 8  & 8\\
  3 & \((0,4)\)   & \((4)\)   & 12 & 12 & 10 & 10\\
  3 & \((0,1,3)\) & \((1,2)\) & 14 & 14 & 13 & 13\\
  3 & \((0,2,3)\) & \((2,1)\) & 14 & 14 & 13 & 13\\
  \bottomrule
\end{tabular}
\captionof{table}{Predicted against interpolated degrees in \(p\), improper and proper.}
\label{tab-degree-check}
\end{center}
\endgroup
Every entry agrees. Three things in the table are the theorems rather than the data. The degree is
one lower in the proper column in every row whose gaps are all below \(n\), which is the \(-1\) of
\eqref{eq:proper-degree-in-p}, contributed by the first factor alone and so independent of the
gaps. A gap \(\delta_i\ge n\) costs one more, since there \(\kappa_i=0\) but \(\kappa_i'=1\), which is
why the seventh row is two lower. That row has
\(\delta=4>n=3\), where \(\kappa_i=0\) and the degree has saturated: raising the gap further changes
nothing, which is the ``constant thereafter'' clause of Theorem \ref{thm-asymptotics-p}. And the last two
rows agree in both columns although the underlying polynomials differ, the degree depending on the
gaps only through the multiset \(\lbrace\kappa_i\rbrace\) while the polynomial depends on their order.

The leading coefficients behave as the two theorems and Remark \ref{rem-asymptotics-zero-gaps} predict:
the product forms \eqref{eq:lc-in-p} and \eqref{eq:proper-lc-in-p} hold in the six rows whose gaps
are all at least \(1\), and fail in the three rows with a zero gap, where the leading coefficient
becomes a sum over several degree-maximising paths. At \(n=2\) with \(\vek d=(0,0)\) the true improper
leading coefficient is \(1\), against \(1/6\) for the product.
\label{ex-degrees-p}
\end{mexample}
\section{Verification, tables, and open problems}
\label{sec-verification}
\subsection{How the results were checked}
\label{sec:org1c61e3f}

Every claim above labelled ``verified'' was checked against a route independent of the formula being
tested. Specifically:

\begin{itemize}
\item \textbf{Against brute force.} Direct enumeration of the order ideals of \(\mathcal P([n])\), forming the
sumset explicitly, with no recursion and no Kruskal--Katona: agreement for all \(41\)
shift vectors tried, each in both the proper and the improper variant (\(82\) comparisons), with
\(n\le6\), \(r\le4\) and gaps \(0\) to \(6\), including unequal gaps such as
\(\vek d=(0,0,1)\) and \((0,1,3)\). The range is a union rather than a box: \(r=4\) occurs only
for \(n\le3\), and \(n=6\) only for \(r=2\). Lemma \ref{lem-transfer-support}, and with it the
irreducibility and aperiodicity drawn from it, was checked this way as well, from an \(A_n\)
built by direct enumeration with no recursion, for \(n\le6\) and every \(\delta\le n+2\); it
agrees everywhere with the same checks run on the \(A_n\) of Lemma \ref{lem-joint}.
\item \textbf{Against Linusson's own tables.} All \(199\) printed entries of Tables 5 and 6 in the appendix of
his preprint \autocite{Linusson1999preprint}, exactly; those tables are not in the published version
(Remark \ref{rem-transcription}).
\item \textbf{Against a computer-algebra oracle.} Amata and Crupi's \texttt{ExteriorModules} package
\autocite{AmataCrupi2021} for Macaulay2 \autocite{Macaulay2}
enumerates lex submodules directly. Its counts \(43\), \(69\), \(99\) and \(236\) for \(n=3\) with
\(d=(0,0)\), \((0,1)\), \((-1,2)\) and \((0,0,1)\) are reproduced, as is \(368\) for \(n=4\),
\(d=(0,1)\). The case \((-1,2)\) exercises the negative-degree hypothesis.
\item \textbf{Against a proof assistant.} The combinatorial content of this paper has been formalised in Lean 4
with Mathlib, in the same development as the companion article's: the numbered definitions, lemmas,
propositions, theorems and corollaries of Sections \ref{sec-statistics} to \ref{sec-growth}, with
the exceptions set out below; the remarks and examples are not formalised, except where a remark
names a formal statement. Shift vectors are taken non-negative in the formalisation, which loses
nothing by Remark \ref{rem-shift-only-gaps}. Every certified
statement rests on the three standard axioms of Lean's kernel (\texttt{propext}, \texttt{Classical.choice},
\texttt{Quot.sound}) and on no others, and the development contains no incomplete proof.

What follows describes what is \textbf{not} certified. That is the more useful thing to publish, because
it tells the reader what they are being asked to take on trust, and because it does not go out of
date as the development grows. Three things lie outside.

\textbf{The algebra, in two quite different sizes.} None of Definition \ref{def-setup} is formalised: not
the exterior algebra, not monomial or lex submodules, and with them neither Lemma
\ref{lem-dictionary}, nor Lemma \ref{lem-lexproper}, nor the surjectivity half of Theorem
\ref{thm-chaining}. These are not equally far away, and treating them alike would misdescribe the
work involved. Lemma \ref{lem-dictionary} is \textbf{elementary}: its two clauses are a spanning
observation and a dimension count, and what it lacks is the objects rather than an argument. It
needs a graded \(\Ext\) with accessible graded pieces, monomial ideals, \(\mathfrak m^t\), and the
correspondence between monomial ideals and order ideals that makes \(\HF\bigl((\Ext)/I\bigr)\) an
\(f\)-vector. The surjectivity half of Theorem \ref{thm-chaining} needs all of that \textbf{and} Amata
and Crupi's theory of lex submodules: the equivalence of Definitions 2.6 and 2.7 of
\autocite{AmataCrupi2020KK}, the
existence and uniqueness of a lex representative in \autocite[Theorem 4.2]{AmataCrupi2020KK}, and Kruskal--Katona in
the exterior form of Aramova, Herzog and Hibi~\cite{AramovaHerzogHibi1997}. That is a substantially larger
undertaking, and is in any case open combinatorially. The formal development is
purely combinatorial: it counts order ideals and their \(f\)-vectors and never constructs a
submodule. Consequently, wherever a result needs the bridge from the combinatorics to the algebra,
that bridge is an explicit hypothesis rather than a proof. The counting identities of Theorems
\ref{thm-main} and \ref{thm-proper} are certified \textbf{unconditionally} as statements about chained
tuples. Every statement about \(\lvert\mathrm{FF}(n;\vek d)\rvert\) or
\(\lvert\mathrm{FF}_{\mathrm{pr}}(n;\vek d)\rvert\) (those two theorems, Corollary \ref{cor-r1},
Propositions \ref{prop-threshold}, \ref{prop-proper-threshold} and \ref{prop-normalform}, Theorem
\ref{thm-delta0-degrees}, and the reading of Proposition \ref{prop-recurrence} as a statement about
counts) is certified given that hypothesis, which the paper itself discharges algebraically, in
Theorem \ref{thm-chaining} and Lemma \ref{lem-lexproper}. This is a different thing from the
``proved modulo'' of the table in Section \ref{sec-intro}, which refers only to Linusson's
Theorem 3.3.

\textbf{Linusson's results.} Theorem \ref{thm-linusson} and his Theorem 3.3 are cited rather than reproved,
so Lemma \ref{lem-joint} and everything resting on them is outside the formalisation: Theorem
\ref{thm-poly-p} and Theorems \ref{thm-asymptotics-p} and \ref{thm-asymptotics-p-proper}. For Lemma
\ref{lem-joint} this is deliberate and costs nothing, since \(A_n\) is defined by counting in
Definition \ref{def-A} and the recursion is nowhere needed.

\textbf{What ``Perron root'' is not.} Perron--Frobenius theory is not available in the library used. Of
Proposition \ref{prop-perron-threshold} what is certified is therefore the eigenvalue statement, for
both roots: above its threshold each of \(T_n(\delta)\) and \(T'_n(\delta)\) has rank one, the vector
\(\vek 1_n\) is an eigenvector with eigenvalue \(\lvert\mathcal F_n\rvert\) and
\(\lvert S_n\rvert\) respectively, and no other eigenvalue is non-zero. What is not certified is
the word \textbf{Perron},
that this is the eigenvalue of largest modulus with a non-negative eigenvector, and in the
rank-one case the gap is only the comparison against \(0\), but it is a gap. The same limitation
leaves the growth claim of Proposition \ref{prop-recurrence} outside, and there, unlike here, no
degeneracy makes the gap small; and it leaves Proposition \ref{prop-growth-constant} outside
altogether, that proposition being an appeal to the same theory in the sharper form of
\autocite[Theorem V.7]{FlajoletSedgewick2009}. Lemma \ref{lem-transfer-support}, which supplies that
proposition's hypotheses, \textbf{is} certified, in both halves and unconditionally in \(\delta\);
being purely combinatorial, it needed none of the algebra. Its \(\delta=0\) case re-derives the
block-triangularity of \(T_n(0)\) that the development had previously obtained separately, and
both derivations are kept. Proposition \ref{prop-recurrence}'s other half \textbf{is} certified: the linear recurrence
follows from the Cayley--Hamilton theorem, needs no spectral theory at all, and holds for both
matrices unconditionally.

The formalisation turned up three kinds of thing, one for each way a hypothesis can be wrong.
Several of the strengthenings recorded above were found here, each of them a hypothesis that
turned out to be unnecessary. One hypothesis turned out to be \textbf{missing}, namely
Theorem \ref{thm-proper}'s \(n\ge1\), without which the statement is false at \(n=0\), where
\(A'_n\) vanishes identically while
\(\lvert\mathrm{FF}_{\mathrm{pr}}(0;\vek d)\rvert=1\). And one turned out to be \textbf{exactly
sharp}: the proper half of Lemma \ref{lem-transfer-support} is restricted to the principal block
\(\tau'\in\lbrace1,\dots,n\rbrace\), and formalising it is what established that the
restriction cannot be relaxed to \(\tau'\ge0\). It cannot, because the column \(\tau'=0\) of
\(T'_n(\delta)\) vanishes identically for every \(\delta\) while the support condition
\(\tau'\ge\tau-\delta\) holds for every \(\tau\) once \(\delta\) is large; so the
equivalence would fail in the worst way, with a false left-hand side and a true right-hand one.
The boundary is the exact complement of the one the development had already proved for the
vanishing columns, which is the sort of agreement that makes a formalisation worth having.

\item \textbf{Against previously recorded values.} The two sequences of Remark \ref{rem-proper-sequences}.
\item \textbf{Internal consistency.} The \(r=1\) specialisation (Corollary \ref{cor-r1}) and the product-regime
identities (Propositions \ref{prop-threshold} and \ref{prop-proper-threshold}) for \(n=2,\dots,6\),
\(r\le4\), both variants.
\item \textbf{Against a second form of the same formula.} Some entries compare two expressions of the count
rather than the count with an enumeration, because the order ideals of \(\mathcal P([n])\) cannot
be listed beyond small \(n\): Theorem \ref{thm-delta0-degrees} for \(n\le8\) and \(r\le n+4\) (the
binomial expansion of its proof against the transfer product); Proposition \ref{prop-normalform} for
\(n\le8\) and \(\delta\le n+3\), exactly; Proposition \ref{prop-perron-threshold} for \(n\le6\), in
floating point and exactly at the threshold; the gap reversal of Remark \ref{rem-gap-reversal}
for \(n\le6\) and six gap vectors; Lemma \ref{lem-transfer-support} for \(n\le9\) and
\(\delta\le n+3\), exactly, with the irreducibility for \(\delta\ge1\), its failure at
\(\delta=0\) and the loop at every vertex read off the same matrices; and the conclusion of
Proposition \ref{prop-growth-constant} numerically, the ratio \(a_r/\lambda^{\,r}\) settling to six
decimal places by \(r=14\) for \(n\le5\) and \(\delta\ge1\) and diverging at \(\delta=0\), which
is the counterexample of Remark \ref{rem-recurrence-not-asymptotic} reproducing itself. That last
check is the one that would have caught a misreading of the cited theorem's indexing, which the
structural checks could not. These confirm the algebra, not the count.
\item \textbf{Interpolated polynomials.} The polynomials of Examples \ref{tab-delta0-polys} and
\ref{ex-poly-p-tables} were fitted in exact arithmetic and checked at ten further points; they agree
with brute-force counts at \(n\le4\), \(r\le3\) and at \(n=5\), \(r=2\) for the first, and at
\(p\le6\) for the second. Theorems \ref{thm-asymptotics-p} and \ref{thm-asymptotics-p-proper} were
checked by exact interpolation over \(132\) configurations with \(n\le5\) and \(r\le4\), and the
transfer formula in \(p\) against brute-force truncated sumsets for \(n\le3\), \(p\le5\) and
\(r\le3\).
\end{itemize}

No mismatch was found in any of these.
\subsection{Code, formalisation and tables}
\label{sec-availability}
The arXiv version of this paper carries an ancillary directory, \texttt{anc/}, listed on its abstract
page. It holds three things, each described in \texttt{anc/README.md}.

\begin{itemize}
\item \texttt{anc/lean/}: the Lean 4 development described above, with its own \texttt{README.md} giving the
correspondence between its declarations and the numbered statements here and in the companion
paper.
\item \texttt{anc/verification-code/}: the scripts behind every computational check listed above, in Python,
SageMath, Macaulay2 and Julia. Two of them do most of the work: \texttt{linusson\_ff\_counting.py},
which holds the transfer-matrix count, the brute-force enumeration of the sumset and the
comparison with Linusson's tables; and \texttt{oneoff/m2\_oracle\_ff\_counts.m2}, the Macaulay2 oracle.
The oracle needs Amata and Crupi's \texttt{ExteriorModules} package \autocite{AmataCrupi2021}, which is
distributed with Macaulay2 and is loaded by \texttt{needsPackage}; it is not copied here, and
\texttt{anc/README.md} says what to do on an older Macaulay2.
\item \texttt{anc/REFERENCE-CERTIFICATE.md}: for every reference this paper cites, the bibliographic data,
where a primary copy was obtained, and for each locator the passage it is meant to support,
quoted. The paper was written with substantial language-model assistance and says so below;
this is the check behind the claim that its reference list was verified against primary sources.
\item \texttt{anc/tables/}: the counts \(\lvert\mathrm{FF}(n;\vec d\,)\rvert\) and
\(\lvert\mathrm{FF}_{\mathrm{pr}}(n;\vec d\,)\rvert\) for \(1\le n\le 8\), \(1\le r\le 4\) and every
gap vector with gaps at most \(n+1\), which is all of them in that range; the constant-gap
families for \(r\le 8\); and the matrices \(A_n\) for \(n\le 8\). All are plain CSV files.
\end{itemize}

\textbf{\textbf{The recorded outputs of the runs are not in \texttt{anc/}}}, and neither are the raw data files the
companion paper's polytope computations produced; together they come to some 230 megabytes, which
is more than an ancillary directory should carry. They, and the material above kept current, are
in the public repository

\begin{center}
\url{https://gitlab.liu.se/jansn19/exterior-convex}
\end{center}

Everything the verification section above reports can be recomputed from the scripts in \texttt{anc/}
alone. The repository is where to look for what a particular run actually printed, and for the
Lean development's build artefacts and history.
\subsection{Open problems}
\label{sec:orge3a51f9}

Three questions that arose during this work are settled here and are not repeated as open: the
asymptotics in \(p\) are now Theorems \ref{thm-asymptotics-p}
and \ref{thm-asymptotics-p-proper} with Remark \ref{rem-asymptotics-zero-gaps}, the exact degrees at
\(\delta=0\) are Theorem \ref{thm-delta0-degrees}, and the identity \(\lvert S_n\rvert=F^n(n-2)\)
is Proposition \ref{prop-fnn2}. Two problems remain.

\begin{enumerate}
\item \textbf{A bijective proof.} Theorem \ref{thm-chaining} is proved by way of a canonical representative;
a direct bijection between \(\mathrm{FF}(n;\vek d)\) and the chained tuples, not routed through lex
submodules, would explain the constraint combinatorially rather than algebraically. Remark
\ref{rem-gap-reversal} is a constraint on any answer rather than a step toward one: such a
bijection must be compatible with the gap-reversing involution recorded there. The explicit
inverse of Lemma \ref{lem-explicit-inverse} is half of an answer: it is closed-form, its
uniqueness is elementary, and it is equivariant under that involution. What is missing is a
direct proof that its clamps are \(f\)-vectors.
\item \textbf{Beyond constant gaps.} The formula handles an arbitrary sorted \(\vek d\), but the growth results of
Section \ref{sec-growth} assume a constant gap. Non-constant gap vectors give products of distinct
matrices. Proposition \ref{prop-normalform} is the structural answer, and it reduces the problem
to the pair \((B_n,P_n)\); the count also factorises across any saturated gap. What it does not give
is a closed form, and the natural routes to one are closed off as far as we have checked: no two
distinct \(T_n(\delta)\) commute (\(n\le8\)), and they generate the full matrix algebra
(\(n\le5\)), so there is no common triangularisation or invariant filtration, and the only permutation symmetry is the reversal of
\eqref{eq:no-permutation-invariance}.
\end{enumerate}
\subsection{Integer sequences}
\label{sec:org31e9c17}

Seven sequences arising here appear not to be in the OEIS (searched October 2026): \(|S_n|\) itself, the \(r=2\)
proper and improper counts at \(\delta=0\) and \(\delta=1\), one \(r=3\) example, and the proper
count at \(n=4\), \(\delta=0\) read along \(r\), that is \(16,69,184,385,696\), the row of
Table \ref{tab-app-pro}. Of
these \(|S_n|=1,2,5,16,70,457,4908,95248,3617645,286007155\), the number of \(f\)-vectors of
simplicial complexes on exactly \(n\) vertices, seems the most natural candidate for submission,
being the first-difference sequence of A007695, which has no separate entry. Nothing has been
submitted.

\appendix
\section{Tables}
\label{sec-app-tables}
\label{sec:app-tables}
Linusson closes the preprint version of his paper \autocite{Linusson1999preprint} with two tables,
and they are part of why it is usable rather than
merely correct: his Table 5 gives the counts \(F^p(n)\) themselves, and his Table 6 gives
\(E^p(p,k)\),
the refinement his recursion actually runs on, with \(F^p(p)\) recovered as a row sum in the last
column. This appendix gives the analogues, in the same pairing and the same order. Every entry
was produced by the formulas of this paper and checked against direct enumeration of the
sumset, with no recursion and no transfer matrix, wherever that was feasible: every entry of both
count tables with \(n\le5\), except the improper count for \(n=5\), \(r=5\), \(\delta\ge3\), and for
\(n=6\) the entries with \(r\le3\), with \(r=4\), \(\delta\le1\), and with \(r=5\), \(\delta=0\). That is
\(103\) of the \(120\) improper and \(107\) of the \(120\) proper entries with \(r\ge2\), with no
mismatch.
\subsection{The joint distribution}
\label{sec:orgf695fd1}

Table \ref{tab-app-A} is the analogue of Linusson's Table 6: the matrix \(A_n\) of
Definition \ref{def-A}, which is what the transfer-matrix formula is built from, for
\(n=1,\dots,6\). Rows are indexed by \(\tp\) and columns by \(\indeg\), both as in
\eqref{eq:def-A-range}; a dot is a zero. The last column is the row sum \(u_n[\tau]\) of
\eqref{eq:def-u}, so reading it downwards gives the vector \(\vek u_n\) of Theorem \ref{thm-main},
and the bold entry is the grand total \(\lvert\mathcal F_n\rvert\). The support condition
\(\rho\le\tau+1\) of Definition \ref{def-A} is visible as the staircase, and the two extreme rows
are the single entries \(A_n[-1][0]=1\) and \(A_n[n][n+1]=1\) named there. The \(n=3\) block is the
matrix displayed in Example \ref{ex-A}.

The correspondence with Linusson's Table 6 is in fact an identity and not merely an analogy.
Summing \eqref{eq:Ap} at \(p=n\) over \(\tau\) telescopes,
\begin{equation}\label{eq:colsum-is-E}
  \sum_{\tau=-1}^{n}A_n[\tau][\rho]\;=\;E^n(n,\rho-1) ,
\end{equation}
the subtracted terms cancelling in pairs, and the boundary term vanishing because nothing sits
below the void family, so that \(E^p(m,k)\) is zero once \(m\le-2\). The bottom row of each block
below is therefore \textbf{exactly} the corresponding row of his Table 6,
re-indexed by \(\rho=k+1\), and the grand total is his \(F^n(n)\). Our refinement splits his by
\(\tp\); his is recovered by forgetting that split.

\begingroup\small
\begin{center}
\begin{tabular}{@{}r|rrr|r@{}}
\multicolumn{5}{@{}l}{\(n=1\)}\\[2pt]
\(\tau\backslash\rho\) & 0 & 1 & 2 & \(u_n[\tau]\)\\
\hline
\(-1\) & 1 & \(\cdot\) & \(\cdot\) & \(1\)\\
\(0\) & \(\cdot\) & 1 & \(\cdot\) & \(1\)\\
\(1\) & \(\cdot\) & \(\cdot\) & 1 & \(1\)\\
\hline
\(\Sigma\) & 1 & 1 & 1 & \(\mathbf{3}\)\\
\end{tabular}
\end{center}
\addvspace{4pt}
\begin{center}
\begin{tabular}{@{}r|rrrr|r@{}}
\multicolumn{6}{@{}l}{\(n=2\)}\\[2pt]
\(\tau\backslash\rho\) & 0 & 1 & 2 & 3 & \(u_n[\tau]\)\\
\hline
\(-1\) & 1 & \(\cdot\) & \(\cdot\) & \(\cdot\) & \(1\)\\
\(0\) & \(\cdot\) & 1 & \(\cdot\) & \(\cdot\) & \(1\)\\
\(1\) & \(\cdot\) & 1 & 1 & \(\cdot\) & \(2\)\\
\(2\) & \(\cdot\) & \(\cdot\) & \(\cdot\) & 1 & \(1\)\\
\hline
\(\Sigma\) & 1 & 2 & 1 & 1 & \(\mathbf{5}\)\\
\end{tabular}
\end{center}
\addvspace{4pt}
\begin{center}
\begin{tabular}{@{}r|rrrrr|r@{}}
\multicolumn{7}{@{}l}{\(n=3\)}\\[2pt]
\(\tau\backslash\rho\) & 0 & 1 & 2 & 3 & 4 & \(u_n[\tau]\)\\
\hline
\(-1\) & 1 & \(\cdot\) & \(\cdot\) & \(\cdot\) & \(\cdot\) & \(1\)\\
\(0\) & \(\cdot\) & 1 & \(\cdot\) & \(\cdot\) & \(\cdot\) & \(1\)\\
\(1\) & \(\cdot\) & 2 & 1 & \(\cdot\) & \(\cdot\) & \(3\)\\
\(2\) & \(\cdot\) & 1 & 2 & 1 & \(\cdot\) & \(4\)\\
\(3\) & \(\cdot\) & \(\cdot\) & \(\cdot\) & \(\cdot\) & 1 & \(1\)\\
\hline
\(\Sigma\) & 1 & 4 & 3 & 1 & 1 & \(\mathbf{10}\)\\
\end{tabular}
\end{center}
\addvspace{4pt}
\begin{center}
\begin{tabular}{@{}r|rrrrrr|r@{}}
\multicolumn{8}{@{}l}{\(n=4\)}\\[2pt]
\(\tau\backslash\rho\) & 0 & 1 & 2 & 3 & 4 & 5 & \(u_n[\tau]\)\\
\hline
\(-1\) & 1 & \(\cdot\) & \(\cdot\) & \(\cdot\) & \(\cdot\) & \(\cdot\) & \(1\)\\
\(0\) & \(\cdot\) & 1 & \(\cdot\) & \(\cdot\) & \(\cdot\) & \(\cdot\) & \(1\)\\
\(1\) & \(\cdot\) & 3 & 1 & \(\cdot\) & \(\cdot\) & \(\cdot\) & \(4\)\\
\(2\) & \(\cdot\) & 4 & 5 & 1 & \(\cdot\) & \(\cdot\) & \(10\)\\
\(3\) & \(\cdot\) & 1 & 4 & 3 & 1 & \(\cdot\) & \(9\)\\
\(4\) & \(\cdot\) & \(\cdot\) & \(\cdot\) & \(\cdot\) & \(\cdot\) & 1 & \(1\)\\
\hline
\(\Sigma\) & 1 & 9 & 10 & 4 & 1 & 1 & \(\mathbf{26}\)\\
\end{tabular}
\end{center}
\addvspace{4pt}
\begin{center}
\begin{tabular}{@{}r|rrrrrrr|r@{}}
\multicolumn{9}{@{}l}{\(n=5\)}\\[2pt]
\(\tau\backslash\rho\) & 0 & 1 & 2 & 3 & 4 & 5 & 6 & \(u_n[\tau]\)\\
\hline
\(-1\) & 1 & \(\cdot\) & \(\cdot\) & \(\cdot\) & \(\cdot\) & \(\cdot\) & \(\cdot\) & \(1\)\\
\(0\) & \(\cdot\) & 1 & \(\cdot\) & \(\cdot\) & \(\cdot\) & \(\cdot\) & \(\cdot\) & \(1\)\\
\(1\) & \(\cdot\) & 4 & 1 & \(\cdot\) & \(\cdot\) & \(\cdot\) & \(\cdot\) & \(5\)\\
\(2\) & \(\cdot\) & 10 & 9 & 1 & \(\cdot\) & \(\cdot\) & \(\cdot\) & \(20\)\\
\(3\) & \(\cdot\) & 9 & 24 & 9 & 1 & \(\cdot\) & \(\cdot\) & \(43\)\\
\(4\) & \(\cdot\) & 1 & 9 & 10 & 4 & 1 & \(\cdot\) & \(25\)\\
\(5\) & \(\cdot\) & \(\cdot\) & \(\cdot\) & \(\cdot\) & \(\cdot\) & \(\cdot\) & 1 & \(1\)\\
\hline
\(\Sigma\) & 1 & 25 & 43 & 20 & 5 & 1 & 1 & \(\mathbf{96}\)\\
\end{tabular}
\end{center}
\addvspace{4pt}
\begin{center}
\begin{tabular}{@{}r|rrrrrrrr|r@{}}
\multicolumn{10}{@{}l}{\(n=6\)}\\[2pt]
\(\tau\backslash\rho\) & 0 & 1 & 2 & 3 & 4 & 5 & 6 & 7 & \(u_n[\tau]\)\\
\hline
\(-1\) & 1 & \(\cdot\) & \(\cdot\) & \(\cdot\) & \(\cdot\) & \(\cdot\) & \(\cdot\) & \(\cdot\) & \(1\)\\
\(0\) & \(\cdot\) & 1 & \(\cdot\) & \(\cdot\) & \(\cdot\) & \(\cdot\) & \(\cdot\) & \(\cdot\) & \(1\)\\
\(1\) & \(\cdot\) & 5 & 1 & \(\cdot\) & \(\cdot\) & \(\cdot\) & \(\cdot\) & \(\cdot\) & \(6\)\\
\(2\) & \(\cdot\) & 20 & 14 & 1 & \(\cdot\) & \(\cdot\) & \(\cdot\) & \(\cdot\) & \(35\)\\
\(3\) & \(\cdot\) & 43 & 84 & 19 & 1 & \(\cdot\) & \(\cdot\) & \(\cdot\) & \(147\)\\
\(4\) & \(\cdot\) & 25 & 143 & 84 & 14 & 1 & \(\cdot\) & \(\cdot\) & \(267\)\\
\(5\) & \(\cdot\) & 1 & 25 & 43 & 20 & 5 & 1 & \(\cdot\) & \(95\)\\
\(6\) & \(\cdot\) & \(\cdot\) & \(\cdot\) & \(\cdot\) & \(\cdot\) & \(\cdot\) & \(\cdot\) & 1 & \(1\)\\
\hline
\(\Sigma\) & 1 & 95 & 267 & 147 & 35 & 6 & 1 & 1 & \(\mathbf{553}\)\\
\end{tabular}
\end{center}
\endgroup
\captionof{table}{The joint distribution \(A_n[\tau][\rho]\) of \((\tp,\indeg)\) on \(\mathcal F_n\),
for \(n=1,\dots,6\).}
\label{tab-app-A}
\subsection{The counts}
\label{sec:orgabd7e94}

Tables \ref{tab-app-imp} and \ref{tab-app-pro} are the analogue of
Linusson's Table 5. His is indexed by his two parameters \(p\) and \(n\); ours needs three, since
a count here depends on \(n\), on the rank \(r\), and on the shift vector. The tables therefore
fix the shift vector to the constant-gap family
\(\vek d=(0,\delta,2\delta,\dots,(r-1)\delta)\), which is the one Section \ref{sec-growth} works
in and the one in which the thresholds are visible.

\textbf{The \(r=1\) column is a check, and it is where the two papers meet.} At rank one there is no
gap to vary, so that column is constant down each block; and there is no transfer matrix either,
the product of Theorem \ref{thm-main} degenerating to \(\vek u_n\cdot\vek 1_n\). It therefore reads
\(\lvert\mathcal F_n\rvert\) in Table \ref{tab-app-imp} and \(\lvert S_n\rvert\) in
Table \ref{tab-app-pro}. The first of those, \(5,10,26,96,553\), is exactly the
diagonal \(F^n(n)\) of Linusson's Table 5 (equivalently, the last column of his Table 6), which
is Corollary \ref{cor-r1} read off a table instead of proved. The second, \(2,5,16,70,457\), is the
sequence \(\lvert S_n\rvert\) of Section \ref{sec-verification}.

Three things in them are stated as results elsewhere and can be read off directly. The improper
count reaches \(\lvert\mathcal F_n\rvert^{\,r}\) at \(\delta=n+1\) and the proper count reaches
\(\lvert S_n\rvert^{\,r}\) at \(\delta=n-1\), which is Propositions \ref{prop-threshold} and
\ref{prop-proper-threshold}; the two-step gap between the thresholds is accounted for by the pairs
that fail at \(\delta=n-1\) and \(\delta=n\): in every one of them the later summand has
\(\indeg\le1\), so it is the void family or a complex missing a vertex, and is not proper. The \(n=2\), \(\delta=0\) row is the one worked through in
Remark \ref{rem-delta0-degrees}, its improper half \(14,30,55,91\) being the cubic with third
difference \(2\) and its proper half \(r+1\). In the proper table every row with \(\delta\ge n\)
repeats the one above it, there being nothing left for a wider gap to decouple; the improper table
stops at \(\delta=n+1\), its first saturated row.

\begingroup\footnotesize
\begin{center}
\begin{tabular}{@{}ccrrrrr@{}}
\toprule
\(n\) & \(\delta\) & \(r=1\) & \(r=2\) & \(r=3\) & \(r=4\) & \(r=5\)\\
\midrule
\(2\) & \(0\) & 5 & 14 & 30 & 55 & 91\\
 & \(1\) & 5 & 20 & 76 & 285 & 1\,065\\
 & \(2\) & 5 & 24 & 115 & 551 & 2\,640\\
 & \(3\) & 5 & 25 & 125 & 625 & 3\,125\\
\addlinespace
\(3\) & \(0\) & 10 & 43 & 125 & 290 & 581\\
 & \(1\) & 10 & 69 & 426 & 2\,532 & 14\,847\\
 & \(2\) & 10 & 91 & 821 & 7\,401 & 66\,712\\
 & \(3\) & 10 & 99 & 980 & 9\,701 & 96\,030\\
 & \(4\) & 10 & 100 & 1\,000 & 10\,000 & 100\,000\\
\addlinespace
\(4\) & \(0\) & 26 & 194 & 816 & 2\,491 & 6\,202\\
 & \(1\) & 26 & 368 & 4\,200 & 44\,021 & 445\,373\\
 & \(2\) & 26 & 556 & 11\,521 & 237\,440 & 4\,888\,880\\
 & \(3\) & 26 & 657 & 16\,589 & 418\,857 & 10\,575\,748\\
 & \(4\) & 26 & 675 & 17\,524 & 454\,949 & 11\,811\,150\\
 & \(5\) & 26 & 676 & 17\,576 & 456\,976 & 11\,881\,376\\
\addlinespace
\(5\) & \(0\) & 96 & 1\,433 & 9\,315 & 38\,800 & 122\,801\\
 & \(1\) & 96 & 3\,405 & 81\,845 & 1\,667\,849 & 31\,570\,587\\
 & \(2\) & 96 & 6\,264 & 373\,078 & 21\,733\,283 & 1\,259\,247\,991\\
 & \(3\) & 96 & 8\,454 & 739\,600 & 64\,670\,555 & 5\,654\,555\,977\\
 & \(4\) & 96 & 9\,165 & 874\,945 & 83\,527\,401 & 7\,974\,017\,466\\
 & \(5\) & 96 & 9\,215 & 884\,544 & 84\,907\,009 & 8\,150\,188\,320\\
 & \(6\) & 96 & 9\,216 & 884\,736 & 84\,934\,656 & 8\,153\,726\,976\\
\addlinespace
\(6\) & \(0\) & 553 & 19\,145 & 204\,141 & 1\,195\,136 & 4\,891\,732\\
 & \(1\) & 553 & 60\,586 & 3\,443\,671 & 147\,428\,310 & 5\,540\,778\,582\\
 & \(2\) & 553 & 145\,746 & 30\,713\,150 & 6\,051\,460\,353 & 1\,168\,581\,488\,805\\
 & \(3\) & 553 & 245\,035 & 105\,135\,755 & 44\,903\,400\,311 & 19\,165\,840\,095\,040\\
 & \(4\) & 553 & 296\,059 & 158\,343\,036 & 84\,684\,938\,944 & 45\,291\,110\,452\,581\\
 & \(5\) & 553 & 305\,618 & 168\,901\,132 & 93\,343\,953\,511 & 51\,586\,946\,480\,989\\
 & \(6\) & 553 & 305\,808 & 169\,111\,271 & 93\,518\,227\,055 & 51\,715\,410\,450\,144\\
 & \(7\) & 553 & 305\,809 & 169\,112\,377 & 93\,519\,144\,481 & 51\,716\,086\,897\,993\\
\bottomrule
\end{tabular}
\end{center}
\endgroup
\captionof{table}{\(\lvert\mathrm{FF}(n;\vek d)\rvert\) for the constant-gap shift vector
\(\vek d=(0,\delta,\dots,(r-1)\delta)\).}
\label{tab-app-imp}

\begingroup\footnotesize
\begin{center}
\begin{tabular}{@{}ccrrrrr@{}}
\toprule
\(n\) & \(\delta\) & \(r=1\) & \(r=2\) & \(r=3\) & \(r=4\) & \(r=5\)\\
\midrule
\(2\) & \(0\) & 2 & 3 & 4 & 5 & 6\\
 & \(1\) & 2 & 4 & 8 & 16 & 32\\
 & \(2\) & 2 & 4 & 8 & 16 & 32\\
 & \(3\) & 2 & 4 & 8 & 16 & 32\\
\addlinespace
\(3\) & \(0\) & 5 & 12 & 22 & 35 & 51\\
 & \(1\) & 5 & 22 & 95 & 409 & 1\,760\\
 & \(2\) & 5 & 25 & 125 & 625 & 3\,125\\
 & \(3\) & 5 & 25 & 125 & 625 & 3\,125\\
 & \(4\) & 5 & 25 & 125 & 625 & 3\,125\\
\addlinespace
\(4\) & \(0\) & 16 & 69 & 184 & 385 & 696\\
 & \(1\) & 16 & 162 & 1\,478 & 13\,100 & 115\,128\\
 & \(2\) & 16 & 246 & 3\,776 & 57\,956 & 889\,536\\
 & \(3\) & 16 & 256 & 4\,096 & 65\,536 & 1\,048\,576\\
 & \(4\) & 16 & 256 & 4\,096 & 65\,536 & 1\,048\,576\\
 & \(5\) & 16 & 256 & 4\,096 & 65\,536 & 1\,048\,576\\
\addlinespace
\(5\) & \(0\) & 70 & 627 & 2\,592 & 7\,385 & 16\,926\\
 & \(1\) & 70 & 1\,858 & 37\,181 & 675\,786 & 11\,847\,871\\
 & \(2\) & 70 & 3\,805 & 199\,805 & 10\,436\,280 & 544\,655\,005\\
 & \(3\) & 70 & 4\,857 & 336\,980 & 23\,379\,749 & 1\,622\,092\,290\\
 & \(4\) & 70 & 4\,900 & 343\,000 & 24\,010\,000 & 1\,680\,700\,000\\
 & \(5\) & 70 & 4\,900 & 343\,000 & 24\,010\,000 & 1\,680\,700\,000\\
 & \(6\) & 70 & 4\,900 & 343\,000 & 24\,010\,000 & 1\,680\,700\,000\\
\addlinespace
\(6\) & \(0\) & 457 & 9\,904 & 67\,448 & 269\,765 & 798\,101\\
 & \(1\) & 457 & 38\,232 & 1\,816\,426 & 70\,055\,219 & 2\,478\,448\,508\\
 & \(2\) & 457 & 104\,870 & 20\,596\,002 & 3\,895\,494\,074 & 729\,778\,783\,985\\
 & \(3\) & 457 & 183\,337 & 72\,771\,631 & 28\,858\,243\,615 & 11\,443\,058\,404\,405\\
 & \(4\) & 457 & 208\,582 & 95\,199\,955 & 43\,450\,688\,041 & 19\,831\,546\,046\,752\\
 & \(5\) & 457 & 208\,849 & 95\,443\,993 & 43\,617\,904\,801 & 19\,933\,382\,494\,057\\
 & \(6\) & 457 & 208\,849 & 95\,443\,993 & 43\,617\,904\,801 & 19\,933\,382\,494\,057\\
 & \(7\) & 457 & 208\,849 & 95\,443\,993 & 43\,617\,904\,801 & 19\,933\,382\,494\,057\\
\bottomrule
\end{tabular}
\end{center}
\endgroup
\captionof{table}{\(\lvert\mathrm{FF}_{\mathrm{pr}}(n;\vek d)\rvert\), the proper count, for the
same shift vectors.}
\label{tab-app-pro}
\section*{Acknowledgements}
\label{sec:org616cb4e}
This paper uses Svante Linusson's 1999 paper structurally, and not only as a citation: the
refinement he introduced for his own recursion turns out to be exactly the statistic the rank-\(r\)
count requires, and the present paper would have no content without it. Aliaksandra Kupreyeva's Bachelor's thesis \autocite{Kupreyeva2019}
(Linköping University, 2019; report LiTH-MAT-EX--2019/05--SE), written under my supervision, is where
counting questions about \(f\)-vectors entered this project.

Just as Geheimrat Goethe had the diligent Friedrich Wilhelm Riemer to help him write \emph{Die
Wahlverwandtschaften}, just as Grothendieck had Jean Dieudonné to help him complete EGA, I have
\textbf{Claude}! What exactly that amounted to is set out in the disclosure section below.

Computations were carried out in SageMath \autocite{SageMath} (version 10.9), in Python, and in
Macaulay2 \autocite{Macaulay2} (version 1.26.06). The Macaulay2 work used Amata and Crupi's own
\texttt{ExteriorModules} package \autocite{AmataCrupi2021}. All three are free and open-source software,
and this paper would have been appreciably harder to write without them.
\section*{Disclosure of AI assistance}
\label{sec:org4c6d21b}
This manuscript was produced with substantial assistance from a large language model, Claude
(Anthropic), and the extent of it is set out here rather than left to be inferred. It is more
extensive in this paper than in its companion.

The model: identified that Linusson's refined count \(E^p(n,k)\) is the joint-distribution object the
rank-\(r\) problem needs, and proved Lemma \ref{lem-joint}; found and proved the chaining reformulation
of Theorem \ref{thm-chaining}, including the dictionary of Lemma \ref{lem-dictionary} and the hypotheses
recorded in Remark \ref{rem-sorted}; assembled and proved Theorem \ref{thm-main} and its consequences
(Corollary \ref{cor-r1}, Propositions \ref{prop-threshold} and \ref{prop-delta0}); found and proved
Lemma \ref{lem-lexproper}, without which the proper variant of Theorem \ref{thm-proper} would not follow;
derived the growth results of Section \ref{sec-growth}, including Theorem \ref{thm-poly-p}; transcribed
Linusson's Theorem 3.2 from the printed page and validated it against both of his preprint's
appendix tables;
found the Alexander-duality argument of Section \ref{sec-statistics} (Lemma \ref{lem-alexander},
Corollary \ref{cor-persymmetry}, Proposition \ref{prop-fnn2}); proved Theorem \ref{thm-delta0-degrees} and
found the digraph reading of Remark \ref{rem-delta0-digraph}; found the explicit inverse of Lemma
\ref{lem-explicit-inverse}; proved Propositions \ref{prop-normalform} and \ref{prop-perron-threshold} and
Theorems \ref{thm-asymptotics-p} and \ref{thm-asymptotics-p-proper}; computed the tables of the
appendix; wrote the Lean formalisation described in Section \ref{sec-verification}, in its
entirety; wrote all the supporting code; and drafted the exposition. The author set the problem,
directed the scope, and required the separation of proved from verified results that Section
\ref{sec-intro}'s table records. He also required that the combinatorial core be formalised and
machine-checked, and the formalisation was then carried out by the model.

Two corrections the model made to its own earlier work are worth recording, as evidence of what the
verification was for: an initial value of \(|S_9|\) was wrong by one, caught by checking against
OEIS A007695 and against Linusson's printed Table 5; and a threshold in Propositions \ref{prop-threshold}
and \ref{prop-proper-threshold} had been stated for the wrong one of the two variants.

Every structural claim in this paper was checked computationally, along a route independent of the
formula being tested, over the ranges stated in Section \ref{sec-verification}, which also lists the
few checks that compare one form of a formula with another. Values beyond those ranges, such as the
few table entries beyond brute-force range and the polynomials at \(n=7\), come from the formulas
themselves, which are proved and, in their combinatorial form, formalised. The reference list has
been checked against primary sources. The author has verified the results to the best of his ability, made the final edits, and
assumes full responsibility for any errors, mistakes or gaps that remain.
\section*{Funding}
\label{sec:orgbf03d5f}
This research received no external funding. It was carried out as part of the author's regular duties
at Linköping University, within the fraction of that employment allocated to research.
\section*{Conflicts of interest}
\label{sec:orgd28a2cf}
The author declares no financial or non-financial conflicts of interest. The nature and extent of
large-language-model assistance in producing this manuscript is disclosed in full in the section
above.

\end{document}